\documentclass[11pt,english]{amsart}

\def\margin_comment#1{\marginpar{\sffamily{\tiny #1\par}\normalfont}}

\usepackage{amssymb,euscript}
\usepackage{amsfonts}
\usepackage{geometry}
\usepackage{fancyhdr}
\usepackage{graphicx}
\usepackage{setspace}
\usepackage{color}
\usepackage{amsmath,amssymb,amsthm}
\usepackage{epsfig}
\usepackage{babel}
\usepackage{tikz}
 \usepackage{amsmath}
\usepackage{tikz-cd}
\usetikzlibrary{matrix, calc, arrows}
\usepackage{tikz}
\usetikzlibrary{arrows,chains,matrix,positioning,scopes}
\usepackage{lipsum}
\usepackage{mathtools}
\usepackage{mathrsfs}
\usepackage{graphicx}
\usepackage{graphics}
\usepackage{setspace}
	\usepackage{color}

\usetikzlibrary{trees}
\usetikzlibrary{arrows,shapes,automata,backgrounds,petri,positioning,scopes,matrix, calc, plothandlers,plotmarks,patterns}
\usepackage{enumerate}
\usepackage{float}
\usepackage[utf8]{inputenc}
\usepackage[T1]{fontenc}
\usepackage[document]{ragged2e}
 \usepackage{verbatim}
\usepackage{tikz-3dplot}
\usepackage{pgfplots}
\usepackage{tkz-graph}
\GraphInit[vstyle = Normal]
\tikzset{
  LabelStyle/.style = {minimum width = 2em, 
                        text = red, font = \bfseries },
  VertexStyle/.append style = { inner sep=2pt,
                                font = \Large\bfseries, fill},
  EdgeStyle/.append style = {->, bend left} }
\newcommand{\boldm}[1] {\mathversion{bold}#1\mathversion{normal}}

\makeatletter
\tikzset{join/.code=\tikzset{after node path={%
\ifx\tikzchainprevious\pgfutil@empty\else(\tikzchainprevious)%
edge[every join]#1(\tikzchaincurrent)\fi}}}
\makeatother
\tikzset{>=stealth',every on chain/.append style={join},
         every join/.style={->}}
\tikzstyle{labeled}=[execute at begin node=$\scriptstyle,
   execute at end node=$]
\newtheorem{thm}{Theorem}[section]
\numberwithin{equation}{section} 
\numberwithin{figure}{section} 
\theoremstyle{plain}
\newtheorem*{thm*}{Theorem}
\theoremstyle{definition}
\theoremstyle{plain}
\newtheorem{thm_A}{Theorem}
\newtheorem*{defn*}{Definition}

\theoremstyle{plain}

\theoremstyle{plain} 

\theoremstyle{plain}

\newtheorem{prop}[thm]{Proposition} 
\theoremstyle{remark}
\newtheorem{ex}[thm]{Example}
\theoremstyle{remark}
\newtheorem{rem}[thm]{Remark}
\theoremstyle{plain}

\theoremstyle{plain}

\theoremstyle{plain}
\newtheorem{lem}[thm]{Lemma} 
\theoremstyle{definition}
\newtheorem{defn}[thm]{Definition}
\newtheorem*{acknowledgment}{Acknowledgment}
\newtheorem*{acknowledgment*}{Addentum}

\theoremstyle{plain}
\newtheorem*{ex*}{Example}
\theoremstyle{plain}

\begin{document}
\title[The Yang-Baxter equation and Thompson's group $F$ ]{The Yang-Baxter equation  and Thompson's group $F$ }
\author{Fabienne Chouraqui}
\begin{abstract}
	We define non-degenerate involutive partial solutions as an extension of  non-degenerate involutive  set-theoretical solutions of 	the quantum Yang-Baxter equation (QYBE).     The induced  operator is  not a classical solution of the 	QYBE,  but either a braiding operator as in  conformal field theory. We define the  structure  inverse monoid of  a  non-degenerate involutive  partial solution and prove that if  the partial solution  is square-free, then it embeds into the restricted product of a commutative inverse monoid and an inverse symmetric monoid. Furthermore, we show that there is a  connection  between partial solutions  and the  Thompson's group $F$. This raises the question of  whether  there are further connections between partial solutions and Thompson's groups in general.
\end{abstract}
\maketitle

\section*{Introduction}
The quantum Yang-Baxter equation is an equation in mathematical physics and it lies in the  foundation of  the theory of quantum groups. One of the fundamental problems is to find all the solutions of this equation. In \cite{drinf}, Drinfeld suggested the study of a particular class of solutions,  derived from the so-called set-theoretic solutions.  A  set-theoretic solution of the Yang-Baxter equation is a pair $(X,r)$, where $X$ is a set and 
\[r: X \times X \rightarrow X \times X\,,\;\;\; r(x,y)=(\sigma_x(y),\gamma_y(x))\]
is a bijective map satisfying $r^{12}r^{23}r^{12}=r^{23}r^{12}r^{23}$, where $r^{12}=r \times Id_X$ and  $r^{23}=Id_X\times r$. A set-theoretic solution  $(X,r)$ is said to be non-degenerate if, for every $x \in X$, the maps $\sigma_x,\gamma_x$ are bijections of $X$ and it  is said to be involutive if $r^2=Id_{X\times X}$.   Non-degenerate and involutive set-theoretic solutions give rise to solutions of the quantum Yang Baxter.  Indeed, by defining  $V$ to be  the  real vector space spanned by $X$,  and $R:V \otimes V \rightarrow V \otimes V$  to be a linear operator induced by  $\tau  \circ r $, where $\tau$ is the flip map $\tau(x,y)=(y,x)$,  $R$ is  a linear operator satisfying the equality  $R^{12}R^{13}R^{23}=R^{23}R^{13}R^{12}$ in $V \otimes V \otimes V$, that is $R$  is a solution of the quantum Yang-Baxter equation. 

Non-degenerate and involutive set-theoretic solutions of the quantum Yang-Baxter equation are  intensively investigated and they give rise to several algebraic structures associated to them. One of these  is the structure group of a solution, $G(X,r)$, which is defined by $G(X,r)=\operatorname{Gp}\langle X\mid x_ix_j=x_kx_l\,\,; r(x_i,x_j)=(x_k,x_l)\rangle$ in \cite{etingof}.  The authors prove that, $G(X,r)$,  the structure group  of a non-degenerate, involutive set-theoretic solution  $(X,S)$ embeds into the semidirect product $\mathbb{Z}^X\rtimes \operatorname{Sym}_X$, where $\operatorname{Sym}_X$ denotes the symmetric group of $X$ and  $\mathbb{Z}^X$  is the free abelian group generated by $X$.   Moreover they prove that if $X$ is finite, then $G(X,r)$ is a solvable group \cite{etingof}. \\

   \setlength\parindent{10pt}	  Another algebraic structure associated to a set-theoretic solution is a  monoid of left $I$-type,  or of  right $I$-type, defined in \cite{gateva_van}. T. Gateva-Ivanova and M. Van den Bergh prove that there is a correspondence between monoids of left $I$-type and non-degenerate,  involutive  set-theoretic solutions. Indeed, they show that a monoid $M$ is of left $I$-type if and only if there exists a non-degenerate,  involutive  set-theoretic solution $(X,r)$, with $X$ finite, such that $M\simeq \operatorname{Mon}\langle X\mid x_ix_j=x_kx_l\,\,; r(x_i,x_j)=(x_k,x_l)\rangle$.  Furthermore, they prove that in this case,  the structure group of $(X,r)$,  $G(X,r)$,  is the group of fractions of $M$ and it is a Bieberbach group \cite[Theorem 1.6]{gateva_van}. In \cite{jespers_i-type}, E. Jespers and J. Okninski prove that a monoid $M$ is of left $I$-type if and only if  $M$ is of  right $I$-type. In \cite{cedo}, the authors initiate the study of  IYB-groups, a special class of finite solvable groups. For each  non-degenerate,  involutive set-theoretic solution $(X,r)$, with $X$ finite, there is a IYB-group, the group generated by the set $\{\sigma_x\mid x\in X\}$. Furthermore, they raise the question whether every finite solvable group is IYB.\\

    T. Gateva-Ivanova conjectured that every square-free set-theoretic solution is decomposable \cite{gateva_conj}. In \cite{rump},  W.Rump proves  the conjecture is true for   square-free finite solutions and that  an extension to infinite solutions is false. He defines the structure of   cycle sets, in correspondence with non-degenerate,  involutive  set-theoretic solutions, and uses it to prove the conjecture.  Cycle sets have been studied also in the context of another conjecture of Gateva-Ivanova   in \cite{vendra-cycle} and \cite{catino-cycle, catino-cycle2}.\\

In \cite{chou_art}, we show there is a one-to-one correspondence between non-degenerate,  involutive  set-theoretic solutions and a particular class of groups, the so-called Garside groups, with a particular presentation. Garside groups  have been defined by P. Dehornoy and L.Paris as a generalization in some sense of the braid groups, and the finite-type Artin groups \cite{DePa}.  In \cite{chou_godel1,chou_godel2}, with E. Godelle, we deepen our understanding of the connection between these structures (see also   \cite{chou_aut, chou_left_garside}, \cite{gateva}, \cite{deh_coxeterlike}).\\

In \cite{rump_braces}, W. Rump introduced braces as a generalization of radical rings related with solutions of the Yang-Baxter equation. He proves there is some correspondence between non-degenerate,  involutive  set-theoretic solutions and left braces.  In his subsequent papers  \cite{rump_braces2, rump-braces3}, he deepened the study of this new structure. In \cite{brace}, the authors give an equivalent definition of brace, and they prove some propeties. In particular,  they use  braces  to solve  a problem araised by T. Gateva-Ivanova and P. Cameron in \cite{gateva-phys2}. Braces are intensively studied and the following list of references on the topic is certainly not exhaustive \cite{bachiller0,bachiller1,bachiller2}, \cite{brace-ag}, \cite{gateva_new},  \cite{smot}.\\

 Roughly, a brace is a triple $(\mathcal{B},+,\cdot)$, where  $(\mathcal{B},+)$ is an abelian group, 
 $(\mathcal{B},\cdot)$ is a group, and there is a  left-distributivity-like  axiom that relates between the two operations in $\mathcal{B}$.  For a left brace,  this  is  the following relation:  $a \cdot(b+c)=a\cdot b+a\cdot c -a$, for every $a,b,c \in \mathcal{B}$.  Several extensions of the structure of left brace have been defined. L. Guarnieri and L. Vendramin define a skew left brace to be a triple $(\mathcal{B},+,\cdot)$, with both  $(\mathcal{B},+)$  and  $(\mathcal{B},\cdot)$  groups, and  a  left-distributivity-like  axiom,  and they prove that there is some  correspondence between skew left braces and non-degenerate set-theoretic solutions that are not necessarily involutive \cite{g-vendra1}. We refer to  \cite{s-vendra2}, \cite{catino1} and others for more results on skew braces.\\

 In \cite{catino3}, the authors define a left semi-brace to be a triple $(\mathcal{B},+,\cdot)$, with $(\mathcal{B},+)$  a left-cancellative semigroup and  $(\mathcal{B},\cdot)$  a group,  and  a  left-distributivity-like  axiom,  and they prove that there is some  correspondence between  left semi-braces and left non-degenerate (non-involutive)  set-theoretic solutions. In \cite {catino4}, the authors define a left inverse semi-brace to be a triple $(\mathcal{B},+,\cdot)$, with $(\mathcal{B},+)$  a semigroup,  $(\mathcal{B},\cdot)$  an inverse semigroup, a  left-distributivity-like  axiom, and they prove that there is some  correspondence between  left  inverse semi-braces and  (degenerate and non-involutive)  set-theoretic solutions.\\
 
 In this paper, we define a partial left brace,   to be  a triple $(\mathcal{B},\oplus,\cdot)$, where $(\mathcal{B},\oplus)$  is a commutative partial monoid in the sense of \cite{partial-s},   $(\mathcal{B},\cdot)$  is an inverse monoid,  and the  axiom is  left-distributivity (whenever defined). This  structure is very reminiscent of  a left inverse semi-brace,  but its motivation is completely different, as it does not correspond to  a set-theoretical solution, but to  an extension of  a set-theoretic solution.  Indeed, we consider a pair $(X,r)$, with $X$ a set and $r:  \mathcal{D}  \rightarrow X \times X$, $r(x,y)=(\sigma_x(y),\gamma_x(y))$,  $\mathcal{D} \subseteq X \times X$, a partial bijection, where  $\sigma_x:\mathcal{D}_{\sigma_x}\rightarrow\mathcal{R}_{\sigma_x}$, 
 $\gamma_y:\mathcal{D}_{\gamma_y}\rightarrow\mathcal{R}_{\gamma_y}$ are maps, and $\mathcal{D}_{\sigma_x},\mathcal{R}_{\sigma_x},\mathcal{D}_{\gamma_y},\mathcal{R}_{\gamma_y}\subseteq X$. We define \emph{a partial set-theoretic solution}  to be such a pair   $(X,r)$  that satisfies an extension of the definition of  braided, and show that a partial brace is the natural corresponding structure (Theorem \ref{theo-intro-partial-brace}). 
 
 The linear operator induced by a partial set-theoretic solution,  $R$,  is defined as $R: W \rightarrow V\otimes V$, where  $V$ is the  real vector space spanned by $X$,  $W$ is a subspace of   $ V\otimes V$ and $R$  satisfies the quantum Yang-Baxter equation in some subspace of  $ V\otimes V\otimes V$. Such a kind of  operators occur in the context of rational conformal field theory, and string theory.  Indeed, if $V$ is a vector space with a spanning vector for each coupling, the constraints on the allowed interactions or couplings  take the form of relations satisfied by  two linear operators,  $B$ and $F$,  called the braiding and the fusion operators respectively.  Both  operators are defined from subspaces of   $ V\otimes V$ to  $ V\otimes V$,  and $B$ satisfies the Yang-Baxter equation in a subspace of   $ V\otimes V \otimes V$. We refer to \cite{moore} for more details. 
 
  We  extend the definitions of non-degenerate and involutive to partial set-theoretic solutions. We define \emph{the structure group} of a non-degenerate, involutive partial set theoretic solution to be $G(X,r)=\operatorname{Gp} \langle X\mid\ xy =\sigma_x(y)\gamma_y(x)\ ;\ (x,y)\in \mathcal{D} \rangle$.  
 The \emph{structure inverse monoid} of $(X,r)$ is 	$\operatorname{IM}(X,r)=\operatorname{Inv} \langle X\mid\ xy =\sigma_x(y)\gamma_y(x)\ ;\ (x,y)\in \mathcal{D} \rangle$.  We prove the following:
 \begin{thm_A}
 	Let $(X,r)$ be a square-free, non-degenerate, involutive partial set-theoretic solution, with structure inverse monoid $\operatorname{IM}(X,r)$. Let $\operatorname{I}_X$ denote the symmetric inverse monoid, that is the set of partial bijections of $X$ (with respect to composition whenever it is defined). Let $A$ denote a commutative inverse monoid generated by a set in bijection with $X$.
 	Then the restricted product $A \Join  \operatorname{I}_X$ exists and is an inverse semigroup. Moreover, $\operatorname{IM}(X,r)$ embeds in  $A \Join  \operatorname{I}_X$.
 \end{thm_A}
We present the relation between partial set-theoretic solutions and partial braces.

\begin{thm_A}\label{theo-intro-partial-brace}
Let $(X,r)$ be a non-degenerate involutive partial solution with structure inverse monoid  $\operatorname{IM}(X,r)$.  Then  there exists a partial left brace $(\mathcal{B},\oplus, \cdot)$  and  a congruence $\rho$ on  $\operatorname{IM}(X,r)$,    such that $\operatorname{IM}(X,r)/\rho$  is isomorphic $(\mathcal{B},\cdot)$.
\end{thm_A}
	\margin_comment{\textcolor{red}{THM 2 CHANGED!!!}} 
 
 In 1965, R. Thompson defined three groups, $F$, $T$ and $V$,  that are nowadays called the  Thompson groups. They were used to construct finitely-presented groups with unsolvable word problem \cite{thompson}. Thompson proved that $T$ and $V$ are finitely-presented, infinite simple groups. The group $F$ is $FP_\infty$ and it is the first example of a torsion-free infinite dimensional  $FP_\infty$ group \cite{brown-geo}.  The present  work arose from an observation of  an   infinite presentation of  the  Thompson group $F$, $F\simeq\operatorname{Gp}\langle x_0,x_1,...\mid x_nx_k=x_kx_{n+1},\, k<n\rangle$ and we found a surprising connection between this and partial solutions. This raises the question of  whether  there are further connections between partial solutions and Thompson's groups in general, and in Section $6$, we present some  points in this direction.
\begin{thm_A}
	There exists a   square-free, non-degenerate, involutive partial set-theoretic solution,  $(X,r)$ such that  its structure group $G(X,r)$ is isomorphic to the Thompson group $F$. Furthermore, $(X,r)$ is irretractable, decomposable and it induces  a non-degenerate and square-free cycle set $(X,\star)$.
 \end{thm_A}

 The paper is organized as follows. In Section $1$, we give some preliminaries on set-theoretic solutions of the Yang-Baxter equation, cycle sets and braces. In Section $2$, we give some preliminaries on inverse semigroups and monoids. In Section $3$, we give some preliminaries on the Thompson group $F$, and some of its properties. In Section $4$, we define partial set-theoretic solutions, we extend the usual  definitions  of non-degenerate, braided and involutive to this context. We prove  some properties of the square-free partial set-theoretic solutions and Theorem $1$. In Section $5$, we define partial braces and introduce the method of right reversing. This process was developed in the context of Garside monoids and groups and it is an important tool in the  proof of  Theorem $2$. In Section $6$, we prove Theorem $3$, and at last we make an attempt to compare between the properties of  set-theoretic solutions and those of partial set-theoretic solutions. Section $7$ is an appendix. 
 
 \begin{acknowledgment}
 	I am very grateful to Mark  Lawson  for his great  help in learning the domain of inverse semigroups, via his book and via the numerous questions I asked him.  	I am also  grateful to Ben Steinberg  for suggesting me the study of the restricted product of a commutative inverse monoid and the symmetric inverse monoid.
 \end{acknowledgment}
\section{Preliminaries on set-theoretic solutions of the  quantum Yang-Baxter equation  (QYBE) }
\subsection{Definition and properties of set-theoretic solutions of the QYBE}
\label{subsec_qybe_Backgd}We  refer to \cite{etingof}, \cite{gateva_van,gateva-phys1,gateva_new,gateva-phys2},  \cite{jespers_i-type,jespers_book,jespers-seul}.\\
Let $X$ be a non-empty set. Let $r: X \times X \rightarrow X \times X$  be a map and write $r(x,y)=(\sigma_{x}(y),\gamma_{y}(x))$,  where $\sigma_x, \gamma_x:X\to X$ are functions  for all  $x,y \in X$.   The pair $(X,r)$ is  \emph{braided} if $r^{12}r^{23}r^{12}=r^{23}r^{12}r^{23}$, where the map $r^{ii+1}$ means $r$ acting on the $i$-th and $(i+1)$-th components of $X^3$.  In this case, we  call  $(X,r)$  \emph{a set-theoretic solution of the quantum Yang-Baxter equation}, and whenever $X$ is finite, we  call  $(X,r)$  \emph{a finite set-theoretic solution of the quantum Yang-Baxter equation}.  The pair $(X,r)$ is \emph{non-degenerate} if for every  $x\in X$,  $\sigma_{x}$ and $\gamma_{x}$  are bijective and it   is  \emph{involutive} if $r\circ r = Id_{X^2}$. If $(X,r)$ is a non-degenerate involutive set-theoretic solution, then $r(x,y)$ can be described as  $r(x,y)=(\sigma_{x}(y),\gamma_{y}(x))=(\sigma_{x}(y),\,\sigma^{-1}_{\sigma_{x}(y)}(x))$.  A set-theoretic solution  $(X,r)$ is \emph{square-free},  if  for every $x \in X$, $r(x,x)=(x,x)$. A set-theoretic  solution $(X,r)$ is \emph{trivial} if $\sigma_{x}=\gamma_{x}=Id_X$, for every  $x \in X$. 
 \begin{lem}\cite{etingof} \label{lem-formules-invol+braided}
 	\begin{enumerate}[(i)]
 	\item  $(X,r)$ is involutive  if  and only if   $\sigma_{\sigma_x(y)}\gamma_{y}(x)=x$ and $\gamma_{\gamma_y(x)}\sigma_x(y)=y$,  $x,y \in X$.
\item   $(X,r)$ is  braided if  and only if, 	for  every $x,y,z \in X$, the following  holds:
 \begin{align*}
& \sigma_x\sigma_y=\sigma_{\sigma_x(y)}\sigma_{\gamma_y(x)}\\ 
& \gamma_y\gamma_x=\gamma_{\gamma_y(x)}\gamma_{\sigma_x(y)}\\
&\gamma_{\sigma_{\gamma_y(x)}(z)}(\sigma_x(y))=\sigma_{\gamma_{\sigma_y(z)}(x)}(\gamma_z(y))
\end{align*}
 \end{enumerate}
\end{lem}
\begin{defn}
Let  $(X,r)$ be a non-degenerate involutive set-theoretic solution of the QYBE. 
\begin{enumerate}[(i)]
	\item  A set $Y \subset X$ is  \emph{invariant} if  $r(Y \times Y)\subset Y \times Y$.
	\item An invariant  subset $Y \subset X$  is  \emph{non-degenerate} if  $(Y,r\mid_{Y^2})$ is non-degenerate involutive set-theoretic solution of the QYBE. 
	\item  $(X,r)$  is \emph{decomposable} if it is a union of two non-empty disjoint non-degenerate invariant subsets. Otherwise, it is called \emph{indecomposable}.
\end{enumerate}

\end{defn}
A very simple class of  non-degenerate involutive set-theoretic solutions of the QYBE is the class of  \emph{permutation solutions}. These solutions have the form $r(x,y)=(\sigma(y),\sigma^{-1}(x))$,  where the bijections  $\sigma_x: X\to X$  are  all equal and equal to $\sigma$, the   bijections  $\gamma_x: X\to X$  are  all equal and equal to $\sigma^{-1}$. If $\sigma$ is a cyclic permutation, $(X,r)$ is a \emph{cyclic permutation solution}. A permutation solution is indecomposable if and only if it is cyclic \cite[p.184]{etingof}.
\begin{defn}
	Let  $(X,r)$   be a set-theoretic solution of the QYBE. The \emph{structure group} of $(X,r)$ is  defined by $G(X,r)=\operatorname{Gp} \langle X\mid\ xy =\sigma_x(y)\gamma_y(x)\ ;\ x,y\in X \rangle$.  
\end{defn}
The  structure group of the trivial solution is $\mathbb{Z}^{X }$.  Two  set-theoretic solutions $(X,r)$ and $(X',r')$ are \emph{isomorphic} if there is a bijection $\alpha:X \rightarrow X'$ such that $(\alpha \times \alpha) \circ r=r'\circ (\alpha \times \alpha)$  \cite{etingof}. If  $(X,r)$ and $(X',r')$ are isomorphic, then $G(X,r) \simeq G(X',r')$, with  $G(X,r)$ and $G(X',r')$ their respective structure groups. 
An important characterisation of  non-degenerate involutive set-theoretic solutions of the QYBE is  presented in the following proposition.
\begin{prop}\cite[p.176-180]{etingof}\label{prop-etingof}
Let  $(X,r)$ be a non-degenerate involutive set-theoretic solution of the QYBE, defined by  $r(x,y)=(\sigma_{x}(y),\gamma_{y}(x))$,   $x,y \in X$,  with structure group $G(X,r)$. Let $\mathbb{Z}^{X}$ denote the free abelian group with basis $\{t_x \mid x \in X\}$, and  $\operatorname{Sym}_X$ denote the symmetric group of $X$. Then 
\begin{enumerate}[(i)]
		\item The map $\varphi: G(X,r) \rightarrow \operatorname{Sym}_X$, defined by $x \mapsto \sigma_{x}$,  is a homomorphism of groups.
		\item  The group $\operatorname{Sym}_X$  acts on $\mathbb{Z}^{X}$.
		\item The group $G(X,r) $  acts on $\mathbb{Z}^{X}$:  if  $g \in G$, then $g \bullet t_x=t_{\alpha(x)}$, with $\alpha=\varphi(g)$. 
	\item The map 	$\pi:G(X,r)\rightarrow \mathbb{Z}^{X}$ is a bijective $1$-cocycle, where  $\pi(x)=t_x$,  for $x \in X$, and  $\pi(gh) =\pi(g)+g \bullet \pi(h)$,  for $g,h \in G(X,r)$.

 \item There  is a monomorphism of groups $\psi: G(X,r) \rightarrow \mathbb{Z}^{X} \rtimes \operatorname{Sym}_X$:  $\psi(x)=(t_x,\sigma_x) $, $\psi(g)=(\pi(g), \varphi(g))$.
 
 \item  The group $G(X,r) $ is isomorphic to  a subgroup of  $\mathbb{Z}^{X} \rtimes \operatorname{Sym}_X$ of the form $H=\{(a,\phi(a))\mid a \in \mathbb{Z}^{X} \}$, where $\phi: \mathbb{Z}^{X} \rightarrow \operatorname{Sym}_X$, is defined by $\phi(a)=\varphi(g)$,  whenever  $\pi(g)=a$.
 \end{enumerate}
\end{prop}
The subgroup of $\operatorname{Sym}_X$  generated by $\{\sigma_x\mid x\in X\}$ is denoted by $\mathcal{G}(X,r)$ and is called \emph{a IYB group} \cite{cedo}, \cite{gateva-phys2}.

\begin{defn}
	The \emph{retract} relation $\sim$ on the set $X$ is defined by $x \sim y$ if $\sigma_x=\sigma_y$. There is a natural induced solution $Ret(X,r)=(X/\sim,r)$, called the \emph{the retraction of $(X,r)$}, defined by $r'([x],[y])=([\sigma_{x}(y)],[\gamma_y(x)])$.
	A non-degenerate involutive set-theoretic solution  $(X,r)$   is called \emph{a multipermutation solution of level $m$} if $m$ is the smallest natural number such that the solution $\mid Ret^m(X,r)\mid=1$, where 
	$Ret^k(X,r)=Ret(Ret^{k-1}(X,r))$, for $k>1$. If such an $m$ exists, $(X,r)$   is also called \emph{retractable}, otherwise it is called \emph{irretractable}.
\end{defn}

\begin{ex} \label{exemple:exesolu_et_gars} Let $X = \{x_1,x_2,x_3,x_4\}$, and ~$r: X\times X\to X\times X$ be defined by $r(x_i,x_j)=(x_{g_{i}(j)},x_{f_{j}(i)})$ where~$g_i$ and $f_j$ are permutations on~$\{1,2,3,4\}$ as follows: $g_{1}=(2,3)$, $g_2=(1,4)$, $g_{3}=(1,2,4,3)$, $g_{4}=(1,3,4,2)$; $f_{1}=(2,4)$, $f_2=(1,3)$, $f_{3}=(1,4,3,2)$, $f_{4}=(1,2,3,4)$.
	Then  $(X,r)$ is  an indecomposable, and irretractable  solution, with   structure group $G=\operatorname{Gp} \langle X\mid  x_{1}x_{2}=x^{2}_{3}; x_{1}x_{3}=x_{2}x_{4};
	x_{2}x_{1}=x^{2}_{4}; x_{2}x_{3}=x_{3}x_{1};
	x_{1}x_{4}=x_{4}x_{2};x_{3}x_{2}=x_{4}x_{1} \rangle$. 
\end{ex}    
Structure monoids of non-degenerate, involutive set-theoretic solutions of the QYBE are Garside monoids,  that satisfy  interesting properties \cite{chou_art,chou_godel1,chou_godel2}, \cite{gateva}. 
\subsection{The Yang-Baxter equation, Cycle sets, and Braces}
Gateva-Ivanova conjectured that square-free, non-degenerate involutive set-theoretic solutions of the QYBE  are decomposable \cite{gateva_conj}.  In \cite{rump}, Rump proved that Gateva-Ivanova's conjecture is true for finite square-free  solutions and not necessarily true whenever the finiteness assumption is removed.  His proof is based on cycle sets, a new structure  introduced in \cite{rump}. 
\begin{defn}\label{defn-cycle-set}
 A \emph{cycle set} is a   a non-empty set  $X$ with a binary operation $\star$ that satisfies:
 \begin{enumerate}[(i)]
 	\item The map $\tau(x)$,   defined by $\tau(x)(y)= x \star y$ is invertible, for every $x\in X$.
 	\item   $(x \star y)\star (x \star z)=(y\star x)\star (y \star z)$,  for every $x,y,z \in X$.
 \end{enumerate} 

A cycle set is \emph{non-degenerate} if  the map $x \mapsto x \star x$ is bijective,  for all $x \in X$.

A cycle set is \emph{square-free} if  $x \star x=x$, for all $x \in X$.

\end{defn}

\begin{thm}\cite[Prop. 1]{rump}
There is a bijective correspondence between non-degenerate cycle sets and non-degenerate involutive  set-theoretic solutions of the QYBE.
\end{thm}
Given $X$  a non-degenerate cycle set. The  pair $(X,r)$ is a non-degenerate involutive  set-theoretic solution of the YBE, with $r(x,y)=((\tau(x))^{-1}(y), (\tau(x))^{-1}(y) \star x)$.
Given a non-degenerate involutive  set-theoretic solution of the YBE,  $(X,r)$,   with $r(x,y)=(\sigma_x(y),\gamma_y(x))$, $X$  is  a non-degenerate cycle set with $\tau(x)(y)=x \star y=\sigma_x^{-1}(y)$.\\

In \cite{rump_braces}, Rump introduced braces as a generalization of radical rings related with non-degenerate 
involutive set-theoretic solutions of the QYBE. In subsequent papers, he developed the theory of this new algebraic structure. In \cite{brace}, the authors give another equivalent definition of a brace and study its structure. We follow the terminology from  \cite{brace}.
\begin{defn}
	A \emph{left brace}  is a set  $G$ with two operations, $+$ and $\cdot$, such that $(G+)$ is an abelian group, $(G,\cdot)$ is a group and for every $a,b,c \in G$:
\begin{equation}\label{eqn-brace}
	a \cdot (b+c) = a \cdot b+a \cdot c -a
\end{equation}
	The groups  $(G,+)$  and $(G\cdot)$ are called \emph{the additive group} and \emph{the multiplicative group of the brace},  respectively.\\
	A right brace is defined similarly, by replacing Equation \ref{eqn-brace} by 
	\begin{equation}\label{eqn-brace-right}
(a+b) \cdot c +c = a \cdot c+b \cdot c
	\end{equation} 
\end{defn}
A \emph{two-sided brace} is a left and right brace, that is both Equations \ref{eqn-brace} and  \ref{eqn-brace-right} are satisfied. From the definition of a left brace $G$, it  follows  that the multiplicative identity of the multiplicative group of $G$ is equal to the neutral element of the additive group of $G$. Additionnally,  for every $a,b,c \in G$, $a\cdot (b-c)=a\cdot b-a \cdot c +a$.

\section{Preliminaries on Inverse semigroups and monoids}

\subsection{Definition and properties of Inverse semigroups and monoids} We refer to \cite{lawson}, \cite{clifford}, \cite{wagner}.
A \emph{regular  semigroup} is a semigroup $S$  such that for every element $s\in S$ there exists at least one   element $s^{*} \in S$ such that $ss^{*}s=s$ and $s^{*} ss^{*} =s^{*} $. The element $s^{*}$ is  called \emph{the inverse of $s$}.
An \emph{inverse semigroup} is a semigroup $S$  such that for every element $s\in S$ there exists a unique element $s^{*} \in S$ such that $ss^{*}s=s$ and $s^{*} ss^{*} =s^{*} $, that is $S$ is a regular semigroup such that every element in $S$ has a unique inverse. An \emph{inverse monoid}  $S$ is an inverse semigroup with $1$, and  if  additionally,  for every $s\in S$,  $s^{*} s=s^{*}s=1$, then $S$ is a group.  Another equivalent definition of an inverse semigroup is a regular semigroup in which all the idempotents commute, that is the set $E(S)$ of idempotents of  an inverse semigroup $S$  is a commutative subsemigroup; it is ordered by $e\leq f$ if and only if $ef=e=fe$. The order on $E(S)$ extends to $S$ as the so-called natural partial order by putting $s\leq t$  if $s=et$ for some idempotent $e$ (or equivalently $s=tf$ for some idempotent $f$). This is equivalent to $s=ts^*s$ or $s=ss^*t$. 
For an idempotent $e$, the set $G_e=\{s\in S \mid ss^*=s^*s=e\}$ is a group (called the maximal subgroup of $S$ at $e$). Idempotents $e$ and $f$ are said to be $\mathscr{D}$-equivalent, written $e\mathscr{D}f$, if there exists $s\in S$ so that $e=s^*s$ and $f=ss^*$. If $S$ is an inverse semigroup and $\varphi:S\rightarrow T$ is a homomorphism  of semigroups, then the homomorphic image of $S$ is an inverse semigroup and the property $\varphi(s^*)=(\varphi(s))^*$  holds for every $s \in S$. \\

An   important class of inverse monoids  is the class of commutative inverse monoids (or semigroups).  Each  element in a commutative inverse monoid  $A$ generated by a set $X$  is  in one-to-one correspondence with  a partial function with finite support  $f: X\rightarrow \mathbb{Z}$, $x_{i_j}\mapsto m_{i_j}$, $1 \leq j \leq k$ and $\mathcal{D}_f$, the  domain of $f$ is  $ \{x_{i_1},...,x_{i_k}\}$.  The operation in $A$ is  defined  pointwise, with  $(f+g)(x)=f(x)+g(x)$, where  $\mathcal{D}_{f+g}=\mathcal{D}_f\cap \mathcal{D}_g$. The identity is $0_X$, the zero function on $X$.

Another    important class of inverse monoids  is the class of symmetric inverse monoids. A \emph{partial bijection} of  a set $X$ is a bijection $f$ between two (non-necessarily proper) subsets of $X$, the domain and the range   of $f$  are  denoted by $\mathcal{D}_f$, and  $\mathcal{R}_f$ respectively. If  $f \neq Id$,  where $Id$ denotes the identity function on $X$, the domain of $f$ is allowed to be the empty set.
The set of all partial bijections of a set $X$, equipped with the operation of  composition of  functions $\circ$, is an inverse monoid with zero, the zero element being the vacuous map $\emptyset \rightarrow\emptyset$. It is  called the \emph{symmetric inverse monoid} and it is denoted by $\operatorname{I}_X$.  If $f,g \in \operatorname{I}_X$, then $f \circ g$ is the composition of partial maps in the largest domain where it makes sense, that is
$\mathcal{D}_{f \circ g}=g^{-1}(\mathcal{D}_{f }\cap \mathcal{R}_{g})$, and  $\mathcal{R}_{f \circ g}=f(\mathcal{D}_{f }\cap \mathcal{R}_{g})$. The idempotents of $I_X$  are precisely the partial identities on $X$ \cite[p.6]{lawson}. From the Wagner-Preston Theorem, the analogue of Cayley's Theorem from  group theory,  any inverse semigroup $S$ can be embedded into  some  symmetric inverse semigroup, that is  $S$  is isomorphic to the subsemigroup of  a  symmetric inverse semigroup \cite{preston}, \cite{wagner}.

\begin{defn}\label{defn-partial-action}
	Let $M$  be an inverse   monoid. Let $X$ be a set.  Let $S$ be an inverse semigroup.  
	We say that  $M$ \emph{acts  on  $X$   by partial permutations} if there exists a homomorphism of monoids  $ M \rightarrow I_X$.
\end{defn}

\begin{defn} \cite{lawson}\label{defn-action}
Let $S$ be a semigroup. The monoid $M$ is said   \emph{to act  on  $S$ (on the left) (by endomorphisms)} if there exists a map $M\times S \rightarrow S$, denoted by $(a,s) \mapsto a\bullet s$ satisfying the following conditions:
\begin{enumerate}[(i)]

	\item  for any $a,b\in M$,  $s\in S$,   $(ab)\bullet s= a\bullet  (b\bullet s)$.
		\item  for any $a\in M$,  $s,s'\in S$,   $a\bullet  (ss')= (a\bullet s)(a\bullet s')$.  
\item  for every $s\in S$, $1 \bullet s=s$.\\
		
		\noindent If $S$ is a semigroup with zero $0$, the additional following property is required:
		\item  for any $a\in M$,   $ a \bullet  0=0$.
\end{enumerate}
\end{defn}
This definition is equivalent to the existence of a homomorphism of monoids from $M$ to the monoid of endomorphisms of $S$. This homomorphism maps inverses  to inverses and idempotents to idempotents. Let $M$ act on $S$ by endomorphisms. The (classical) semidirect product of  $S$ by $M$ is the set $S\times M$ equipped with the product $(s,a)(s',b)=(s(a\bullet  s'),ab)$. If $M$ and $S$ are inverse semigroups, their semidirect product is not necessarily an inverse semigroup. This leads to the definition of the $\lambda$-semidirect product and the restricted product of inverse semigroups or monoids.
\begin{defn}\cite[p.147]{lawson}\label{defn-lamb}
Let $M$, $S$ be inverse semigroups. Assume $M$ acts  on  $S$ by endomorphisms. Let $S \rtimes^{\lambda}M$ be the following set with the binary operation defined below:
\[S \rtimes^{\lambda}M=\{(s,m)\in S \times M \mid  r(m)\bullet s=s\}\]
\[(s,m)(s',m')=\,(\,(r(mm')\bullet s)(m\bullet s'),\,mm' )\]
$S \rtimes^{\lambda}M$ is an inverse semigroup called \emph{the $\lambda$-semidirect product of $S$ by $M$}. If both $S$ and $M$ are inverse monoids, then the  $\lambda$-semidirect product of $S$ by $M$, $S \rtimes^{\lambda}M$,  is an inverse monoid.
\end{defn}

\begin{defn}\cite[p.155]{lawson}\label{defn-restricted}
	Let $M$, $S$ be inverse semigroups. Let $E(M)$ denote the set of idempotents of $M$ (ordered by $e\leq f$ if and only if $ef=e=fe$).  Assume the following assumptions:
	\begin{enumerate}[(i)]
		\item 	 $M$ acts  on  $S$ by endomorphisms. 
		\item   There exists  a surjective homomorphism $\epsilon: S \rightarrow E(M)$.
		\item For each $s \in S$, there exists $\epsilon(s) \in E(M)$ such that  \[\epsilon(s)\leq e \Longleftrightarrow e \bullet s=s, \,\forall  e \in E(M)\]
			\end{enumerate}
Let $S \Join M$ be the following set with the binary operation defined below:
	\[S \Join M=\{(s,m)\in S \times M \mid  r(m)=\epsilon(s) \}\]
	\[(s,m)(s',m')=\,(\,s(m\bullet s'), \,mm' )\]
$S \Join M$ is an inverse semigroup. Furthermore, it is  an inverse subsemigroup  of 	$S \rtimes^{\lambda}M$.
\end{defn}

\subsection{The free inverse monoid, Inverse monoid presentations}
We refer the reader to \cite[Chapter 6]{lawson} for more details.
\begin{defn}\cite[p.171]{lawson}
	Let $X$ be a non-empty set. An inverse (semigroup)  monoid $\operatorname{FIM}(X)$, equipped with a function $i:X\rightarrow \operatorname{FIM}(X)$, is said to be a \emph{free inverse (semigroup) monoid  on $X$} if for every inverse (semigroup) monoid $M$ and function $\kappa: X\rightarrow M$ there exists a unique homomorphism $\theta: \operatorname{FIM}(X) \rightarrow M$ such that $\theta\circ i=\kappa$.
\end{defn}
Free inverse (semigroup) monoids exist, and $\operatorname{FIM}(X)$ is  constructed as follows. Let $X^*=\{x^*\mid x\in X\}$, a set in bijection with $X$ and disjoint from $X$. Let $FM$ be the free monoid generated by $X\cup X^*$, and define  the following unary operation: first, if $y=x\in X$, then $y^*=x^*$ and if $y=x^* \in X^*$,  $y^*=x$; next, if $y_1...y_k\in FM$, then  $(y_1...y_k)^*=y_k^*...y_1^*$. This unary operation turns $FM$ into a free monoid with involution. The  free inverse  monoid  $\operatorname{FIM}(X)$ is  then obtained by factoring  $FM$ by $\rho'$, the congruence generated by the following binary relation $\rho=\{(u,uu^*u), \,(uu^*vv^*,vv^*uu^*)\mid u,v \in FM\}$. If $u$ and $v$ belong to the same $\rho'$-class, then we shall  say that  \emph{$u$ and $v$ are \emph{equivalent} or $u$ and $v$  represent the same element in $\operatorname{FIM}(X)$}. \\

The word problem is solvable in  free inverse monoids. There exists an approach to the solution of the word problem in  free inverse monoids  which is very similar to the one in free groups, that is based on the existence of a unique normal (reduced) form.  The existence of free inverse monoids gives the possibility to define the notion of an inverse monoid presentation.
\begin{defn}
	An \emph{inverse monoid presentation} is a pair $(X,\rho)$, where $X$ is a set and $\varrho$ is a relation on $\operatorname{FIM}(X)$. The inverse monoid $\operatorname{FIM}(X)/\varrho'$ , where $\varrho'$ is the congruence generated by $\varrho$, is said to be \emph{presented} by the \emph{generators} $X$ and the \emph{relations} $\varrho$ and is denoted by $IM=\operatorname{Inv}\langle X \mid \varrho \rangle$. If both $X$ and $\varrho$ are finite, we say that $IM$ is finitely generated.
\end{defn}

\section{Preliminaries on Thomson group $F$}
In 1965, R. Thompson defined three groups, $F$, $T$ and $V$,  that are nowadays called the  Thompson groups. They were used to construct finitely-presented groups with unsolvable word problem \cite{thompson}. Thompson proved that $T$ and $V$ are finitely-presented, infinite simple groups. The group $F$ is $FP_\infty$ and it is the first example of a torsion-free infinite dimensional  $FP_\infty$ group \cite{brown-geo}. In this paper, we consider the Thompson group $F$. There are several descriptions of the Thompson group $F$, and we present two of them: as the group of dyadic rearrangements (or piecewise-linear homeomorphisms with additional properties), and as the group of tree diagrams with a certain product.  We refer to \cite{thompson}, \cite{canon}, \cite{brin}, \cite{stein}, \cite{belk} and many others for more details.
\subsection{$F$ as the group of dyadic rearrangements, and  tree diagrams}
Any subdivision of  the interval $[0,1]$ obtained by repeatedly cutting intervals in half is called \emph{a dyadic subdivision}. The intervals of a dyadic subdivision are all of the form : $[\frac{k}{2^n}, \frac{k+1}{2^n}]$, $k,n \in \mathbb{N}$. The set of all dyadic rearrangements forms a group under composition, the  Thompson group $F$.  
\begin {ex}
To illustrate, take  the interval $[0,1]$  and cut in half, like this:\\
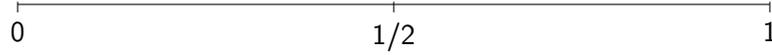
\begin{figure}[H]
	\begin{tikzpicture}[y=.0cm, x=10cm,font=\sffamily]
	\draw (0,0) -- coordinate (x axis mid) (1,0);
	
	\foreach \x in {0,1/2,1}
	\draw (\x,1pt) -- (\x,-3pt)
	node[anchor=north] {\x};
	

	\draw plot[mark=*, mark options={fill=black}] 
	file {div_soft.data};
	
	\end{tikzpicture}
\end{figure} 
We then cut each of the resulting intervals in half:\\

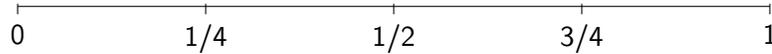
\begin{figure}[H]
	
	\begin{tikzpicture}[y=.0cm, x=10cm,font=\sffamily]
	\draw (0,0) -- coordinate (x axis mid) (1,0);
	
	\foreach \x in {0,1/4,1/2,3/4,1}
	\draw (\x,1pt) -- (\x,-3pt)
	node[anchor=north] {\x};
	

	\draw plot[mark=*, mark options={fill=black}] 
	file {div_soft.data};
	
	\end{tikzpicture}
\end{figure}
and then cut some  of the new  intervals in half to get a certain subdivision of $[0,1]$:\\

\begin{figure}[H]\label{fig-dyadic-interval}
	
	\begin{tikzpicture}[y=.0cm, x=10cm,font=\sffamily]
	\draw (0,0) -- coordinate (x axis mid) (1,0);
	
	\foreach \x in {0,1/8,1/4,1/2,5/8,3/4,1}
	\draw (\x,1pt) -- (\x,-3pt)
	node[anchor=north] {\x};
	

	\draw plot[mark=*, mark options={fill=black}] 
	file {div_soft.data};
	
	\end{tikzpicture}
	\caption{A standard dyadic interval}
\end{figure}
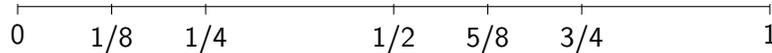
Given a pair of dyadic subdivisions, $\mathcal{D} $, $\mathcal{R}$, with the same number of cuts, a \emph{dyadic rearrangement of  $[0,1]$} is a picewise-linear homomorphism $f: [0,1] \rightarrow[0,1]$ that sends each interval of $\mathcal{D} $ linearly onto the corresponding interval of $\mathcal{R}$.

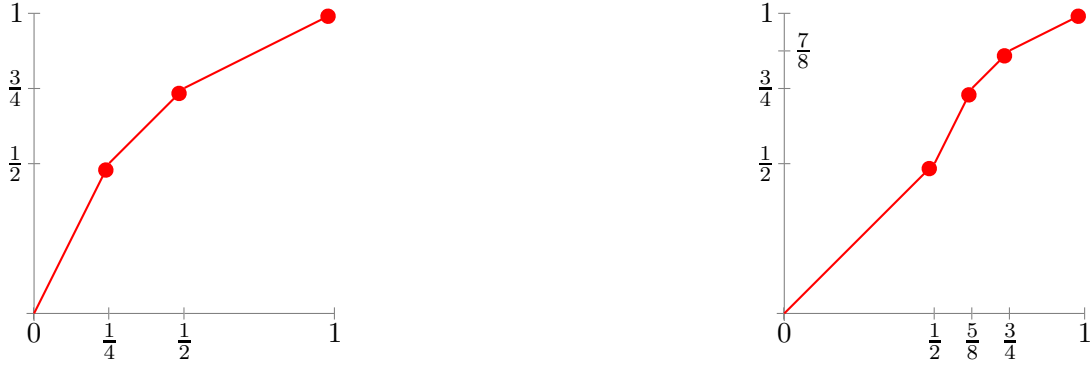
\begin{figure}[H]\label{fig-maps}
	\begin{tikzpicture}
	\draw [thin, gray, -|] (-0.1,0) -- (0,0)
	node [below, black] {$0$};    
	\draw [thin, gray, -|] (0,0) -- (1,0)
	node [below, black] {$\frac{1}{4}$};   
	\draw [thin, gray, -|] (1,0) -- (2,0)
	node [below, black] {$\frac{1}{2}$};   
	\draw [thin, gray, -|] (2,0) -- (4,0)
	node [below, black] {$1$};

	\draw [thin, gray, -|] (0,-0.1) -- (0,2)      
	node [left, black] {$\frac{1}{2}$};              
	\draw [thin, gray, -|] (0,2) -- (0,3)      
	node [left, black] {$\frac{3}{4}$};                   
	\draw [thin, gray, -|] (0,3) -- (0,4)      
	node [left, black] {$1$};                            
	
	\draw [draw=red, thick,-*] (0,0) -- (1,2);
	\draw [draw=red, thick,-*] (1,2) -- (2,3);
	\draw [draw=red, thick,-*] (2,3) -- (4,4);
	\end{tikzpicture}
	\hspace{5cm}
	\begin{tikzpicture}
	\draw [thin, gray, -|] (-0.1,0) -- (0,0)
	node [below, black] {$0$};    
	\draw [thin, gray, -|] (0,0) -- (2,0)
	node [below, black] {$\frac{1}{2}$};   
	\draw [thin, gray, -|] (2,0) -- (2.5,0)
	node [below, black] {$\frac{5}{8}$};   
	\draw [thin, gray, -|] (2.5,0) -- (3,0)
	node [below, black] {$\frac{3}{4}$};   
	\draw [thin, gray, -|] (3,0) -- (4,0)
	node [below, black] {$1$};

	\draw [thin, gray, -|] (0,-0.1) -- (0,2)      
	node [left, black] {$\frac{1}{2}$};              
	\draw [thin, gray, -|] (0,2) -- (0,3)      
	node [left, black] {$\frac{3}{4}$};     
	\draw [thin, gray, -|] (0,3) -- (0,3.5)      
	node [right, black] {$\frac{7}{8}$};                   
	\draw [thin, gray, -|] (0,3.5) -- (0,4)      
	node [left, black] {$1$};

	\draw [draw=red, thick,-*] (0,0) -- (2,2);
	\draw [draw=red,thick,-*] (2,2) -- (2.5,3);
	\draw [draw=red, thick,-*] (2.5,3) -- (3,3.5);
	\draw [draw=red, thick,-*] (3,3.5) -- (4,4);
	\end{tikzpicture}
	\caption{Dyadic rearrangements corresponding to  $x_0$ at left and $x_1$ at right}
\end{figure}
\end{ex}
To each standard dyadic interval there corresponds  a binary tree. The binary tree corresponding to the dyadic interval from Figure \ref{fig-dyadic-interval} is described in Figure \ref{fig-tree}:
\begin{figure}[H]\label{fig-tree}
\scalebox{0.9}[0.9]{	\begin{tikzpicture}[level distance=1.5cm,
	level 1/.style={sibling distance=4cm},
	level 2/.style={sibling distance=2cm},
	level 3/.style={sibling distance=1.5cm}]
	\node {[0,1]}
	child {node {[0,$\frac{1}{2}$]}
	            	child { node {[0,$\frac{1}{4}$]} child {node {[0,$\frac{1}{8}$]}} child {node {[$\frac{1}{8},\frac{1}{4}$]}}
	                }
					child {node {[$\frac{1}{4}$,$\frac{1}{2}$]}}
		}
child {node {[$\frac{1}{2}$,1]}
	   child {node {[$\frac{1}{2}$,$\frac{3}{4}$]} child {node {[$\frac{1}{2}$,$\frac{5}{8}$]}} child {node {[$\frac{5}{8}$,$\frac{3}{4}$]}}
	   }
	   child {node {[$\frac{3}{4}$,1]}}
	};
	\end{tikzpicture}}\hspace{0.2cm}\scalebox{0.5}[0.5]{	\begin{tikzpicture} 
\draw [draw=black, thick,-*] (0,0) -- (-7,-9);
 \draw [draw=black, thick,-*] (-4.7,-6) -- (-2,-9);
  \draw [draw=black, thick,-*] (-6.2,-8) -- (-5,-9);

 \draw [draw=black, thick,-*] (0,0) -- (7,-9);
 \draw [draw=black,thick,-*] (4.7,-6) -- (2,-9);
\draw [draw=black,thick,-*] (2.9,-8) -- (4,-9);

	\end{tikzpicture}}
\caption{The binary tree  for the  dyadic interval from Fig. \ref{fig-dyadic-interval}  (at left the detailed tree and  at right the schematic one).}
\end{figure}
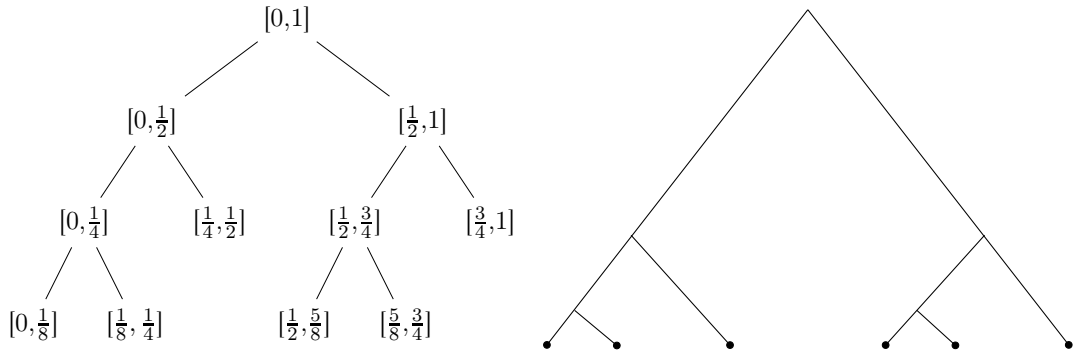
Any element of the group $F$ can be described by a pair of finite binary trees, called \emph{a tree diagram}. The two trees are aligned vertically  so that corresponding leaves match up. The domain tree appears on the top and the range tree appears on the bottom. It is illustrated in the following figure.
	\begin{figure}[H]
	\scalebox{0.8}[0.8]{	\begin{tikzpicture}
		\draw [draw=black, thick,-*] (0,0) -- (-2,-2);
		\draw [draw=black, thick,-*] (-1,-1) -- (-0.5,-2);
	\draw [draw=black, thick,-*] (0,0) -- (2,-2);

		\draw [draw=black,thick,-*] (0,-5) -- (-2,-3);
		\draw [draw=black,thick,-*] (0,-5) -- (2,-3);
			\draw [draw=black, thick,-*] (1,-4) -- (-0.2,-3);
	\end{tikzpicture}} \hspace{1cm}\scalebox{0.7}[0.7]{\begin{tikzpicture}
		\draw [draw=black, thick,-*] (0,0) -- (-3,-3);
		\draw [draw=black, thick,-*] (0,0) -- (3,-3);
		\draw [draw=black, thick,-*] (1.2,-1.2) -- (-0.5,-3);
		\draw [draw=black, thick,-*] (0,-2.4) -- (0.7,-3);

		\draw [draw=black,thick,-*] (0,-7) -- (-3,-4);
		\draw [draw=black,thick,-*] (0,-7) -- (3,-4);
		\draw [draw=black, thick,-*] (1.2,-5.8) -- (-0.5,-4);
		\draw [draw=black, thick,-*] (2.4,-4.7) -- (1.6,-4);
		\end{tikzpicture}}\hspace{1cm}\scalebox{0.6}[0.6]{\begin{tikzpicture}
		\draw [draw=black, thick,-*] (0,0) -- (-4,-4);
		\draw [draw=black, thick,-*] (0,0) -- (4,-4);
		\draw [draw=black, thick,-*] (1,-1) -- (-2.1,-4);
		\draw [draw=black, thick,-*] (2,-2) -- (0.1,-4);
		\draw [draw=black, thick,-*] (0.7,-3.3) -- (1.5,-4);

		\draw [draw=black,thick,-*] (0,-9) -- (-4,-5);
		\draw [draw=black,thick,-*] (0,-9) -- (4,-5);
		\draw [draw=black, thick,-*] (1,-8) -- (-2,-5);
		\draw [draw=black, thick,-*] (2,-7) -- (0,-5);
		\draw [draw=black, thick,-*] (2.9,-6.15) -- (1.8,-5);
		\end{tikzpicture}}
	\caption{The tree diagrams corresponding to $x_0$ and $x_1$ (as  described in Fig. \ref{fig-maps}) at left and the tree diagram of  $x_2=x_0x_1x_0^{-1}$ at right.}
	\end{figure}
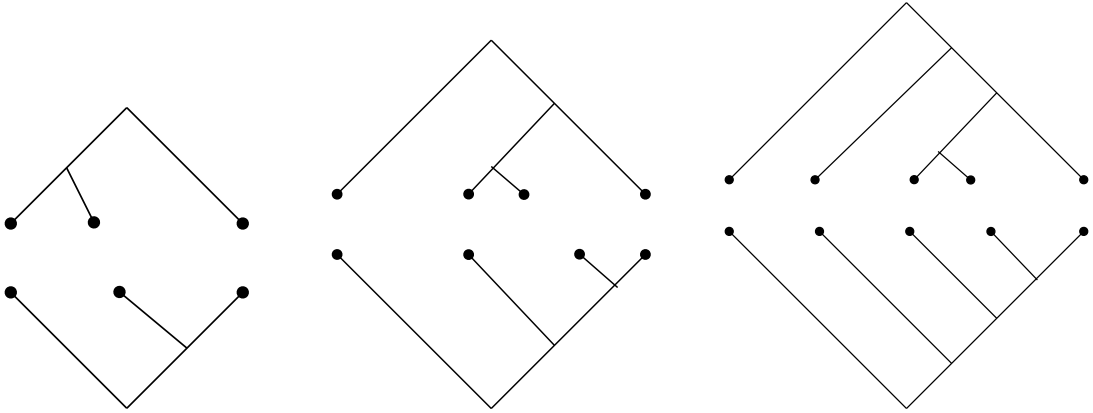

We do not get into  details  on the existence of a unique  reduced tree diagram for any element of $F$, on the product of tree diagrams and  the determination of the presentation of $F$. We refer to the litterature on the topic, and in particular to \cite{belk} for a very detailed introduction to the topic.
\subsection{Presentations of  the group $F$ and some of its  properties}

In the following theorem, we describe two finite presentations of $F$.
\begin{thm}
	\begin{enumerate}[(i)]
		\item  The elements $x_0$ and $x_1$ generate Thompson's group $F$ with  presentation of the form $\langle x_0,x_1\mid x_2x_1=x_1x_3,x_3x_1=x_1x_4 \rangle$, where $x_2=x_0x_1x_0^{-1}$ and $x_3=x_0^2x_1x_0^{-2}$, $x_4=x_0^3x_1x_0^{-3}$ and more generally $x_n=x_0x_{n-1}x_0^{-1}=x_0^{n-1}x_1x_0^{-(n-1)}$,  for every $n \geq 2$.
		\item  The elements $x_0$ and $x_1$ generate Thompson's group $F$ with  presentation of the form $\langle x_0,x_1\mid x_2x_0=x_0x_3,x_3x_0=x_0x_4\rangle $, where $x_2=x_0^{-1}x_1x_0$ and more generally $x_{n+1}=x_{n-1}^{-1}x_nx_{n-1}$, for every $n \geq 2$.
		\end{enumerate}
	\end{thm}
 Surprisingly, although it is usually far more convenient to work with a finite presentation,  in this paper  we work with the infinite presentation of $F$ presented in the following theorem.
\begin{thm}\label{theo-F-infinite-pres}
	\begin{enumerate}[(i)]
		\item  The elements $\{x_0,x_1,x_2,...\}$ generate Thompson's group $F$ and $F$ has an infinite presentation of the form $\langle x_0,x_1,x_2,... \mid x_nx_k=x_kx_{n+1},\, k<n\rangle$.
\item Every element of $F$ can be expressed uniquely in the form $x_0^{a_0}...x_n^{a_n}x_n^{-b_n}...x_0^{-b_0}$, where $a_0,...,a_n,b_0,...,b_n \in \mathbb{N}$, exactly one of $a_n,b_n$ is non-zero and if  both $a_i$ and $b_i$ are not equal $0$, then either $a_{i+1}\neq 0$ or 
$b_{i+1}\neq 0$  for all $i$.
\item Every proper quotient of $F$ is abelian.
	\end{enumerate}
\end{thm}
There  are several questions  about Thompson's group $F$  that are still  open. It is still not known wether $F$ is an automatic  group. It is known that $F$ is not elementary amenable and that it does not conbtain the free group of rank $2$, but  it is still unknown wether $F$ is an amenable  group.

\section{Definition of partial set-theoretic solutions and proof of Theorem $1$ }

\subsection{Definition of partial set-theoretic solutions  and their properties}
 \begin{defn} \label{defn-partial-solution}
 	Let $X$ be a non-empty set. Let $r: \mathcal{D}\rightarrow \mathcal{R}$  be a map, where $ \mathcal{D}, \mathcal{R} \subseteq X \times X$.  We write $r(x,y)=(\sigma_{x}(y),\gamma_{y}(x))$,  where   $\sigma_{x}:\mathcal{D}_{\sigma_{x}}\rightarrow \mathcal{R}_{\sigma_{x}}$ and $\gamma_{y}:\mathcal{D}_{\gamma_{y}}\rightarrow \mathcal{R}_{\gamma_{y}}$, with  $\mathcal{D}_{\sigma_{x}},\mathcal{R}_{\sigma_{x}}, \mathcal{D}_{\gamma_{y}}, \mathcal{R}_{\gamma_{y}} \subseteq X$.     With this notation,  $(x,y)\in \mathcal{D}$  if and only if $y \in  \mathcal{D}_{\sigma_x}$ and $x \in \mathcal{D}_{\gamma_y}$.  Let 
 	$r^{i,i+1}: X^{i-1}\times \mathcal{D}\times X^{k-i-1}\rightarrow X^{i-1}\times \mathcal{R}\times \times X^{k-i-1}$ be  the maps defined by $r^{i,i+1}=Id_{X^{i-1}}\circ r \circ Id_{X^{k-i-1}}$.
 	\begin{enumerate}[(i)]
 			\item The pair $(X,r)$ is \emph{non-degenerate}, if     for every  $x,y \in X$, $\sigma_{x}:\mathcal{D}_{\sigma_{x}}\rightarrow \mathcal{R}_{\sigma_{x}}$ and $\gamma_{y}:\mathcal{D}_{\gamma_{y}}\rightarrow \mathcal{R}_{\gamma_{y}}$ are bijective,  (i.e. $\sigma_{x}$ and $\gamma_{y}$ are partial bijections of $X$).
 			
 				\item 	The pair $(X,r)$ is  \emph{involutive} if   for all pairs  $(x,y)\in X^2$,  $x \in \mathcal{D}_{\gamma_{y}}$	if and only if   $y \in  \mathcal{D}_{\sigma_{x}}$, and additionally  if  $r(x,y)$  is defined,  then $r^{2}(x,y)$  is also defined and satisfies $r\circ r = Id_{X^2}$, that is  $r^{2}(x,y)=(x,y)$. 
 				
 			\item  The pair $(X,r)$ is  \emph{braided} if for all 	 $x,y,z\in X$, 	$r^{12}r^{23}r^{12}(x,y,z)=r^{23}r^{12}r^{23}(x,y,z)$, whenever    $r^{12}r^{23}r^{12}(x,y,z)$ and	$r^{23}r^{12}r^{23}(x,y,z)$ are  defined.
 		\item  The pair $(X,r)$ is    \emph{square-free},  if  for every $x\in X$, $(x,x) \in \mathcal{D}$ and $r(x,x)=(x,x)$. 
 	\end{enumerate}
 If $(X,r)$ is  braided,   we  call  $(X,r)$  \emph{a partial set-theoretic  solution}.  If $(X,r)$ is  a non-degenerate, involutive  partial set-theoretic  solution,  we call it \emph{a partial solution}.
\end{defn}

 \begin{ex}\label{ex-square-free-partial-sol}
 	Let $X=\{x_0,x_1,x_2\}$.   Let $r: \mathcal{D} \rightarrow \mathcal{R}$, be defined by $r(x,y)=(\sigma_{x}(y),\gamma_{y}(x))$, with  $\mathcal{D}=\mathcal{R}=\{(x_0,x_2), (x_1,x_2),  (x_2,x_0), (x_2,x_1),(x_0,x_0),(x_1,x_1),(x_2,x_2)\}$.  The functions $\sigma_{x}$ and $\gamma_{x}$ are described  like permutations, with  a specification of their domain of definition. Let  $\sigma_0=\gamma_0=(0)(2)$, $\mathcal{D}_{\sigma_0}=\mathcal{D}_{\gamma_0}=\{0,2\}$; 
 	$\sigma_1=\gamma_1=(1)(2)$, $\mathcal{D}_{\sigma_1}=\mathcal{D}_{\gamma_1}=\{1,2\}$;  $\sigma_2=\gamma_2=(0,1)(2)$, $\mathcal{D}_{\sigma_2}=\mathcal{D}_{\gamma_2}=\{0,1,2\}=X$. 
 The functions	$\sigma_0,\gamma_0,\sigma_1,\gamma_1$ are partial bijections of $X$ and $\sigma_2,\gamma_2$ are bijections of $X$. A technical computation shows that $(X,r)$ is a square-free partial solution, with $r(x_0,x_2) =(x_2,x_1)$, $r(x_2,x_1)=(x_0,x_2)$, $r (x_1,x_2)=(x_2,x_0)$, and $r (x_2,x_0)= (x_1,x_2)$.
 \end{ex}
 Lemma \ref{lem-formules-invol+braided} can be directly extended to partial solutions in the following way:
 \begin{lem}\label{lem-formules-invol+braided-partial}
 	Let $(x,y),(y,z)\in \mathcal{D}$, that is $x \in \mathcal{D}_{\gamma_y}$, $y \in  \mathcal{D}_{\sigma_x}\cap \mathcal{D}_{\gamma_z}$, $z \in  \mathcal{D}_{\sigma_y}$.
 	\begin{enumerate}[(i)]
 	
 			\item  $(X,r)$ is involutive  if  and only if  $\gamma_y (x)\in  \mathcal{D}_{\sigma_{\sigma_x(y)}}$, $\sigma_x(y)\in  \mathcal{D}_{\gamma_{\gamma_y(x)}}$      and  additionally 
 			\begin{gather}
 			 \sigma_{\sigma_x(y)}\gamma_{y}(x)=x  \label{eq-inv1}\\
 			 	\gamma_{\gamma_y(x)}\sigma_x(y)=y  \label{eq-inv2}
 			\end{gather}
 				\item 	If $(X,r)$ is  involutive, then  $(x,y)\in \mathcal{D}$  if and only if  $y \in  \mathcal{D}_{\sigma_x}$ and $x \in \mathcal{D}_{\sigma^{-1}_{\sigma_x(y)}}$, and 	in this case $\sigma_x(y)=\gamma^{-1}_{\gamma_y(x)}$  and $\gamma_y(x)=\sigma^{-1}_{\sigma_x(y)}(x)$.
 		
 	\item   $(X,r)$ is  braided if  and only if  whenever 	$x \in  \mathcal{D}_{\gamma_z\gamma_y}$, 
 		$z \in  \mathcal{D}_{\sigma_x\sigma_y}$,  	$y \in \mathcal{D}_{\sigma_{\gamma_{\sigma_y(z)}(x)}\gamma_z}$,  the following equations hold:\\
 	 \begin{gather}
 	\sigma_x\sigma_y=\sigma_{\sigma_x(y)}\sigma_{\gamma_y(x)}  \label{eqn-sigma-par}\\
 	 \gamma_z\gamma_y=\gamma_{\gamma_z(y)}\gamma_{\sigma_y(z)}   \label{eqn-gamma-par}\\
 	\gamma_{\sigma_{\gamma_y(x)}(z)}(\sigma_x(y))=\sigma_{\gamma_{\sigma_y(z)}(x)}(\gamma_z(y))
 	\end{gather} 
 	\end{enumerate}
 \end{lem}

 \begin{defn}
 	Let  $(X,r)$   be a  partial  set-theoretic  solution. 
 	The \emph{structure group} of $(X,r)$ is 	$G(X,r)=\operatorname{Gp} \langle X\mid\ xy =\sigma_x(y)\gamma_y(x)\ ;\ (x,y)\in \mathcal{D} \rangle$.  
 		The \emph{structure inverse monoid} of $(X,r)$ is 	$\operatorname{IM}(X,r)=\operatorname{Inv} \langle X\mid\ xy =\sigma_x(y)\gamma_y(x)\ ;\ (x,y)\in \mathcal{D} \rangle$.  
 \end{defn}

A partial solution $(X,r)$ is \emph{trivial} if  for every  $x \in X$, $\sigma_{x}=\operatorname{Id}_{\mathcal{D}_{\sigma_x}}$, $\gamma_{x}=\operatorname{Id}_{\mathcal{D}_{\gamma_x}}$. 
 So, for all pairs  $(x,y)\in \mathcal{D}$,   $r(x,y)=(y,x)$, that is 
 the  structure group of  a  trivial partial solution is  a partially commutative group  (or a  right-angled Artin group)  with generating set  $X$, and defining relations  that depend on $\mathcal{D}_{\sigma_x}$ and  $\mathcal{D}_{\gamma_x}$. 
 \begin{ex}\label{ex-trivial-partial-sol}
 	Let $X=\{x_0,x_1,x_2\}$.   Let $r: \mathcal{D} \rightarrow \mathcal{R}$, be defined by $r(x,y)=(\sigma_{x}(y),\gamma_{y}(x))$, with  $\mathcal{D}=\mathcal{R}=\{(x_0,x_2), (x_1,x_2),  (x_2,x_0), (x_2,x_1),(x_0,x_0),(x_1,x_1),(x_2,x_2)\}$.  Let  $\sigma_0=\gamma_0=(0)(2)$, $\mathcal{D}_{\sigma_0}=\mathcal{D}_{\gamma_0}=\{0,2\}$; 
 	$\sigma_1=\gamma_1=(1)(2)$, $\mathcal{D}_{\sigma_1}=\mathcal{D}_{\gamma_1}=\{1,2\}$;  $\sigma_2=\gamma_2=(0)(1)(2)$, $\mathcal{D}_{\sigma_2}=\mathcal{D}_{\gamma_2}=\{0,1,2\}$. So,  $r(x_0,x_2) =(x_2,x_0)$, $r(x_2,x_0)=(x_0,x_2)$, $r (x_1,x_2)=(x_2,x_1)$, $r (x_2,x_1)= (x_1,x_2)$, and $r (x_i,x_i)= (x_i,x_i)$, for $0 \leq i \leq 2$, that is  $(X,r)$ is a trivial  partial solution, with  structure group  the partially commutative group $\operatorname{Gp}\langle x_0,x_1,x_2 \mid x_0x_2 =x_2x_0, x_1x_2=
 	x_2x_1\rangle$.
 \end{ex}
 Two  partial set-theoretic solutions $(X,r)$ and $(X',r')$ are \emph{isomorphic} if there is a partial bijection $\alpha:X \rightarrow X'$ such that   if  $(x,y)\in \mathcal{D}$,  then  $(\alpha(x),\alpha(y))\in \mathcal{D}' $ and  $r(x,y)\in \mathcal{D}_{\alpha^2}$,  and additionally     $(\alpha \times \alpha) \circ r=r'\circ (\alpha \times \alpha)$.  If  $(X,r)$ and $(X',r')$ are isomorphic, then $G(X,r) \simeq G(X',r')$, with  $G(X,r)$ and $G(X',r')$ their respective structure groups. 
\begin{defn}
	Let  $(X,r)$ be a non-degenerate involutive partial set-theoretic solution. 
	\begin{enumerate}
		\item  A set $Y \subset X$ is  \emph{invariant} if   $r(Y \times Y)\subseteq Y \times Y$, whenever $r(Y \times Y)$ is defined.
		\item An invariant  subset $Y \subset X$  is  \emph{non-degenerate} if  $(Y,r\mid_{Y^2})$ is a non-degenerate involutive  partial set-theoretic solution of the YBE. 
		\item  $(X,r)$  is \emph{decomposable} if it is a union of two non-empty disjoint non-degenerate invariant subsets. Otherwise, it is called \emph{indecomposable}.
	\end{enumerate}
	
\end{defn}
	The \emph{retract} relation $\sim$ on the set $X$ is defined  by $x \sim y$ if  $\mathcal{D}_{\sigma_x}= \mathcal{D}_{\sigma_y}$ and  $\sigma_x=\sigma_y$. There is a natural induced solution $Ret(X,r)=(X/\sim,r)$, called the \emph{the retraction of $(X,r)$}, defined by $r'([x],[y])=([\sigma_{x}(y)],[\gamma_y(x)])$.
A non-degenerate involutive partial set-theoretic solution  $(X,r)$   is called \emph{a multipermutation partial solution of level $m$} if $m$ is the smallest natural number such that the solution $\mid Ret^m(X,r)\mid=1$, where 
$Ret^k(X,r)=Ret(Ret^{k-1}(X,r))$, for $k>1$. If such an $m$ exists, $(X,r)$   is also called \emph{retractable}, otherwise it is called \emph{irretractable}.

\subsection{Characterization of square-free  partial solutions and proof of Theorem $1$}
An important characterisation of  non-degenerate involutive set-theoretic solutions  is  presented in Proposition  \ref{prop-etingof} (\cite[p.176-180]{etingof}). We extend some of the properties of the classical set-theoretic solutions to the partial ones.

From now on, we use the following notation.	Let  $(X,r)$ be a non-degenerate involutive partial  set-theoretic solution, defined by  $r(x,y)=(\sigma_{x}(y),\gamma_{y}(x))$,   $(x,y) \in \mathcal{D}$,  with structure inverse monoid $\operatorname{IM}(X,r)$. We assume that for every $x \in X$, $x\in \mathcal{R}_{\sigma_x}$. This property always holds for square-free solutions. Let   $\operatorname{I}_X$ denote the symmetric inverse  monoid of $X$.  \\

Let $A$ denote the commutative inverse monoid, that is $A$ is  the  set of partial functions $f: X \rightarrow  \mathbb{Z}$ with finite support, with the following operation: 
for any two elements $f,f' \in A$, the  domain of the sum $f+f'$ is $\mathcal{D}_{f}\cap \mathcal{D}_{f'}$ and the operation in $A$ is defined pointwise, that is  $(f+f')(x)=f(x)+f'(x)$, for every $x \in \mathcal{D}_{f}\cap \mathcal{D}_{f'}$.

 \begin{lem}\label{lem-IX-acts-A}
	The symmetric inverse monoid   $ \operatorname{I}_X$ acts  (totally) on $A$ by endomorphisms:  \[\tau \bullet  f= f \circ \tau^{-1}\]
	  where $\tau \in \operatorname{I}_X$ is a partial bijection of $X$ and   $f \in A$, 	$f:X\rightarrow\mathbb{Z}$ is a partial function with finite support.	It is described  by the following diagram,
 		\begin{figure}[H] 
 		\begin{tikzpicture}
 			\matrix (m) [matrix of math nodes,row sep=3em,column sep=4em,minimum width=2em]
 			{ X & X \\
 				& \mathbb{Z}\\};
 			\path[-stealth]
 		
 			(m-1-2) edge node [right] {$f$} (m-2-2)
 			(m-1-1) edge node [above] {$\tau^{-1}$} (m-1-2)
 			(m-1-1) edge [double,red]node [left] {} (m-2-2)	;
 			\end{tikzpicture}
 			
 		\end{figure}
 \end{lem}
\begin{proof}
Let  $\varphi$ be the map from $\operatorname{I}_X$ to the monoid of endomorphisms of $A$, defined by \\$\varphi(\tau)(f) = f \circ \tau^{-1}$, $f\in A$, $\tau \in \operatorname{I}_X$. First, we show that  $\varphi(\tau)$ is indeed an endomorphism of $A$. Let $f,f'\in A$. Then  $\varphi(\tau)( f+f')= (f +f')\circ \tau^{-1}=f \circ \tau^{-1}+f '\circ \tau^{-1}=\varphi(\tau)(f)+\varphi(\tau)(f')$.  We show that $\varphi$ is a  homomorphism of monoids. Let $\tau,\nu \in \operatorname{I}_X$,  then \\$\varphi(\tau\circ\nu)(f) = f \circ (\tau \circ \nu)^{-1}=  (f \circ \nu^{-1} )\circ \tau^{-1}=\varphi(\tau)(f \circ \nu^{-1}) = \varphi(\tau)\varphi(\nu)(f )$.  Additionally,  $\operatorname{Id}_X$ acts trivially on every $f \in A$, that is $f \circ  \operatorname{Id}_{X}^{-1}=f$, and for every  $\tau \in \operatorname{I}_X$, $\tau \bullet 0_X=0_X$, where $0_X\in A$ is the zero function.  So, from Definition \ref{defn-action}, $\operatorname{I}_X$ acts   on $A$ by endomorphisms.
\end{proof}

	Note that   $\varphi(\tau)$ is not necessarily an automorphism of $A$, since, for any two elements $f,f' \in A$, the equality of the maps $(\varphi(\tau))(f)=(\varphi(\tau))(f')$ does not necessarily imply $f=f'$. Indeed, if  $f \circ \tau^{-1}= f' \circ \tau^{-1}$, then by precomposing each function with  $\tau$, we have $f \circ  \operatorname{Id}_{\mathcal{D}_\tau}=f' \circ  \operatorname{Id}_{\mathcal{D}_\tau}$,  that is these two partial functions have the same domain $\mathcal{D}_\tau \cap\mathcal{D}_f\,=\,\mathcal{D}_\tau \cap \mathcal{D}_{f'} $ and 
		$f(x)=f'(x)$, for every $x \in \mathcal{D}_\tau \cap\mathcal{D}_f$. But this does not imply  necessarily  that  $\mathcal{D}_f=\mathcal{D}_{f'}$, nor $f=f'$.

We show the structure  inverse monoid $\operatorname{IM}(X,r)$ acts on  $X$ by partial permutations and that this action extends to an action of $\operatorname{IM}(X,r)$  on  $A$.
\begin{lem}\label{lem-existence-epsilon}
	\begin{enumerate}[(i)]
		\item  The map $\alpha: \operatorname{IM}(X,r) \rightarrow \operatorname{I}_X$, defined by  $\alpha(x)= \sigma_{x}$,  $x \in X$, is a  homomorphism of monoids. 
		\item There is an action of $\operatorname{IM}(X,r)$ on  $X$ by partial permutations, defined by:
	\begin{align*}
	x_i\bullet x_j= x_{\sigma_i(j)}\\
	g\bullet x_j= x_{\sigma_g(j)}
	\end{align*} 
		 where $x_i,g \in \operatorname{IM}(X,r) $, $x_j \in X$, $\sigma_g=\alpha(g)$.
		 \item There is an action of $\operatorname{IM}(X,r)$ on  itself  by endomorphisms, defined  by:
		 	\begin{gather*}
		 g\bullet x_j= x_{\sigma_g(j)}\\
		 g\bullet h= g\bullet x_{j_1}...x_{j_k}=x_{\sigma_g(j_1)}...x_{\sigma_g(j_k)}
		 \end{gather*} 
		
		 where $g ,h\in \operatorname{IM}(X,r) $.
		\item There is an action of $\operatorname{IM}(X,r)$ on  $A$ by endomorphisms, defined  by:
	\[g\bullet f= \alpha (g) \bullet f=f \circ \sigma_g^{-1}\]
where $g \in \operatorname{IM}(X,r) $, $f\in A$

	\end{enumerate}
	\end{lem}

 \begin{proof}
 	$(i),(ii),(iii)$ From the definition of  the structure  inverse monoid,  the defining relations in $\operatorname{IM}(X,r)$, have the form $xy=\sigma_x(y)\gamma_y(x)$, $(x,y) \in \mathcal{D}$. So, from Equation  \ref{eqn-alpha-hom},  $\alpha: \operatorname{IM}(X,r) \rightarrow \operatorname{I}_X$   is a  homomorphism of monoids. That is, $G(X,r)$ acts on  $X$ by partial permutations, and this action extends to an  action by endomorphisms on itself.\\
$(iv)$	Let $g \in \operatorname{IM}(X,r)$. Then $g$ has a  homomorphic image 	$\alpha(g)=\sigma_g \in \operatorname{I}_{X}$, and there is an action of $\operatorname{I}_X$  on $A$ by endomorphisms,  from Lemma \ref{lem-IX-acts-A}. So,  $\operatorname{IM}(X,r)$ acts  on  $A$ by endomorphisms.
	\end{proof}
Note that, since $\alpha$ is a  homomorphism of monoids,  for every $x \in X$, $\alpha(x_i^{*})= \sigma_{i}^{-1}$,  and $\alpha(x_ix_i^{*})= \sigma_i\sigma_{i}^{-1}= Id_{\mathcal{R}_{\sigma_i}}$,  $\alpha(x_i^{*}x_i)= \sigma_{i}^{-1}\sigma_i=\operatorname{Id}_{\mathcal{D}_{\sigma_i}}$.

\begin{defn}
	For every  $x \in X$, we define a partial function  with finite support, 
$ \delta_x$,   such that 
$ \mathcal{D}_{\delta_x}=\mathcal{R}_{\sigma_x}\subseteq X$  and $\delta_x:\mathcal{D}_{\delta_x}\rightarrow\mathbb{Z} $ is defined by:
\begin{gather*}
\delta_x(y)= \left\{
\begin{array}{ll}
1 & y=x\\
0 & y\in\mathcal{R}_{\sigma_x}, \,\,  y\neq x \\
\end{array} 
\right.
\end{gather*}
Furthermore, $\delta_x(y)$ is not defined for  $y \in X \setminus \mathcal{R}_{\sigma_x}$.
\end{defn}
In what follows, we give a characterization of  square-free partial solutions, in analogy with the characterization of classical solutions  given in Proposition \ref{prop-etingof}. 
\begin{thm}	\label{theo-pi-injective}
	Let  $(X,r)$ be a square-free, non-degenerate involutive partial  set-theoretic solution, defined by  $r(x,y)=(\sigma_{x}(y),\gamma_{y}(x))$,   $(x,y) \in \mathcal{D}$, with structure inverse monoid $\operatorname{IM}(X,r)$.  Let  $\pi$  be the map defined by:
	\begin{gather*}
	\pi:\operatorname{IM}(X,r)\rightarrow    A\\
	\pi(x_i)= \delta_i  \;\;;\;\; \pi(x_i^*)=- \delta_i\circ \sigma_i  \\
\pi(gh)	= \pi(g)+g \bullet \pi(h) \;\;;\pi(g^*)=-\pi(g)\circ \sigma_g
	\end{gather*}
		Then $\pi$  is an injective map.
\end{thm}
For convenience, we divide the proof   of Theorem \ref{theo-pi-injective} into several lemmas. First, we show that $\pi$ is well-defined.
\begin{lem}\label{lem-pi-well-def1}
	Let $g,h \in \operatorname{IM}(X,r)$. Then
	\begin{enumerate}[(i)]
		\item  $\pi(gg^*g)=\pi(g)$  in $A$.
		\item  $\pi(gg^*hh^*)=\pi(hh^*gg^*)$ in $A$.
	\end{enumerate}
\end{lem}
\begin{proof}
$(i),(ii)$ Let $x_i \in X$. We show that $\pi(x_ix_i^*x_i)=\pi(x_i)=\delta_i$. From the definition of $\pi$, 
$\pi(x_i^*x_i)=\pi(x_i^*)+x_i^* \bullet \pi(x_i)=-\delta_i\circ \sigma_i +\delta_i\circ \sigma_i =0_{\mathcal{D}_{\sigma_i}}$, and $\pi(x_ix_i^*x_i)=\pi(x_i)+x_i \bullet \pi(x_i^*x_i)=\delta_i +0_{\mathcal{D}_{\sigma_i}} \circ\sigma_i ^{-1}=\delta_i +0_{\mathcal{R}_{\sigma_i}} =\delta_i$. In the same way, we  show $\pi(x_i^*x_ix_j^*x_j)=\pi(x_j^*x_jx_i^*x_i)=0_{\mathcal{D}_{\sigma_i}\cap \mathcal{D}_{\sigma_j}} $, and for any $g,h \in \operatorname{IM}(X,r)$, $\pi(gg^*g)=\pi(g)$, and $\pi(gg^*hh^*)=\pi(hh^*gg^*)$.
\end{proof}
	Let $(x_i,x_j)\in \mathcal{D}$,  $i\neq j$. So,  $x_ix_j=x_{\sigma_i(j)}x_{\gamma_j(i)}$ is a defining relation in  $\operatorname{IM}(X,r)$.  We  show  that  $\pi(x_ix_j)=\pi(x_{\sigma_i(j)}x_{\gamma_j(i)})$ in $A$.  From the definition of $\pi$:\\
	\begin{center}$	\left\{
	\begin{array}{lll}
	\pi(x_ix_j)&=&\delta_i+\delta_j \circ \sigma_i^{-1}\\
	\pi(x_{\sigma_i(j)}x_{\gamma_j(i)})&=&\delta_{\sigma_i(j)}+\delta_{\gamma_j(i)}\circ \sigma^{-1}_{\sigma_i(j)}\\
	\end{array}
	\right.$\end{center}

\begin{lem}\label{lem-pi-well-def2}
	Let denote   $f'=\delta_{\gamma_j(i)}\circ \sigma^{-1}_{\sigma_i(j)}$ and   $f=\delta_{\sigma_i(j)}+f'$. Then $\mathcal{D}_{\delta_i+\delta_j \circ \sigma_i^{-1}}= \mathcal{D}_{f}$.
\end{lem}
\begin{proof}
We compute both $\mathcal{D}_{\delta_i+\delta_j \circ \sigma_i^{-1}}$ and $\mathcal{D}_{f}$:\\

 $\left\{
\begin{array}{lll}
	\mathcal{D}_{\delta_i+\delta_j \circ \sigma_i^{-1}}=\mathcal{D}_{\delta_i}\cap \mathcal{D}_{\delta_j \circ \sigma_i^{-1}}= \mathcal{R}_{\sigma_i}\cap \sigma_i(\mathcal{D}_{\sigma_i} \cap\mathcal{D}_{\delta_j})=\mathcal{R}_{\sigma_i}\cap \sigma_i(\mathcal{D}_{\delta_j})=\mathcal{R}_{\sigma_i}\cap \sigma_i(\mathcal{R}_{\sigma_j})=\mathcal{R}_{\sigma_i\sigma_j}\\
	\mathcal{D}_{f}= \mathcal{D}_{\delta_{\sigma_i(j)}}\cap \mathcal{D}_{f'}=\mathcal{D}_{\delta_{\sigma_i(j)}}\cap\sigma_{\sigma_i(j)}(\mathcal{D}_{\sigma_{\sigma_i(j)}} \cap\mathcal{D}_{\delta_{\gamma_j(i)}})
	=\mathcal{R}_{\sigma_{\sigma_i(j)}}\cap \sigma_{\sigma_i(j)}(\mathcal{R}_{\sigma_{\gamma_j(i)}})
	= \mathcal{R}_{\sigma_{\sigma_i(j)}\sigma_{\gamma_j(i)}}
	\end{array}
\right.$
From Lemma \ref{lem-formules-invol+braided-partial}$(iii)$,  $\sigma_i\sigma_j=\sigma_{\sigma_i(j)}\sigma_{\gamma_j(i)}$, so $\mathcal{D}_{\delta_i+\delta_j \circ \sigma_i^{-1}}=\mathcal{D}_{f}$.
\end{proof}
	
	\begin{lem}\label{lem-pi-well-def3}
The partial maps $\delta_i+\delta_j \circ \sigma_i^{-1}$ and $f$ are equal. 
	\end{lem}
	\begin{proof}
 For all the values in  $\mathcal{D}_{f}$, except maybe for $i$, $\sigma_i(j)$ and  $\sigma_{\sigma_i(j)}\gamma_j(i)$ (assuming they belong to $\mathcal{D}_{f}$), both functions are zero. As, from Equation \ref{eq-inv1},  $\sigma_{\sigma_i(j)}\gamma_j(i)=i$,  we need to check the value of the partial maps $\delta_i+\delta_j \circ \sigma_i^{-1}$ and $f$  at two  values:  $i$  and $\sigma_i(j)$.\\
   For $i \in 	\mathcal{D}_{\delta_i+\delta_j \circ \sigma_i^{-1}}$.   As the partial solution is square-free,   $\sigma_i^{-1}(i)=i $, $\sigma_i(j)\neq i$,  $i\neq j$:\\
 $\left\{
 \begin{array}{lll}
  (\delta_i+\delta_j \circ \sigma_i^{-1})(i)=\delta_i(i)+\delta_j \circ \sigma_i^{-1}(i)=1+\delta_j (i)=1+0=1\\
 f(i)=(\delta_{\sigma_i(j)}+\delta_{\gamma_j(i)}\circ \sigma^{-1}_{\sigma_i(j)})(i)=\delta_{\sigma_i(j)}(i)+\delta_{\gamma_j(i)}\circ \sigma^{-1}_{\sigma_i(j)}(i)=0+\delta_{\gamma_j(i)}(\sigma^{-1}_{\sigma_i(j)}(i))=1
 \end{array}
 \right.$
 
	The last step comes from Lemma \ref{lem-formules-invol+braided-partial}$(ii)$. Indeed, 
$\sigma^{-1}_{\sigma_i(j)}(i)=\gamma_j(i)$, so  $f(i)= \delta_{\gamma_j(i)}(\gamma_j(i))=1$, that is   $(\delta_i+\delta_j \circ \sigma_i^{-1})(i)=f(i)=1$. \\
For $\sigma_i(j) \in 	\mathcal{D}_{\delta_i+\delta_j \circ \sigma_i^{-1}}$. As the partial solution is square-free,   $\sigma^{-1}_{\sigma_i(j)}(\sigma_i(j))=\sigma_i(j)$,  for $i\neq j$,  $\sigma_i(j)\neq i$, $\gamma_j(i)\neq \sigma_i(j)$, and:\\
 $\left\{
\begin{array}{lll}(\delta_i+\delta_j \circ \sigma_i^{-1})(\sigma_i(j))=
\delta_i(\sigma_i(j))+\delta_j \circ \sigma_i^{-1}(\sigma_i(j) )=0+\delta_j (j)=0+1=1\\
f(\sigma_i(j))=(\delta_{\sigma_i(j)}+\delta_{\gamma_j(i)}\circ \sigma^{-1}_{\sigma_i(j)})(\sigma_i(j))=
\delta_{\sigma_i(j)}(\sigma_i(j))+\delta_{\gamma_j(i)}(\sigma_i(j))=1+0=1\\
\end{array}
\right.$

So,   $f(\sigma_i(j) )=(\delta_i+\delta_j \circ \sigma_i^{-1})(\sigma_i(j))=1$, and  the partial maps $\delta_i+\delta_j \circ \sigma_i^{-1}$ and $f$ are equal, that is  $\pi(x_ix_j)=\pi(x_{\sigma_i(j)}x_{\gamma_j(i)})$ in $A$. 
\end{proof}
\begin{lem}\label{lem-pi-inj}
	 $\pi$ is an injective map. 
\end{lem}
\begin{proof}
We define a map $\omega: A \rightarrow \operatorname{IM}(X,r)$ in the following way:\\
\begin{gather*}
\begin{array}{llll}
&\omega(\delta_i)=x_i\ & &\omega(\delta_i\mid_{\mathcal{D}})=x_i, \;\; \mathcal{D}\subseteq \mathcal{D}_{\delta_i}\\
&\omega(-\delta_i \circ \sigma_i)=x_i^* & &\omega((-\delta_i \circ \sigma_i)\mid_{\mathcal{D}})=x_i^*, \;\;\;\mathcal{D}\subseteq \mathcal{D}_{-\delta_i \circ \sigma_i}\\
\end{array}\\
\omega(f+f')=\omega(f)\;\;\omega(\,(\omega(f))^{*}\,\bullet f') \;\;,\;\;f,f' \in A
\end{gather*}\\
 We show that $\omega\circ \pi=Id_{ \operatorname{IM}(X,r)}$. Let $g=x_{j_1}^{\epsilon_1}x_{j_2}^{\epsilon_2}...x_{j_k} ^{\epsilon_k}\,\in \operatorname{IM}(X,r)$, where for every $1 \leq i \leq k$, $x_{j_i}\in X$,  $\epsilon_i =\pm1$ (we replace here the $*$ by $-1$).  Let $g_i=x_{j_1}^{\epsilon_1}x_{j_2}^{\epsilon_2}...x_{j_i} ^{\epsilon_i}$,  $1 \leq i \leq k$, and $g_0=1$. So,  $\sigma_{g_{i-1}}= \sigma_{j_{1}}^{\epsilon_1}\sigma_{j_{2}}^{\epsilon_2}...\sigma_{j_{i-1}}^{\epsilon_{i-1}}$ and $\sigma_{g_0}=Id_X$. Let $\mu_i=\epsilon_i-1 \,(\equiv mod3)$. We show by induction on $k$ that $\omega\circ \pi(g)=g$, where  $\pi(g)=\sum\limits_{i=1}^{i=k}\epsilon_i\,  (\delta_{j_i}\circ \sigma_{j_i}^{\mu_i})\circ \sigma_{g_{i-1}}^{-1}$. \\
	For $k=1$, it holds from the definition of $\pi$ and $\omega$. 
 Assume  $\omega\circ \pi(g')=g'$, for every $g'\in \operatorname{IM}(X,r)$ of  the form  $x_{j_1}^{\epsilon_1}x_{j_2}^{\epsilon_2}...x_{j_{k-1}} ^{\epsilon_{k-1}}$,  $\epsilon_i =\pm1$, and $g=g'x_{j_k}^{\epsilon_k}$. Let $\delta=\pi(x_{j_k}^{\epsilon_k})$. So, $\omega\circ \pi(g)=\omega(\pi(g')+g'\bullet \delta)=\omega(\pi(g'))\,\omega((\omega(\pi(g')))^* \, \bullet g' \bullet \delta)$. From the induction assumption, we have $\omega\circ \pi(g)=g'\,\omega(g'^*\bullet g'\bullet \delta)=g'\,\omega(\delta\circ \sigma_{g'}^{-1}\circ \sigma_{g'})=g'\,\omega(\delta\circ \operatorname{Id}_{\mathcal{D}_{\sigma_{g'}}})=g'x_{j_k}^{\epsilon_k}=g$, since $\delta\circ \operatorname{Id}_{\mathcal{D}_{\sigma_{g'}}}$ is equal $\delta\mid_{\mathcal{D}}$, with $D=\mathcal{D}_{\sigma_{g'}}\cap \mathcal{D}_{\delta}$.
\end{proof}

Note that the map $\pi$ is injective, but certainly not surjective. Indeed, the function $\omega$, as defined, is not injective. Finally, we prove Theorem \ref{theo-pi-injective}.
\begin{proof}[Proof of Theorem \ref{theo-pi-injective}]
From Lemmas \ref{lem-pi-well-def1}, \ref{lem-pi-well-def2}, \ref{lem-pi-well-def3},  $\pi$ is well-defined and from Lemma \ref{lem-pi-inj},  $\pi$ is injective.
\end{proof} 
We show that $A  \Join \operatorname{I}_X$,  the restricted product of $A$ and  $\operatorname{I}_X$,  can be defined, and is an inverse monoid. Furthermore, we prove that, if $(X,r) $ is square-free, $\operatorname{IM}(X,r)$  embeds in the inverse monoid  $A  \Join \operatorname{I}_X$.
\begin{lem}\label{lem-epsilon}
Let $E(I_X)$ denote the set of idempotents of $I_X$. Let $\epsilon$ be defined by:
\begin{gather*}
\epsilon: A \rightarrow E(I_X)\\
\epsilon(f)=\;\operatorname{Id}_{\mathcal{D}_f}
\end{gather*}
 Then $\epsilon$ is  a surjective homomorphism of monoids  and it satisfies the following condition: for each $f\in A$, there exists $\epsilon(f) \in E(I_X)$ such that  $\epsilon(f)\leq \operatorname{Id}_{\chi} \Longleftrightarrow \operatorname{Id}_{\chi}\bullet f=f $,  for all  $\chi \subseteq X$.
\end{lem}
\begin{proof}
Clearly, the map $\epsilon$ is surjective, since the idempotents of $I_X$ are precisely the partial identities of $X$, and for each subset $\chi $ of $X$, there is $f\in A$ such that $\chi=\mathcal{D}_f$.\\
We show that $\epsilon$ is  a homomorphism of monoids. Let $f,f' \in A$. So, 
\[\epsilon(f+f')=
\operatorname{Id}_{\mathcal{D}_{f+f'}}=\operatorname{Id}_{\mathcal{D}_{f}\cap \mathcal{D}_{f'}}=\operatorname{Id}_{\mathcal{D}_{f}}\circ \operatorname{Id}_{\mathcal{D}_{f'}}=\epsilon(f)\circ \epsilon(f')\]
Let  $\chi \subseteq X$. We show  that  $\epsilon(f)\leq \operatorname{Id}_{\chi} \Longleftrightarrow \operatorname{Id}_{\chi}\bullet f=f $. From the definition of the action  of $\operatorname{I}_X$ on $A$, $\operatorname{Id}_{\chi}\bullet f=f \circ \operatorname{Id}_{\chi}^{-1}$ , so $\operatorname{Id}_{\chi}\bullet f=f $  $\Longleftrightarrow$ $f \circ \operatorname{Id}_{\chi}^{-1}=f$ $\Longleftrightarrow $  $f \circ \operatorname{Id}_{\chi}=f$, which occurs if  and only if 
$\mathcal{D}_f \subseteq \chi$. On the other hand, 
$\epsilon(f)\leq \operatorname{Id}_{\chi}$ $\Longleftrightarrow $ $\operatorname{Id}_{\mathcal{D}_f}\leq \operatorname{Id}_{\chi}$ $\Longleftrightarrow$ $\operatorname{Id}_{\mathcal{D}_f}=\operatorname{Id}_{\mathcal{D}_f}\circ  \operatorname{Id}_{\chi}= \operatorname{Id}_{\chi}\circ \operatorname{Id}_{\mathcal{D}_f}$, and this  occurs if  and only if 
$\mathcal{D}_f \subseteq \chi$.
\end{proof}

\begin{thm}	\label{theo-join-inverse-monoid}
	Let  $\mathscr{I}$ be the following set and with the following operation
	\begin{gather*}
\mathscr{I}=\{(f,\tau) \in A \times \operatorname{I}_X \mid  \mathcal{R}_{\tau}= \mathcal{D}_f \}\\
(f,\tau)(f',\nu)=(f+(\tau \bullet f')\,,\,\tau\nu)
	\end{gather*}
Then $\mathscr{I}$ is the restricted product $A  \Join \operatorname{I}_X$. Furthermore,  $\mathscr{I}$  is  an inverse monoid.
\end{thm}
\begin{proof}
	From Lemma \ref{lem-IX-acts-A}, $\operatorname{I}_X$ acts on $A$ by endomorphisms and from  Lemma \ref{lem-epsilon}, this action satisfies additional conditions that ensure the existence of the restricted product $A  \Join \operatorname{I}_X$, so  $A  \Join \operatorname{I}_X$ exists and is by definition 
	$A  \Join \operatorname{I}_X=\{(f,\tau) \in A \times \operatorname{I}_X \mid  r(\tau)=\epsilon(f)\}$ (see Defn. \ref{defn-restricted}).  Since,  $r(\tau)=\tau\tau^*=\operatorname{Id}_{\mathcal{R}_{\tau}}$	 and
	$\epsilon(f)=\operatorname{Id}_{\mathcal{D}_f} $,  the condition 
	$r(\tau)=\epsilon(f)$ is equivalent to $\mathcal{R}_{\tau}= \mathcal{D}_f$. That is, 
$\mathscr{I}=$ $\{(f,\tau) \in A \times \operatorname{I}_X \mid \mathcal{R}_{\tau}= \mathcal{D}_f\}=A  \Join \operatorname{I}_X$,  an inverse monoid.
\end{proof}

\begin{thm}\label{theo_psi-hom}
	Let  $(X,r)$ be a square-free, non-degenerate involutive partial  set-theoretic solution,  defined by  $r(x,y)=(\sigma_{x}(y),\gamma_{y}(x))$,   $(x,y) \in \mathcal{D}$,  with structure  inverse monoid $\operatorname{IM}(X,r)$. Then the map 
	\begin{align*}
	\psi: \operatorname{IM}(X,r) \rightarrow A  \Join \operatorname{I}_X\\
	\psi(x)=(\delta_x,\sigma_x)\\
	\psi(g)=(\pi(g), \sigma_g)
		\end{align*}
	is an injective  homomorphism of monoids. Furthermore, $\operatorname{Im}(\psi)$, the image of  $\operatorname{IM}(X,r)$ in $A  \Join \operatorname{I}_X$,  is an inverse monoid. 
\end{thm}
\begin{proof}
Let $g \in \operatorname{IM}(X,r)$, with 	$\psi(g)=(\pi(g), \sigma_g)$.  First, we show that $(\pi(g), \sigma_g)$ belongs indeed to  $A \Join \operatorname{I}_X$, that is  $\mathcal{D}_{\pi(g)}=\mathcal{R}_{\sigma_g}$ is satisfied. 	Let $g=x_{j_1}^{\epsilon_1}x_{j_2}^{\epsilon_2}...x_{j_k} ^{\epsilon_k}\,\in \operatorname{IM}(X,r)$, where for every $1 \leq i \leq k$, $x_{j_i}\in X$,  $\epsilon_i =\pm1$ (here the $*$ is $-1$).  The proof is  by induction on $k$. \\
	For $k=1$, if $g=x_{j_1}$, then 	$\psi(x_{j_1})=(\delta_{j_1},\sigma_{j_1})$ and   $\mathcal{D}_{\delta_{j_1}}=\mathcal{R}_{\sigma_{j_1}}$, from the definition of 
	$\delta_{j_1}$.  If  $g=x_{j_1}^{-1}$,  then 	$\psi(x_{j_1}^{-1})=(-\delta_{j_1}\circ \sigma_{j_1},\sigma_{j_1}^{-1})$ and we have 
	$\mathcal{R}_{	\sigma_{j_1}^{-1}}=\mathcal{D}_{\sigma_{j_1}}$ and $\mathcal{D}_{\delta_{j_1}\circ \sigma_{j_1}}=\mathcal{D}_{ \sigma_{j_1}}$ also,  since $\mathcal{R}_{\sigma_{j_1}}=\mathcal{D}_{\delta_{j_1}}$. For $k>1$, assume $\mathcal{D}_{\pi(g')}=\mathcal{R}_{\sigma_{g'}}$, where  $g'=x_{j_1}^{\epsilon_1}x_{j_2}^{\epsilon_2}...x_{j_{k-1}} ^{\epsilon_{k-1}}$, $g=g'\,x_{j_k} ^{\epsilon_k}$ and  $\delta=\pi(x_{j_k}^{\epsilon_k})$. So, $ \pi(g)=\pi(g')+g'\bullet \delta$   and  from the  induction assumption,  $\mathcal{D}_{\pi(g)}=\mathcal{D}_{\pi(g')}\cap \mathcal{D}_{g'\bullet \delta}= \mathcal{R}_{\sigma_{g'}}\cap \mathcal{D}_{\delta\circ \sigma^{-1}_{g'}}$. \\
			If $\epsilon_k=1$, then $\delta=\delta_k$, and: \begin{gather*}
	\mathcal{D}_{\delta\circ \sigma_{g'}^{-1}}=\sigma_{g'}(\mathcal{D}_{ \sigma_{g'}^{}}\cap\mathcal{D}_{ \delta_{j_k}})
	=\sigma_{g'}(\mathcal{D}_{ \sigma_{g'}^{}}\cap\mathcal{R}_{ \sigma_{j_k}})\\
\Longrightarrow	\mathcal{D}_{\pi(g)}=\mathcal{R}_{\sigma_{g'}}\cap \sigma_{g'}(\mathcal{D}_{ \sigma_{g'}^{}}\cap\mathcal{R}_{ \sigma_{j_k}})=\mathcal{R}_{ \sigma_{g'}^{}\sigma_{j_k}}=\mathcal{R}_{ \sigma_{g}}
		\end{gather*}
\margin_comment{\textcolor{red}{cosmetic changes}}
		If $\epsilon_k=-1$, then $\delta=-\delta_{j_k}\circ \sigma_{j_k}$, and:
		\begin{gather*}
\mathcal{D}_{g'\bullet \delta}=\mathcal{D}_{\delta_{j_k}\circ \sigma_{j_k}\circ \sigma_{g'}^{-1}}=	\sigma_{g'}^{}\circ \sigma_{j_k}^{-1}(	\mathcal{R}_{\sigma_{j_k}}\cap \sigma_{j_k}(	\mathcal{D}_{\sigma_{g'}}))\\
=\sigma_{g'}(	\mathcal{D}_{\sigma_{j_k}}\cap 	\mathcal{D}_{\sigma_{g'}})=
\mathcal{R}_{\sigma_{g'}} \cap \sigma_{g'}(	\mathcal{D}_{\sigma_{j_k}})=
\mathcal{R}_{\sigma_{g'}} \cap \sigma_{g'}(	\mathcal{R}_{\sigma_{j_k}^{-1}})\\
\Longrightarrow \mathcal{D}_{\pi(g)}=	\mathcal{R}_{\sigma_{g'}} \cap \sigma_{g'}(	\mathcal{R}_{\sigma_{j_k}^{-1}})= \mathcal{R}_{\sigma_{g'}\sigma_{j_k}^{-1}} =\mathcal{R}_{\sigma_{g}}
	\end{gather*}

	Next, we show that $\psi$ is an injective homomorphism of monoids. Let $g,h \in \operatorname{IM}(X,r)$. Then $\psi(gh)=(\pi(gh), \sigma_{gh})=(\pi(g)+g \bullet \psi(h), \sigma_g\circ \sigma_h)=\psi(g)\psi(h)$, from the definition of the product in $ A  \Join \operatorname{I}_X$. The injectivity of $\psi$ results from the injectivity of $\pi$  (see Theorem \ref{theo-pi-injective}). As the homomorphic image of an inverse monoid is an inverse monoid (\cite[p.30]{lawson}), $\operatorname{Im}(\psi)$ is an inverse monoid.

\end{proof}
Note that from the definition of the inverse in an inverse semigroup, 	$\psi(g)^*\psi(g)\psi(g)^*=\psi(g)$, so  $\psi(g)^*= (-\pi(g)\circ \sigma_g,\;\sigma_g^{-1})$.

\section{Definition of partial  braces and their properties}
In \cite{partial-s}, the authors define \emph{a partial semigroup} to be a set  $S$ together with an operation $\oplus$ that maps a subset $\mathcal{D}\subset S \times S$ into $S$ and satisfies the associative law $(g \oplus h)\oplus k=g \oplus(h\oplus k)$, in the sense that if either side is defined then so is the other and they are equal. We define a \emph{a partial monoid} to be a partial semigroup, with an identity $1$ such that $g \oplus 1$, $1\oplus g$  are always defined, equal and equal to $g$. We say that a partial monoid is \emph{commutative}, if $g \oplus h=h \oplus g$, whenever both are defined.
\begin{defn}
	A \emph{partial left   brace}  is a set  $\mathcal{B}$ with two operations, $\oplus$ and $\cdot$, such that $(\mathcal{B},\oplus)$ is a commutative partial monoid, $(\mathcal{B},\cdot)$ is an inverse monoid  and for every $a,b,c \in \mathcal{B}$, such that  $b\oplus c$ and  $a \cdot b \oplus a \cdot c $ are defined, the following holds:
	\begin{equation}\label{defn-dist}
a \cdot (b\oplus c) = a \cdot b \oplus a \cdot c 
	\end{equation}
	$(\mathcal{B},\oplus)$  is called \emph{the partial monoid of the brace } and $(\mathcal{B},\cdot)$  is called \emph{the  multiplicative inverse monoid of the brace}.\\
	A \emph{partial  right brace} is defined similarly:  for every $a,b,c \in \mathcal{B}$,    $(**): (a\oplus b) \cdot c  = a \cdot c \oplus b \cdot c$.
\end{defn}
In what follows, we show that a  non-degenerate involutive partial set-theoretic solution, with structure inverse monoid $\operatorname{IM}(X,r)$,  induces  a partial left brace\\ $(\operatorname{IM}(X,r), \oplus,\cdot)$. We define a new partial operation $\oplus$ in $\operatorname{IM}(X,r)$, and in order to make the definition of $\oplus$  easier we use the method
of right reversing. P. Dehornoy introduced this tool in the context of braids and Garside groups, and we refer the reader to \cite{garside1}, \cite{dehornoy},  and in particular to \cite{dehornoy2} for a wider understanding of this topic. 
\subsection{A   very brief presentation of   the method of right-reversing}

 Roughly, reversing can be used as  a  tool for constructing van Kampen diagrams in the context of presented semigroups or monoids.  Let $M=\operatorname{Mon}\langle X\mid R\rangle$, with $R$ a 
 family of pairs of  nonempty words in the alphabet $X$, called relations. As is well-known, two strings  $w$ and $w'$  are equal in $M$ if and only if there exists an $R$-derivation from $w$ to $w'$, defined to be a finite sequence of words $(w_0,..., w_p)$ such that $w_0$ is $w$, $w_p$ is $w'$, and, for each $i$, $w_{i+1}$ is obtained from $w_i$ by substituting some subword that occurs in a relation of $R$  with the other element of that relation. A van Kampen diagram for a pair of words  $(w,w')$ is a planar oriented graph with a unique source vertex and a unique  sink  vertex and edges labeled by letters of  $X$ so that the labels of each face correspond to a relation of $R$ and the labels of the bounding paths form the words $w$ and $w'$,  respectively.
The strings $w$ and $w'$ are equal in $M$ if and only if there exists a  van Kampen diagram for $(w,w')$. It  is  convenient to standardize van Kampen  diagrams,  so that they only contain vertical and horizontal edges, plus dotted arcs connecting vertices that are to be identified. Such standardized diagrams are called \emph{reversing diagrams}.  Here we consider only right-reversing diagrams. We  illustrate in the following example the construction  of a  right-reversing diagram.
\begin{ex}\label{ex-reversing}
	We consider the solution  from Example \ref{exemple:exesolu_et_gars}, with structure monoid 
	$M=\operatorname{Mon} \langle X\mid  x_{1}x_{2}=x^{2}_{3}; x_{1}x_{3}=x_{2}x_{4};
	x_{2}x_{1}=x^{2}_{4}; x_{2}x_{3}=x_{3}x_{1};
	x_{1}x_{4}=x_{4}x_{2};x_{3}x_{2}=x_{4}x_{1} \rangle$. 
We illustrate with the following figure how to construct a reversing diagram, that represents a van kampen diagram for a  pair of words $(w,w')$, such that $x_1x_2$ is the prefix of $w$ and $x_2x_1$ is the prefix of $w'$. We begin with the left-most figure, with source the star, and two paths labelled  $x_1x_2$  and $x_2x_1$ respectively. We  complete the first left square with the defining relation $x_{1}x_{3}=x_{2}x_{4}$. At the next step, we  complete simultaneously two squares, using $x_{2}x_{3}=x_{3}x_{1}$ and $x_{1}x_{4}=x_{4}x_{2}$. At the last step, we close the diagram using $x_{1}x_{3}=x_{2}x_{4}$. In the down right diagram, the labels of all the directed paths from the upper left star to the other star represent the same element in $M$.
	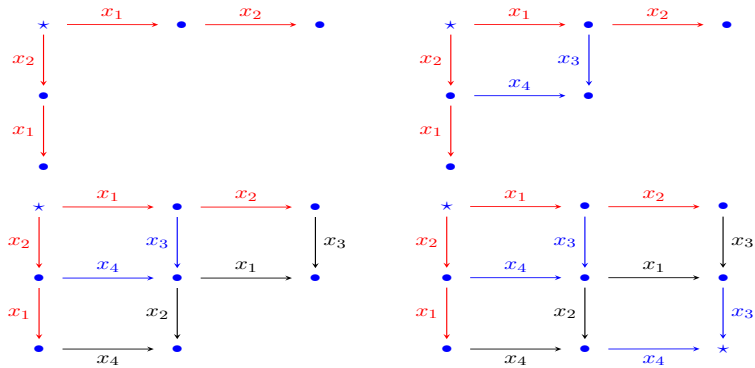
\begin{figure}[H] 
		\scalebox{0.8}[0.6]{
	\begin{tikzpicture}
	\matrix (m) [matrix of math nodes,row sep=3em,column sep=4em,minimum width=2em,color=blue, ampersand replacement=\&]
	{ 	\star \& \bullet \& \bullet \\
			\bullet\& \&  \\
		\bullet\& \& \\};
	\path[-stealth]
	(m-1-1) edge [red] node [above] {$x_1$} (m-1-2)
	(m-1-2) edge [red] node [above] {$x_2$} (m-1-3)

	(m-1-1) edge [red]node [left] {$x_2$} (m-2-1)
		(m-2-1) edge [red]node [left] {$x_1$} (m-3-1)	
	;
	\end{tikzpicture} }
\hspace{10pt}
	\scalebox{0.8}[0.6]{
			\begin{tikzpicture}
		\matrix (m) [matrix of math nodes,row sep=3em,column sep=4em,minimum width=2em,color=blue, ampersand replacement=\&]
		{ 	\star\& \bullet \& \bullet \\
			\bullet\& \bullet \&  \\
			\bullet\& \& \\};
		\path[-stealth]
		(m-1-1) edge [red] node [above] {$x_1$} (m-1-2)
		(m-1-2) edge [red] node [above] {$x_2$} (m-1-3)

		(m-1-1) edge [red]node [left] {$x_2$} (m-2-1)
		(m-2-1) edge [red]node [left] {$x_1$} (m-3-1)
		
			(m-1-2) edge [blue] node [left] {$x_3$} (m-2-2)
				(m-2-1) edge [blue] node [above] {$x_4$} (m-2-2)
		;
		\end{tikzpicture}}
		
			\scalebox{0.8}[0.6]{
			\begin{tikzpicture}
		\matrix (m) [matrix of math nodes,row sep=3em,column sep=4em,minimum width=2em,color=blue, ampersand replacement=\&]
		{ 	\star\& \bullet \& \bullet \\
			\bullet\& \bullet \& \bullet \\
			\bullet\& \bullet \& \\};
		\path[-stealth]
		(m-1-1) edge [red] node [above] {$x_1$} (m-1-2)
		(m-1-2) edge [red] node [above] {$x_2$} (m-1-3)

		(m-1-1) edge [red]node [left] {$x_2$} (m-2-1)
		(m-2-1) edge [red]node [left] {$x_1$} (m-3-1)
		
		(m-1-2) edge [blue] node [left] {$x_3$} (m-2-2)
		(m-2-1) edge [blue] node [above] {$x_4$} (m-2-2)
		
			(m-2-2) edge [] node [above] {$x_1$} (m-2-3)
		(m-1-3) edge [] node [right] {$x_3$} (m-2-3)
			(m-2-2) edge [] node [left] {$x_2$} (m-3-2)
		(m-3-1) edge [] node [below] {$x_4$} (m-3-2)
		;
		\end{tikzpicture}}
		\hspace{10pt}
		\scalebox{0.8}[0.6]{		\begin{tikzpicture}
		\matrix (m) [matrix of math nodes,row sep=3em,column sep=4em,minimum width=2em,color=blue, ampersand replacement=\&]
		{ 	\star\& \bullet \& \bullet \\
			\bullet\& \bullet \& \bullet \\
			\bullet\& \bullet \& \star\\};
		\path[-stealth]
		(m-1-1) edge [red] node [above] {$x_1$} (m-1-2)
		(m-1-2) edge [red] node [above] {$x_2$} (m-1-3)

		(m-1-1) edge [red]node [left] {$x_2$} (m-2-1)
		(m-2-1) edge [red]node [left] {$x_1$} (m-3-1)
		
		(m-1-2) edge [blue] node [left] {$x_3$} (m-2-2)
		(m-2-1) edge [blue] node [above] {$x_4$} (m-2-2)
		
		(m-2-2) edge [] node [above] {$x_1$} (m-2-3)
		(m-1-3) edge [] node [right] {$x_3$} (m-2-3)
		(m-2-2) edge [] node [left] {$x_2$} (m-3-2)
		(m-3-1) edge [] node [below] {$x_4$} (m-3-2)
		
			(m-2-3) edge [blue] node [right] {$x_3$} (m-3-3)
		(m-3-2) edge [blue] node [below] {$x_4$} (m-3-3)		
		;
		\end{tikzpicture}}
	\caption{Right reversing  to build a van kampen diagram for the pair $(x_1x_2x_3^2, x_2x_1x_4^2)$. }\label{fig-revers}
\end{figure}
\end{ex}
The process of reversing is not successful for  every monoid, and for every pair of elements.  Indeed, if a  monoid is a Garside monoid, as in the case of the monoid $M$ in Example \ref{ex-reversing}, then right and left reversing are successful for every pair of elements in the monoid  \cite{dehornoy,dehornoy2}. But in an arbitrary monoid, it is not necessarily true anymore, and it may occur that the process never terminates, and even if it terminates    there may be  some  obstructions. Assume we have a subdiagram with horizontal edge labelled $x$ and vertical edge labelled $y$.  If there is no relation $x...=y... $ in $R$, then the subdiagram  cannot be completed, and so the diagram neither.   On the opposite, if there are more than one relation $x...=y... $ in $R$, then there may be several  different  ways to close the diagram. 
We illustrate with the following example an obstruction of the first kind  in the use of right reversing that occurs  for the structure inverse monoid of a partial solution.

\begin{ex}\label{ex-revers-not-works}
Let $\operatorname{IM}(X,r)=\operatorname{Inv}\langle x_0,x_1,x_2 \mid 	 x_0x_2 =x_2x_1; x_1x_2=x_2x_0\rangle$ be the  structure inverse monoid of the partial solution described in  Example \ref{ex-square-free-partial-sol}. As there is no defining relation $x_0...=x_1...$, the mostright  diagram in Figure \ref{fig-partial} cannot be completed

\begin{figure}

\scalebox{0.9}[0.9]{
\begin{tikzpicture}
\matrix (m) [matrix of math nodes,row sep=3em,column sep=4em,minimum width=2em,color=blue, blue, ampersand replacement=\&]
{ \star  \& \bullet \\
	\bullet \&  \star \\
};
\path[-stealth]
(m-1-1) edge [red] node [above] {$x_0$} (m-1-2)
(m-1-1) edge [red]node [left] {$x_2$} (m-2-1)
(m-2-1) edge [red] node [below] {$x_1$} (m-2-2)
(m-1-2) edge [red]node [right] {$x_2$} (m-2-2)
;
\end{tikzpicture}}
\hspace{15pt}
\scalebox{0.9}[0.9]{
 \begin{tikzpicture}
 \matrix (m) [matrix of math nodes,row sep=3em,column sep=4em,minimum width=2em,color=blue, ampersand replacement=\&]
 { \star  \& \bullet \\
 	\bullet \&  \star \\
 };
 \path[-stealth]
 (m-1-1) edge [red] node [above] {$x_1$} (m-1-2)
 (m-1-1) edge [red]node [left] {$x_2$} (m-2-1)
 (m-2-1) edge [red] node [below] {$x_0$} (m-2-2)
 (m-1-2) edge [red]node [right] {$x_2$} (m-2-2)
 ;
 \end{tikzpicture}}
 \hspace{15pt}
 \scalebox{0.9}[0.9]{
 \begin{tikzpicture}
 \matrix (m) [matrix of math nodes,row sep=3em,column sep=4em,minimum width=2em,color=blue, ampersand replacement=\&]
 { \star  \& \bullet \\
 	\bullet \&  \\
 };
 \path[-stealth]
 (m-1-1) edge [red] node [above] {$x_0$} (m-1-2)
 (m-1-1) edge [red]node [left] {$x_1$} (m-2-1)
 ;
 \end{tikzpicture}}
 \caption{Right reversing in $\operatorname{IM}(X,r)$, $(X,r)$  a partial solution: the mostright cube could not be completed.}\label{fig-partial} 
\end{figure}
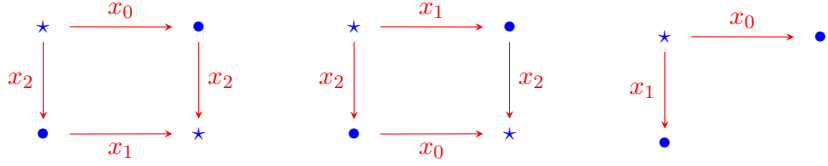

\end{ex}
\subsection{Partial braces and  partial solutions}\margin_comment{\textcolor{red}{MANY CHANGES  !!!}}
Given $(X,r)$  a non-degenerate involutive partial solution with structure inverse monoid  $\operatorname{IM}(X,r)$, the set of defining relations can be described in several forms that are useful for computations.
\begin{lem}\label{lem-defn-R}\margin_comment{\textcolor{red}{added lemma}}
	Let $(X,r)$ be a non-degenerate involutive partial solution with structure inverse monoid  $\operatorname{IM}(X,r)$.  Let $\mathcal{R}$ denote the set of defining relations of   $\operatorname{IM}(X,r)$. Then 
	\begin{gather}
\mathcal{R}=\{x_ix_{\sigma_i^{-1}(j)}\,=\,x_jx_{\sigma_j^{-1}(i)} \;\;\;\mid \;\;\;(x_i,x_{\sigma_i^{-1}(j)}) \in \mathcal{D}\} \label{eqn-R-sigma}\\
\mathcal{R}=	\{x_{\gamma_k^{-1}(l)}x_k\,=\,x_{\gamma_l^{-1}(k)}x_l\;\;\; \mid \;\;\;(x_{\gamma_k^{-1}(l)},x_k)\in \mathcal{D}\}\label{eqn-R-gamma}
	\end{gather}
\end{lem}
\begin{proof}
If $(x_i,x_{\sigma_i^{-1}(j)}) \in \mathcal{D}$, then $r(x_i,x_{\sigma_i^{-1}(j)})=(x_{\sigma_i\sigma_i^{-1}(j)}, x_{\gamma_{\sigma_i^{-1}(j)}(i)})=(x_j, x_{\gamma_{\sigma_i^{-1}(j)}(i)})$. From Lemma \ref{lem-formules-invol+braided-partial}, by replacing $x$ by  $i$ and $y$ by  $\sigma_i^{-1}(j)$  in $(\ref{eq-inv1})$, we have  $\gamma_{\sigma_i^{-1}(j)}(i)=\sigma_j^{-1}(i)$, that is $r(x_i,x_{\sigma_i^{-1}(j)})=(x_j, x_{\sigma_j^{-1}(i)})$ and  $\mathcal{R}$ can be described by  Equation \ref{eqn-R-sigma}.            

\noindent If  $(x_{\gamma_k^{-1}(l)},x_k)\in \mathcal{D}$, then $r(x_{\gamma_k^{-1}(l)},x_k)=(x_{\sigma_{\gamma_k^{-1}(l)}(k)},x_{\gamma_k\gamma_k^{-1}(l)})=(x_{\sigma_{\gamma_k^{-1}(l)}(k)},x_l)$. 
From Lemma \ref{lem-formules-invol+braided-partial}, by replacing $x$ by  $\gamma_k^{-1}(l)$ and $y$ by  $k$  in $(\ref{eq-inv2})$, we have  $\sigma_{\gamma_k^{-1}(l)}(k)=\gamma_l^{-1}(k)$, that is $r(x_{\gamma_k^{-1}(l)},x_k)=(x_{\gamma_l^{-1}(k)}, x_l)$ and the set  $\mathcal{R}$  can be described by Equation \ref{eqn-R-gamma}.
\end{proof}
Note that the set $\mathcal{R}$ as defined from $(X,r)$  is unique and Equations \ref{eqn-R-sigma}-\ref{eqn-R-gamma} are just different ways to describe it.
\margin_comment{\textcolor{red}{added rmk}} 
\begin{rem}\label{rem-unique-rel}
From these presentations of   $\mathcal{R}$ and the properties of  a  partial solution ($\sigma_i, \gamma_i$, $1 \leq i \leq n$,  are partial bijections), for every  pair $x_i,x_j \in X$, there is at most one defining relation of the form $x_it=x_jz$, with $t,z \in X$, and for every  pair  $x_k,x_l \in X$, there is at most one defining relation of the form $t'x_k=z'x_l$, with $t',z' \in X$. 
\end{rem}
Let $\mathcal{R}^*$ denote the following set of  relations, described in two ways (as in Lemma \ref{lem-defn-R}): 
 	\begin{gather}
 \mathcal{R}^*=\{x^*_{\sigma_i^{-1}(j)}x_i^*\,=\,x^*_{\sigma_j^{-1}(i)} x^*_j\;\;\;\mid \;\;\;(x_i,x_{\sigma_i^{-1}(j)}) \in \mathcal{D}\} \label{eqn-R*-sigma}\\
 \mathcal{R}^*=	\{x^*_kx^*_{\gamma_k^{-1}(l)}\,=\,x^*_lx^*_{\gamma_l^{-1}(k)}\;\;\; \mid \;\;\;(x_{\gamma_k^{-1}(l)},x_k)\in \mathcal{D}\}\label{eqn-R*-gamma}
 \end{gather}
Let $\mathcal{R}_{\rho}$ denote the following set of  relations, described in two ways (as in Lemma \ref{lem-defn-R}):
	\begin{gather*}
	\mathcal{R}_\rho=	\mathcal{R} \cup \mathcal{R} ^*\cup \{x^*_ix_j=x_{\sigma_i^{-1}(j)}x^*_{\sigma_j^{-1}(i)}\;\;;\;\; x^*_jx_i=x_{\sigma_j^{-1}(i)}x^*_{\sigma_i^{-1}(j)}\;\;\;\mid \;\;\; (x_ix_{\sigma_i^{-1}(j)}\,=\,x_jx_{\sigma_j^{-1}(i)})\in \mathcal{R}\}\\
		\mathcal{R}_\rho=	\mathcal{R} \cup \mathcal{R} ^*\cup \{x^*_{\gamma_k^{-1}(l)}x_{\gamma_l^{-1}(k)}=x_lx^*_k\;\;;\;\; x^*_{\gamma_l^{-1}(k)}x_{\gamma_k^{-1}(l)}=x_kx^*_l\;\;\;\mid \;\;\; (x_{\gamma_k^{-1}(l)}x_k\,=\,x_{\gamma_l^{-1}(k)}x_l)\in \mathcal{R}\}
	\end{gather*}
	Let  $\rho$ be the  congruence generated by 	$\mathcal{R}_\rho$ on  $\operatorname{IM}(X,r)$ and $\operatorname{IM}(X,r)/\rho$ the corresponding quotient monoid. We can reformulate now Theorem 2 in  a  more precise way:
\begin{thm}	 \label{thm-2-precise} \margin_comment{\textcolor{red}{Thm changed}}
	Let $(X,r)$ be a non-degenerate involutive partial solution with structure inverse monoid  $\operatorname{IM}(X,r)$.  	Let  $\rho$ be the  congruence generated by 	$\mathcal{R}_\rho$ on  $\operatorname{IM}(X,r)$. Then  there exists a partial left brace $(\mathcal{B},\oplus, \cdot)$  such that $\operatorname{IM}(X,r)/\rho$  is isomorphic $(\mathcal{B},\cdot)$.
\end{thm}
\begin{proof}
 Let $\operatorname{IM}_\rho$ denote the quotient monoid $\operatorname{IM}(X,r)/\rho$. 
 We define in $\operatorname{IM}_\rho$    a partial operation,  $\oplus$, and show that  $(\operatorname{IM}_\rho,\,\oplus,\,\cdot)$ is a partial brace.  Clearly, 
 $(\operatorname{IM}_\rho , \,\cdot)$ is an inverse monoid (and we  omit the $\cdot$). It remains to show that  $(\operatorname{IM}_\rho , \,\oplus)$ is a partial commutative monoid. The proof contains two parts: the iterative definition of $\oplus$ in $\operatorname{IM}_\rho$  and the proof $\oplus$ is well-defined and satisfies the relevant properties. 

	\margin_comment{\textcolor{red}{defn on $X \times X$}}
	First, we define $\oplus$  for pairs of elements of $X$. Let $x_i,x_j \in X$ such that  $(x_i,x_{\sigma_i^{-1}(j)}) \in \mathcal{D}$. Then  $x_i \oplus x_j$ and $x_j \oplus x_i$ exist  and are  defined by:
		\begin{gather}
	x_i\oplus x_j=x_ix_{\sigma_i^{-1}(j) }\;\;\;\;\;\;  x_j\oplus x_i=x_jx_{\sigma_j^{-1}(i) }\label{eqn++}
	\end{gather}
Since $x_i x_{\sigma_i^{-1}(j)}=x_jx_{\sigma_j^{-1}(i)}$ in $\operatorname{IM}(X,r)$, from Lemma \ref{lem-defn-R},   $x_i \oplus x_j=x_j \oplus x_i$ in $\operatorname{IM}_\rho$. \\
	\margin_comment{\textcolor{red}{defn on $X^* \times X^*$}}
Second, we define $\oplus$  for pairs of elements of $X^{*}$, where $X^{*}=\{x^*\mid x\in X\}$.  Let $x_k,x_l \in X$,  
	such that  $(x_{\gamma_k^{-1}(l)},x_k)\in \mathcal{D}$. 	Then $x_k^*\oplus x_l^*$ and $x_l^*\oplus x_k^*$ exist  and  are defined by:
	\begin{gather}
x_k^*\oplus x_l^*= x_k^*x^*_{\gamma_k^{-1}(l)}\;\;\;\;\;\; x_l^*\oplus x_k^*= x_l^*x^*_{\gamma_l^{-1}(k)}\label{eqn--}
	\end{gather}
	From Lemma \ref{lem-defn-R},  $x_{\gamma_k^{-1}(l)}x_k\,=\,x_{\gamma_l^{-1}(k)}x_l$,   so  $x_k^*x^*_{\gamma_k^{-1}(l)}= x_l^*x^*_{\gamma_l^{-1}(k)}$ also  in $\operatorname{IM}(X,r)$, that is  $x_k^*\oplus x_l^*=x_l^*\oplus x_k^*$ in $\operatorname{IM}_\rho$. 
	
Third, we define $\oplus$  for pairs of elements of $X\cup X^{*}$.  Let $x_i,x_j \in X$. If  	\margin_comment{\textcolor{red}{defn on $X \times X^*$}}
	$x_i \oplus x_j$ exists and $x_i \oplus x_j=x_it=x_jz$, with $t,z \in X$,  then $t\oplus x^*_i$  and $z\oplus x^*_j$ exist  and, for $x^*_i,x^*_j ,t^*,z^*\in X^*$: 
	\begin{gather}
	t\oplus x_i^*= tz^*   \;\;\;\;\;\;           	x_i^*\oplus t=x_i^*x_j \label{eqn-+1}\\
	z\oplus x_j^*= zt^*\;\;\;\;\;\;  x_j^*\oplus z= x_j^*x_i\label{eqn-+2}
	\end{gather}
From the definition of  $\mathcal{R}_\rho$,  $tz^*=    x_i^*x_j$  and  $zt^*= x_j^*x_i	$ in $\operatorname{IM}_\rho$, that is  $	z\oplus x_j^*=x_j^*\oplus z$  and 
$t\oplus x_i^*=x_i^*\oplus t$  in $\operatorname{IM}_\rho$. Note that whenever defined, $\oplus$ is well-defined for pairs of elements in $X\cup X^*$  (see Remark \ref{rem-unique-rel}).  In Figure \ref{fig-defn-+-}, we illustrate diagrammatically  the motivation for the equations (\ref{eqn++}) - (\ref{eqn-+2}). 
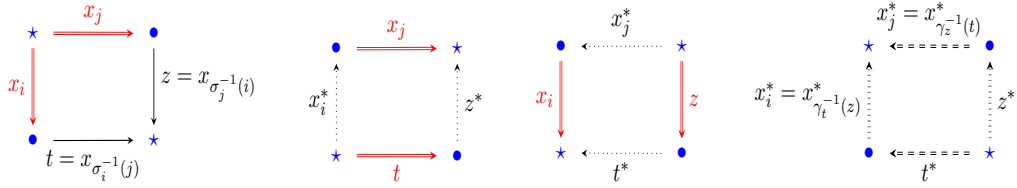
\begin{figure}[H]
	\scalebox{0.7}[0.9]{
		\begin{tikzpicture}
		\matrix (m) [matrix of math nodes,row sep=3em,column sep=4em,minimum width=2em,color=blue,ampersand replacement=\&]
		{ \star \& \bullet\\
			\bullet \& \star\\};
		\path[-stealth]
		(m-1-1) edge [double,red] node [above] {$x_j$} (m-1-2)
		
		(m-1-1) edge [double,red]node [left] {$x_i$} (m-2-1)
		
		(m-1-2) edge node [right] {$z=x_{\sigma^{-1}_j(i)}$} (m-2-2)
		
		(m-2-1) edge node [below] {$t=x_{\sigma^{-1}_i(j)}$} (m-2-2)
		;
		\end{tikzpicture}}
	\hspace{5pt}
	\scalebox{0.7}[0.9]{	\begin{tikzpicture}
		\matrix (m) [matrix of math nodes,row sep=3em,column sep=4em,minimum width=2em,color=blue,ampersand replacement=\&]
		{ \bullet \& \star\\
			\star 	\& \bullet\\};
		\path[-stealth]
		(m-1-1) edge [double,red] node [above] {$x_j$} (m-1-2)
		
		(m-2-1) edge [dotted]node [left] {$x_i^{*}$} (m-1-1)
		
		(m-2-2) edge[dotted] node [right] {$z^{*}$} (m-1-2)
		
		(m-2-1) edge[double,red] node [below] {$t$} (m-2-2)
		;
		\end{tikzpicture}}
	\hspace{5pt}
	\scalebox{0.7}[0.9]{	\begin{tikzpicture}
		\matrix (m) [matrix of math nodes,row sep=3em,column sep=4em,minimum width=2em,color=blue,ampersand replacement=\&]
		{  \bullet \& \star\\
			\star \& \bullet\\};
		\path[-stealth]
		(m-1-2) edge [dotted] node [above] {$x_j^{*}$} (m-1-1)
		
		(m-1-1) edge [double,red]node [left] {$x_i$} (m-2-1)
		
		(m-1-2) edge[double,red] node [right] {$z$} (m-2-2)
		
		(m-2-2) edge[dotted] node [below] {$t^{*}$} (m-2-1)
		;
		\end{tikzpicture}}
	\hspace{5pt}
	\scalebox{0.7}[0.9]{	\begin{tikzpicture}
		\matrix (m) [matrix of math nodes,row sep=3em,column sep=4em,minimum width=2em,color=blue,ampersand replacement=\&]
		{ \star \&  \bullet  \\
			\bullet \&	\star  \\};
		\path[-stealth]
		(m-1-2) edge [double, dashed] node [above] {$x_j^{*}=x^*_{\gamma^{-1}_z(t)}$} (m-1-1)
		
		(m-2-1) edge  [double, dotted]node [left] {$x_i^{*}=x^*_{\gamma^{-1}_t(z)}$} (m-1-1)
		
		(m-2-2) edge[double,dotted] node [right] {$z^{*}$} (m-1-2)
		
		(m-2-2) edge [double, dashed]node [below] {$t^{*}$} (m-2-1)
		;
		\end{tikzpicture}}
	\caption{From $x_it=x_jz$ in  the reversing diagram for $x\oplus y$ (most-left), we define $x_i^{*} \oplus z$,  $x_j^{*} \oplus t$, $t^{*} \oplus z^{*}$ by reversing parallel arrows. After  inversion of the direction of an arrow labelled $x$, the arrow is labelled $x^*$. Directed labelled paths  from  a star to another give equal  elements in $\operatorname{IM}_{\rho}$.}	\label{fig-defn-+-}
\end{figure}
	\margin_comment{\textcolor{red}{extended defn  to $\operatorname{IM}_{\rho}$}}
	Next, we  define $\oplus$  for pairs of elements in $\operatorname{IM}_{\rho}$.  
	For every $g \in \operatorname{IM}_{\rho}$, we define 
	\begin{gather}\label{eq-plus1}
	g \oplus 1=g1=1g\\
g\oplus g=g 1=1g \label{eq-plusss}
	\end{gather}
 For any two elements $g,h \in \operatorname{IM}_{\rho}$, we use 
 an extended version of  a  reversing diagram  in which we allow  arrows labelled with elements from $X^*$, to define iteratively $g\oplus h$.
 
   \begin{figure}[H]
 	\scalebox{0.8}[0.9]{
 		\begin{tikzpicture}
 		\matrix (m) [matrix of math nodes,row sep=3em,column sep=4em,minimum width=2em,color=blue,ampersand replacement=\&]
 		{ \star \& \bullet \& \bullet \& \bullet \& \bullet\\
 			\bullet \&  \bullet\& \bullet\& \bullet\& \bullet\\
 			\bullet \&  \bullet\& \bullet\& \bullet\& \bullet\\
 			\bullet \&  \bullet\& \bullet\& \bullet\&  \star\\};
 		\path[-stealth]
 		(m-1-1) edge [double,blue] node [above] {$x_{i_1}$} (m-1-2)
 		(m-1-2) edge [double,blue] node [above] {$x^*_{i_2}$} (m-1-3)
 		(m-1-3) edge [double,blue] node [above] {$...$} (m-1-4)
 		(m-1-4) edge [double,blue] node [above] {$x_{i_k}$} (m-1-5)
 		
 		(m-1-1) edge [double,red]node [left] {$x_{j_1}^{*}$} (m-2-1)
 		(m-2-1) edge [double,red]node [left] {$...$} (m-3-1)
 		(m-3-1) edge [double,red]node [left] {$x_{j_l}$} (m-4-1)
 		
 		(m-1-2) edge[blue] node [right] {$x_{n_1}^{*}$} (m-2-2)
 		(m-2-2) edge[dotted,blue] node [right] {$...$} (m-3-2)
 		(m-3-2) edge[dotted,blue] node [right] {} (m-4-2)
 		
 		(m-1-3) edge[dotted,blue] node [right] {} (m-2-3)
 		(m-1-4) edge[dotted,blue] node [right] {} (m-2-4)
 		(m-1-5) edge[double,blue] node [right] {} (m-2-5)

 		(m-2-2) edge[dotted,blue] node [right] {} (m-3-2)
 		(m-2-3) edge[dotted,blue] node [right] {} (m-3-3)
 		(m-2-4) edge[dotted,blue] node [right] {} (m-3-4)
 		(m-2-5) edge[double,blue] node [right] {$...$} (m-3-5)
 		
 		(m-3-2) edge[dotted,blue] node [right] {} (m-4-2)
 		(m-3-3) edge[dotted,blue] node [right] {} (m-4-3)
 		(m-3-4) edge[dotted,blue] node [right] {} (m-4-4)
 		(m-3-5) edge[double,blue] node [right] {} (m-4-5)
 		
 		(m-2-1) edge node [below] {$x_{m_1}$} (m-2-2)
 		(m-2-2) edge [] node [below] {$x^*_{m_2}$} (m-2-3)
 		(m-2-3) edge [] node [below] {$...$} (m-2-4)
 		(m-2-4) edge [] node [below] {} (m-2-5)
 		
 		(m-3-1) edge[dotted,red] node [below] {} (m-3-2)
 		(m-3-2) edge [dotted,red] node [below] {} (m-3-3)
 		(m-3-3) edge [dotted,red] node [below] {$...$} (m-3-4)
 		(m-3-4) edge [dotted,red] node [below] {} (m-3-5)
 		
 		(m-4-1) edge [double,red] node [below] {} (m-4-2)
 		(m-4-2) edge [double,red] node [below] {} (m-4-3)
 		(m-4-3) edge [double,red] node [below] {$...$} (m-4-4)
 		(m-4-4) edge [double,red] node [below] {} (m-4-5)
 		;
 		\end{tikzpicture}}
 	\caption{Computation of $g\oplus h$, with   $g=x_{i_1} x^*_{i_2}...x_{i_k}$, $h=x_{j_1}^{*}...x_{j_l}^{}$. }\label{fig-defn-oplus-mixed}
 \end{figure}
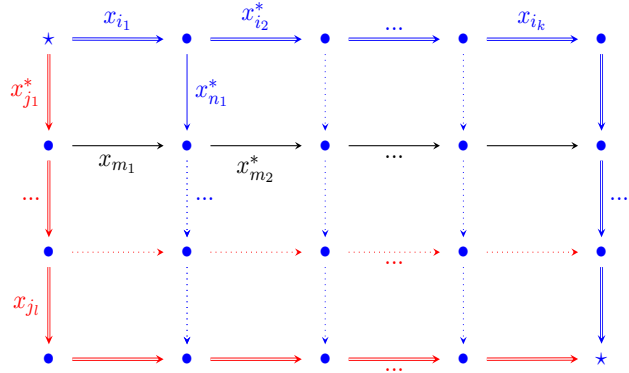	\margin_comment{\textcolor{red}{extension of reversing}}
 We begin at the left-most upper square and whenever possible we complete each square according to the unique definition of $\oplus$ for pairs of elements in $X\cup X^*$.  If each square can be closed, then $g \oplus h$ is defined and can be read from the directed path from the upper star to the lower star, going right and then down, $h \oplus g$ is also defined and it can be read from the directed path from the upper star to the lower star, going down  and then right. Clearly, $g \oplus h$  and $h \oplus g$ are equal in $\operatorname{IM}_{\rho}$.  If  the process of  reversing to define   $g\oplus h$ is not terminating, or if at any step, any square  could not be completed, then $g \oplus h$ is not defined.  
 If $g\oplus h$ exists and $g\oplus h=gu=hv$, then $g^*\oplus u$ and  $h^*\oplus v$ exist and  using the  extended version of   a reversing diagram  in which the inversion of the direction of an arrow labelled $x$ gives an arrow labelled $x^*$  (as in Figure \ref{fig-defn-oplus-mixed})  we have:
	\begin{gather}\label{eqn-welldef}
	g\oplus h=gu=hv\\
g^*\oplus u=g^*h=uv^* \label{eqn-welldef2}\\
h^*\oplus v=h^*g=vu^*
	\end{gather}\margin_comment{\textcolor{red}{proof well-defined}}
	We show  that 	 $\oplus$  is well-defined.  As $\operatorname{IM}_{\rho}$  is an inverse monoid with respect to the product, for every $g \in \operatorname{IM}_{\rho}$, $g$ and $gg^*g$ are equivalent,  so, we need to  show   that whenever $g\oplus h$ and $ gg^*g \oplus h$ exist, then $g\oplus h$ and $gg^*g \oplus h$ are also equivalent (see Figure \ref{fig-well-def}).
		\begin{figure}[h]
			\begin{tikzpicture}
		\matrix (m) [matrix of math nodes,row sep=3em,column sep=4em,minimum width=2em,color=blue, ampersand replacement=\&]
		{ \star\& \bullet\&\bullet\&\bullet	\\
		\bullet\& \bullet\&\bullet\&\star	\\};
		\path[-stealth]
		(m-1-1) edge [double,red] node [above] {$g$} (m-1-2)
		(m-1-2) edge [double,red] node [above] {$g^*$} (m-1-3)
		(m-1-3) edge [double,red] node [above] {$g$} (m-1-4)

		(m-1-1) edge [double,red]node [left] {h} (m-2-1)
		
		(m-1-2) edge[blue] node [right] {$u$} (m-2-2)
		(m-1-3) edge [blue]node [right] {$h$} (m-2-3)
		(m-1-4) edge [blue]node [right] {$u$} (m-2-4)

		(m-2-1) edge node [below] {$v$} (m-2-2)
		(m-2-2) edge   node [below] {$v^*$} (m-2-3)
		(m-2-3) edge    node [below] {$v$} (m-2-4)
		;
		\end{tikzpicture}
		\hspace{10pt}
		\begin{tikzpicture}
		\matrix (m) [matrix of math nodes,row sep=3em,column sep=4em,minimum width=2em,color=yellow]
		{ \star& \bullet\\
		\bullet& \star\\};
		\path[-stealth]
		(m-1-1) edge [double,red] node [above] {$g$} (m-1-2)
		
		(m-1-1) edge [double,red]node [left] {h} (m-2-1)
		
		(m-1-2) edge[blue] node [right] {$u$} (m-2-2)
	
		(m-2-1) edge node [below] {$v$} (m-2-2)
	
		;
		\end{tikzpicture}
		\caption{Right reversing  to compute $gg^*g\oplus h$ at left and  $g\oplus h$ at right.}\label{fig-well-def}
	\end{figure}
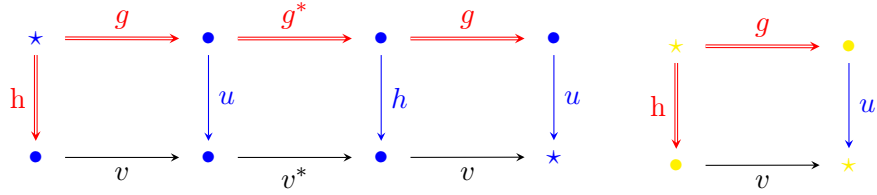\margin_comment{\textcolor{red}{well-defined for equiv elts}}

Note that if  $g\oplus h$ exists and $g\oplus h=gu=hv$, then $g^*\oplus u$ exists and  so $gg^*g \oplus h$ exists also.  From Equations \ref{eqn-welldef}, \ref{eqn-welldef2}, we  have $gg^*g \oplus h=  gg^*gu=hvv^*v$. Since  $hvv^*v$ and $hv$ are equivalent,  $gg^*g \oplus h=gu=hv=g\oplus h$ in $\operatorname{IM}_{\rho}$. In the same way, we can show that for any $g,h,w,z\in \operatorname{IM}_{\rho}$, $gzz^*ww^*\oplus h=gww^*zz^*\oplus h$.\\
 \margin_comment{\textcolor{red}{well-def for  elts $\equiv$mod $\rho$}}
 Let $g,g',h\in \operatorname{IM}_{\rho}$ and assume $g$ and $g'$ are equal in $\operatorname{IM}_{\rho}$. We show that  $g'\oplus h=g\oplus h$ in $\operatorname{IM}_{\rho}$. From the assumption,  $g'$ can be obtained from $g$ by  successive applications of  relations from $\mathcal{R}_{\rho}$,  so it is enough to show that after each single application there is equality. That  is,  we can assume  that  $g'$ is  obtained from $g$ after the application of a single relation $(\mathfrak{l}=\mathfrak{r})$  from $\mathcal{R}_{\rho}$: $g=a\mathfrak{l}b$, $g'=a\mathfrak{r}b$, $a,b,\mathfrak{l},\mathfrak{r} \in \operatorname{IM}_{\rho}$.  Assume  $h=x^{\epsilon}_{l}w$, with  $x^{\epsilon}_{l}\in X\cup X^*$, $w \in \operatorname{IM}_{\rho}$, as in Figure \ref{fig-well-def2}. Then, if we show that $(\mathfrak{l'}=\mathfrak{r'})$  belongs to  $\mathcal{R}_{\rho}$ and $x_{r'}=x_r$, then $x_{s'}=x_s$ and $b'_{2}$ is $b'_{1}$, that is $g'\oplus x^{\epsilon}_{l}=g\oplus x^{\epsilon}_{l}$ and $a'\mathfrak{l'}b'_1= a'\mathfrak{r'}b'_1$  in $\operatorname{IM}_{\rho}$. 
 	\begin{figure}[h]
 		\scalebox{0.8}[0.7]{
 	\begin{tikzpicture}
 	\matrix (m) [matrix of math nodes,row sep=3em,column sep=4em,minimum width=2em,color=blue, ampersand replacement=\&]
 	{ \star \& \bullet \&\bullet \&\bullet	\\
 		\bullet \& \bullet  \&\bullet \&\bullet	\\
 		\bullet \& \bullet \&\bullet \&\star	\\};
 	\path[-stealth]
 	(m-1-1) edge [double,red] node [above] {$a$} (m-1-2)
 	(m-1-2) edge [double,red] node [above] {$\mathfrak{l}$} (m-1-3)
 	(m-1-3) edge [double,red] node [above] {$b$} (m-1-4)

 	(m-1-1) edge [double,red]node [left] {$x^{\epsilon}_{l}$} (m-2-1)
 		(m-2-1) edge [double,red]node [left] {$w$} (m-3-1)
 	
 	(m-1-2) edge[blue] node [right] {$x^{\epsilon}_k$} (m-2-2)
 	(m-1-3) edge [blue]node [right] {$x^{\epsilon}_r$} (m-2-3)
 	(m-1-4) edge [blue]node [right] {$x^{\epsilon}_s$} (m-2-4)

 	(m-2-1) edge node [below] {$a'$} (m-2-2)
 	(m-2-2) edge   node [below] {$\mathfrak{l'}$} (m-2-3)
 	(m-2-3) edge    node [below] {$b'_{1}$} (m-2-4)
 	
 		(m-2-2) edge[dotted,blue] node [right] {} (m-3-2)
 	(m-2-3) edge [dotted,blue]node [right] {} (m-3-3)
 	(m-2-4) edge [dotted,blue]node [right] {} (m-3-4)

 	(m-3-1) edge[dotted,blue] node [below] {} (m-3-2)
 	(m-3-2) edge[dotted,blue]   node [below] {} (m-3-3)
 	(m-3-3) edge[dotted,blue]    node [below] {} (m-3-4)
 	;
 	\end{tikzpicture} }
 	\hspace{2pt}
		\scalebox{0.8}[0.7]{\begin{tikzpicture}
 \matrix (m) [matrix of math nodes,row sep=3em,column sep=4em,minimum width=2em,color=blue,ampersand replacement=\&]
 { \star \& \bullet \&\bullet \&\bullet	\\
 	\bullet \& \bullet \&\bullet \&\bullet	\\
 	\bullet \& \bullet \&\bullet \&\star	\\};
 \path[-stealth]
 (m-1-1) edge [double,red] node [above] {$a$} (m-1-2)
 (m-1-2) edge [double,red] node [above] {$\mathfrak{r}$} (m-1-3)
 (m-1-3) edge [double,red] node [above] {$b$} (m-1-4)

 (m-1-1) edge [double,red]node [left] {$x^{\epsilon}_{l}$} (m-2-1)
 (m-2-1) edge [double,red]node [left] {$w$} (m-3-1)
 
 (m-1-2) edge[blue] node [right] {$x^{\epsilon}_k$} (m-2-2)
 (m-1-3) edge [blue]node [right] {$x^{\epsilon}_{r'}$} (m-2-3)
 (m-1-4) edge [blue]node [right] {$x^{\epsilon}_{s'}$} (m-2-4)

 (m-2-1) edge node [below] {$a'$} (m-2-2)
 (m-2-2) edge   node [below] {$\mathfrak{r'}$} (m-2-3)
 (m-2-3) edge    node [below] {$b'_{2}=b'_{1}$} (m-2-4)
 
 (m-2-2) edge[dotted,blue] node [right] {} (m-3-2)
 (m-2-3) edge [dotted,blue]node [right] {} (m-3-3)
 (m-2-4) edge [dotted,blue]node [right] {} (m-3-4)

 (m-3-1) edge[dotted,blue] node [below] {} (m-3-2)
 (m-3-2) edge[dotted,blue]   node [below] {} (m-3-3)
 (m-3-3) edge[dotted,blue]    node [below] {} (m-3-4)
 ;
 \end{tikzpicture}}
 	\caption{ The iterative computation of  $g\oplus h$ at left and  $g'\oplus h$ at right, with $g=a\mathfrak{l}b$, $g'=a\mathfrak{r}b$, $h=x^{\epsilon}_{l}w$.}\label{fig-well-def2}
 \end{figure}\margin_comment{\textcolor{red}{enough to prove:}} \margin_comment{\textcolor{red}{$\mathfrak{l}\oplus x^{\epsilon}=\mathfrak{r}\oplus x^{\epsilon}$}}

Using the same argument in the completion of the diagram row by row in the same way, we then obtain $g'\oplus h=g\oplus h$ in $\operatorname{IM}_{\rho}$. We prove that  $(\mathfrak{l'}=\mathfrak{r'})$  belongs to  $\mathcal{R}_{\rho}$ and that $x_{r'}=x_r$  in a  case by case proof which relies entirely on the properties of the partial solution from Lemma \ref{lem-formules-invol+braided-partial}. The proof  is in  Appendix  $2$. So,   $\oplus$ is well-defined in $\operatorname{IM}_{\rho}$.\\
\margin_comment{\textcolor{red}{associativity}}
We show now that for every $g,h,k \in \operatorname{IM}_{\rho}$, $(g\oplus h) \oplus k= g \oplus(h \oplus k)$,  whenever they are defined. Assume  $g \oplus h=gu=hv$, $h \oplus k=hw=kz$, $v \oplus w=vk'=wv'$. So, we can compute $(g\oplus h ) \oplus k$ and $g  \oplus(h\oplus k)$ as described in Figure  \ref{fig-associative}. As one can see,  along the thick path from the upper left star 	to the down right star,  we read  in both diagrams a common element, so  $(g\oplus h ) \oplus k=g  \oplus(h\oplus k)$ in $\operatorname{IM}_{\rho}$  (whenever  defined).  So,   $(\operatorname{IM}_{\rho}, \oplus)$ is a  commutative partial  monoid, with identity element equal to $1$, the identity of  $ (\operatorname{IM}_{\rho},\cdot)$.
  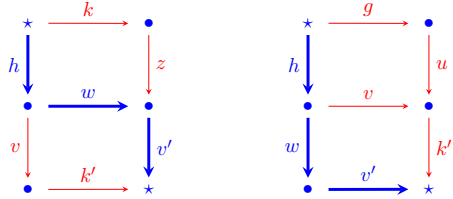
\begin{figure}[h] 	
  	\scalebox{0.7}[0.7]{
  			\begin{tikzpicture}
  			\matrix (m) [matrix of math nodes,row sep=3em,column sep=4em,minimum width=2em,color=blue,ampersand replacement=\&]
  			{ 	\star \& \bullet \\
  				\bullet\& 	\bullet\\		
  				\bullet\&   \star\\};
  			\path[-stealth]
  			(m-1-1) edge [red] node [above] {$k$} (m-1-2)
  			(m-1-2) edge [red] node [right] {$z$} (m-2-2)
  			
  				(m-1-1) edge [ultra thick,blue] node [left] {$h$} (m-2-1)
  			(m-2-1) edge [ultra thick,blue] node [above] {$w$} (m-2-2)
  			(m-2-1) edge [red] node [left] {$v$} (m-3-1)
  			(m-2-2) edge [ultra thick,blue] node [right] {$v'$} (m-3-2)
  				(m-3-1) edge [red] node [above] {$k'$} (m-3-2)
  			;
  			\end{tikzpicture}}	\hspace{30pt} \scalebox{0.7}[0.7]{\begin{tikzpicture}
  		\matrix (m) [matrix of math nodes,row sep=3em,column sep=4em,minimum width=2em,color=blue, ampersand replacement=\&]
  		{ 	\star \& \bullet \\
  		\bullet\& 	\bullet\\		
  		\bullet\&   \star\\};
  	\path[-stealth]
  	(m-1-1) edge [red] node [above] {$g$} (m-1-2)
  	(m-1-2) edge [red] node [right] {$u$} (m-2-2)
  	
  	(m-1-1) edge [ultra thick,blue] node [left] {$h$} (m-2-1)
  	(m-2-1) edge[red]  node [above] {$v$} (m-2-2)
  	(m-2-1) edge [ultra thick,blue]node [left] {$w$} (m-3-1)
  	(m-2-2) edge [red] node [right] {$k'$} (m-3-2)
  	(m-3-1) edge [ultra thick,blue]node [above] {$v'$} (m-3-2)
  	;
  		\end{tikzpicture}}
  	\caption{Computation of  $(g\oplus h ) \oplus k$ at left  and $g  \oplus(h\oplus k)$ at right }	\label{fig-associative}	
  \end{figure}  		
\margin_comment{\textcolor{red}{left-distributive}}
 	It remains to show that Equation \ref{defn-dist} holds, that is $a (g\oplus h)=ag\oplus ah$, for every $a,g,h \in \operatorname{IM}_{\rho}$, such  that this expression is defined. From $g\oplus h=gu=hv$, multiplying at left by $a$, we have $a (g\oplus h)=agu=ahv$, and using the reversing diagram for $ag\oplus ah$, and Equations \ref{eq-plus1}, \ref{eq-plusss}, we obtain $ag\oplus ah=agu=ahv$, that is $a (g\oplus h)=ag\oplus ah$. So, $(\operatorname{IM}_{\rho}, \cdot,\oplus)$ is a partial  left brace.	
\end{proof}
\begin{rem}\margin_comment{\textcolor{red}{added rmk+diagram}}
	Note that in the proof of Theorem \ref{thm-2-precise},  the operation $\oplus$ defined satisfies  a monoidal version of  the left-distributivity-like axiom in a left brace as described in Equation  \ref{eqn-brace},  that is  $a \oplus \,a\cdot(g\oplus h)  =a\cdot g \oplus a\cdot h$, for every $a,g,h \in \operatorname{IM}_\rho$ (see Figure \ref{fig-last}).
	\begin{figure}[H] 
		\scalebox{0.7}[0.7]{
		\begin{tikzpicture}
		\matrix (m) [matrix of math nodes,row sep=3em,column sep=4em,minimum width=2em,color=blue, ampersand replacement=\&]
		{ \star\& \bullet\\
			\bullet\& \star\\};
		\path[-stealth]
		(m-1-1) edge [double,red] node [above] {$g$} (m-1-2)
		
		(m-1-1) edge [double,red]node [left] {$h$} (m-2-1)
		
		(m-1-2) edge[blue] node [right] {$u$} (m-2-2)
		
		(m-2-1) edge node [below] {$v$} (m-2-2)
		;
		\end{tikzpicture}}
		\hspace{10pt}
		\scalebox{0.7}[0.7]{	\begin{tikzpicture}
		\matrix (m) [matrix of math nodes,row sep=3em,column sep=4em,minimum width=2em,color=blue, ampersand replacement=\&]
		{ 	\star\& \bullet \& \bullet \\
			\bullet\& \bullet \& \bullet \\
			\bullet\& \bullet \& \star\\};
		\path[-stealth]
		(m-1-1) edge [double,red] node [above] {$a$} (m-1-2)
		(m-1-2) edge [double,red] node [above] {$g$} (m-1-3)

		(m-1-1) edge [double, red]node [left] {$a$} (m-2-1)
		(m-2-1) edge [double, red]node [left] {$h$} (m-3-1)
		
		(m-1-2) edge [blue] node [left] {$1$} (m-2-2)
		(m-2-1) edge [] node [above] {$1$} (m-2-2)
		
		(m-2-2) edge [] node [above] {$g$} (m-2-3)
		(m-1-3) edge [blue] node [right] {$1$} (m-2-3)
		(m-2-2) edge [blue] node [left] {$h$} (m-3-2)
		(m-3-1) edge [] node [below] {$1$} (m-3-2)
		
		(m-2-3) edge [blue] node [right] {$u$} (m-3-3)
		(m-3-2) edge [] node [below] {$v$} (m-3-3)
		;
		\end{tikzpicture} }
	\hspace{10pt}
	\scalebox{0.7}[0.7]{	\begin{tikzpicture}
		\matrix (m) [matrix of math nodes,row sep=3em,column sep=4em,minimum width=2em,color=blue, ampersand replacement=\&]
		{ 	\star\& \bullet \& \bullet  \& \bullet \\
			\bullet\& \bullet \& \bullet \&\star\\};
		\path[-stealth]
		(m-1-1) edge [double,red] node [above] {$a$} (m-1-2)
		(m-1-2) edge [double,red] node [above] {$g$} (m-1-3)
			(m-1-3) edge [double,red] node [above] {$u$} (m-1-4)
		
		(m-1-1) edge [double, red]node [left] {$a$} (m-2-1)
	
		(m-1-2) edge [blue] node [left] {$1$} (m-2-2)
		(m-2-1) edge [] node [below] {$1$} (m-2-2)
		
		(m-2-2) edge [] node [below] {$g$} (m-2-3)
		(m-1-3) edge [blue] node [right] {$1$} (m-2-3)
		(m-1-4) edge [blue] node [right] {$1$} (m-2-4)

		(m-2-3) edge [blue] node [below] {$u$} (m-2-4)
	
		;
		\end{tikzpicture} }
		\caption{Computation of  $g \oplus h$,  $ ag\oplus ah$, and $a \oplus a(g \oplus h)$ from left to right}
		\label{fig-last}
	\end{figure}
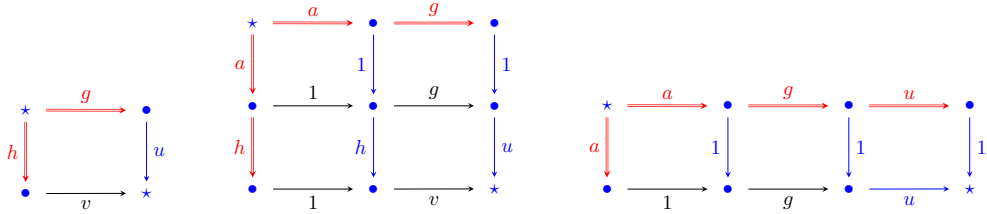 
\end{rem}
Note that the same process could be done  with left-reversing instead of right-reversing, and in this case  $(\operatorname{IM}_{\rho}, \cdot,\oplus)$ would  have been  a  partial  right  brace.
Although the structure of  $(\operatorname{IM}_{\rho}, \oplus)$  is reminiscent of that of  a category, it is not one. Indeed, in a category,  a strong associativity axiom is satisfied: if $a \oplus b$ and $b \oplus c$ are defined, then $a \oplus c$  is also defined, which  is not  satisfied in $(\operatorname{IM}_{\rho}, \oplus)$. 

Given a partial left brace $\mathcal{B}$, it is not clear how to define  a  non-degenerate involutive  partial solution  associated to  $\mathcal{B}$. If we try, like in the case of the classical brace,   to define a homomorphism
 $\lambda:  (\mathcal{B},\cdot)\rightarrow \operatorname{Aut}(\mathcal{B},\oplus) $ such that  $r(x,y)=(\lambda_x(y), \,\lambda^{-1}_{\lambda_x(y)}(x))$ is a   non-degenerate involutive  partial set-theoretic solution, we encounter many difficulties. Indeed, as $(\mathcal{B},\cdot)$ is an inverse monoid and not a group, there is no left or right cancellation rules, and the maps $\lambda_x$  are not necessarily injective.
 \section{The Thompson group $F$ as the structure group of a  partial set-theoretical solution of the  QYBE}
 \subsection{Proof of Theorem $3$}
 We consider the infinite presentation of the group $F$ from Theorem \ref{theo-F-infinite-pres} and show that $F$ is the structure group of a square-free, non-degenerate, involutive  partial set-theoretic solution.
 We recall that $F= \operatorname{Gp}\langle x_0,x_1,x_2,... \mid x_nx_k=x_kx_{n+1},\, 0 \leq k<n\rangle$.  Let $X=\{x_0,x_1,x_2, ...\}$. For brevity of notation, we often denote  these elements by $\{0,1,2, ...\}$. Let $\sigma_n:X\rightarrow X$ and $\gamma_n:X\rightarrow X$ be the following partial functions.

 \begin{equation}\label{defn-functions-for F}
 \sigma_n(k) = \left\{
 \begin{array}{lll}
 
 k              && k \leq n\\
 \text{not defined}     && k=n+1\\
 k-1&&  k \geq n+2\\
 \end{array} 
 \right.  \hspace{3cm}
 \gamma_n(k) = \left\{
 \begin{array}{lll}
 k              && k \leq n-2\\
 \text{not defined}             && k=n-1\\
 n && k=n\\
 k+1&&  k \geq n+1\\
 \end{array} 
 \right. 
 \end{equation}
 Let  define the following subsets of $X$:  $\mathcal{D}_{\sigma_n}=X \setminus\{x_{n+1}\}$,  $\mathcal{R}_{\sigma_n}=X $, $\mathcal{D}_{\gamma_n}=X \setminus\{x_{n-1}\}$  and 
 $\mathcal{R}_{\gamma_n}=X \setminus\{x_{n-1},x_{n+1}\}$. So, $\sigma_n: \mathcal{D}_{\sigma_n} \rightarrow\mathcal{R}_{\sigma_n}$, and $\gamma_n:\mathcal{D}_{\gamma_n} \rightarrow \mathcal{R}_{\gamma_n}$ are functions.

  Let define the following partial function:
   \begin{equation}\label{defn-S-for F}
\begin{array}{ll}
& r: X \times X \rightarrow X \times X\\
 & r(x_i,x_j)=(x_{\sigma_i(j)}\,,\,x_{\gamma_j(i)})
  \end{array}
\end{equation}
We denote by $\mathcal{D} \subset X \times X$ and $\mathcal{R}\subset X \times X$  the domain 
and  range of $r$ respectively. We show that the pair $(X,r)$ just defined is a square-free, non-degenerate, involutive  partial set-theoretic solution.

 \begin{lem}\label{lem-thompson-partial-sol}
Let  $r: \mathcal{D} \rightarrow \mathcal{R}$  as defined in Eq. \ref{defn-S-for F}. Then $(X,r)$  satisfies the following properties:
\begin{enumerate}[(i)]
	\item $ \mathcal{D}=\mathcal{R}=X^2\setminus\{(x_n,x_{n+1})\mid n\geq 0\}$.
	\item $r: \mathcal{D} \rightarrow \mathcal{R}$  is  bijective.
	\item $r(x_i,x_i)=(x_i,x_i)$, for all pairs  $(x_i,x_i)\in X^2$. 
	\item for every  $x \in X$, $\sigma_{x}:\mathcal{D}_{\sigma_{x}}\rightarrow \mathcal{R}_{\sigma_{x}}$ and $\gamma_{y}:\mathcal{D}_{\gamma_{y}}\rightarrow \mathcal{R}_{\gamma_{y}}$ are bijective.
	\item $r^{2}(x_i,x_j)=(x_i,x_j)$, for all pairs  $(x_i,x_j)\in X^2$ such that $r^{2}(x_i,x_j)$  is defined. 
	\item $r^{12}r^{23}r^{12}(x_i,x_j,x_k)=r^{23}r^{12}r^{23}(x_i,x_j,x_k)$,  for all triples $(x_i,x_j,x_k) \in X^3$ such that  both $r^{12}r^{23}r^{12}(x_i,x_j,x_k)$ and $r^{23}r^{12}r^{23}(x_i,x_j,x_k)$ are defined.

\end{enumerate}
 That is, $(X,r)$  is a 
 a non-degenerate,  square-free, involutive  partial set-theoretic solution. 
\end{lem}
\begin{proof}
$(i)$, $(ii)$, $(iii)$  hold from the definition of $r$, $r(x_i,x_j) = \left\{
\begin{array}{lll}
	
(x_{j-1},x_i)              && i \leq j-2\\
	\text{not defined}     && i=j-1\\
	(x_j,x_i) &&  i=j\\
		(x_j,x_{i+1}) &&  i\geq j+1\
\end{array} 
\right. $\\
$(iv)$  Clearly, from their definition, $\sigma_i$ and $\gamma_i$ are partial bijections of $X$, for every $i \geq 0$.
$(v)$  $r^2(x_i,x_j) = r(\left\{
\begin{array}{lll}

(x_{j-1},x_i)              && i \leq j-2\\
\text{not defined}     && i=j-1\\
(x_j,x_i) &&  i=j\\
(x_j,x_{i+1}) &&  i\geq j+1\
\end{array} 
\right.=
 \left\{
 \begin{array}{lll}
 
 (x_{i},x_j)              && i \leq j-2\\
 \text{not defined}     && i=j-1\\
 (x_i,x_j) &&  i=j\\
 (x_i,x_{j}) &&  i\geq j+1\
\end{array} 
\right.$\\

$(vi)$ is a case by case proof  and it appears in the appendix.
\end{proof} 

 We denote by  $\mathcal{F}$ the non-degenerate involutive partial  set-theoretic solution defined above.  We prove some properties of  $\mathcal{F}$.
 
 \begin{thm}\label{theo-F-structure-gp}
 	Let  $r: \mathcal{D} \rightarrow \mathcal{R}$  as defined in Eq. \ref{defn-S-for F}. Then
 	\begin{enumerate}[(i)]
 		\item $G(X,r)$, the structure group of $\mathcal{F}$,  is isomorphic to the Thompson group $F$.
 		\item $\operatorname{IM}(X,r)$, the structure inverse monoid  of $\mathcal{F}$,   embeds into the inverse  monoid $A \Join \operatorname{I}_X$, where $A$ is the  commutative  inverse monoid $\{f: \mathcal{D}_f\rightarrow \mathbb{Z} \mid \mathcal{D}_f \subseteq X\}$, with pointwise operation, and $ \operatorname{I}_X$ is the inverse symmetric monoid.
 	\end{enumerate}   
 \end{thm}
 \begin{proof}
 $(i)$ Fom its definition, $G(X,r)=\operatorname{Gp}\langle x_0,x_1,x_2...\mid x_ix_j=x_{\sigma_i(j)}\,x_{\gamma_j(i)}; \; i,j\geq 0\rangle $.  That is, using the description of $r$ in the proof of Lemma \ref{lem-thompson-partial-sol}, the defining relations of  $G(X,r)$ have the following form:
$\left\{	\begin{array}{cc}
 	x_ix_j= x_{j-1}x_i  &   i \leq j-1\\
 	x_ix_j=x_jx_{i+1}   &  i\geq j+2\\
 	\end{array}\right.$\\
 	There are trivial relations of the form $x_ix_i=x_ix_i$, for every $i \geq 0$,   and there are no relations of the form $x_ix_{i+1}=...$. Rewriting the above relations in a unified form,  we have	$G(X,r)=\operatorname{Gp}\langle x_0,x_1,x_2...\mid x_nx_k=x_{k}\,x_{n+1}; \;0 \leq  k<n \rangle $, the infinite presentation of the Thompson group $F$ (from  Theorem \ref{theo-F-infinite-pres}), that is $G(X,r)$ is isomorphic to   $F$.\\
 	$(ii)$ results directly from Theorem \ref{theo_psi-hom}.
 	\end{proof}

 \begin{lem}\label{lem-f-decomposable}
 	Let $\mathcal{F}$ as defined above, with  $X=\{x_0,x_1,..\}$. Then  
 	\begin{enumerate}[(i)]
 		\item 	$\mathcal{F}$  is irretractable.
 			\item 	$\mathcal{F}$  is decomposable.
 	\end{enumerate}
 
 \end{lem}

 \begin{proof}
 	$(i)$ results directly from the definition of the functions $\{\sigma_x\mid x\in X\}$.\\
 $(ii)$	Let  $X=\{x_0\}\cup \,Y_1$, where $Y_1=\{x_1,x_2,...\}$. 	We show that $\{x_0\}$ and $ Y_1=\{x_1,x_2,...\}$ are non-degenerate invariant subsets of $X$. Clearly, $r(x_0,x_0)=(x_0,x_0)$, and  furthermore,  $r(Y_1,Y_1)\subseteq (Y_1,Y_1)$, since for $0<k<n$,  $r(x_n,x_k) = (x_k,x_{n+1})$. For $n\geq 1$,  the partial bijections of $X$, written in the form of (infinite) permutations,    $\sigma_n=(0)(1)...(n)(...,n+3,n+2)$ and  $\gamma_n=(0)(1)...(n-2)(n)(n+1,n+2,...)$ restricted to $Y_1$ are  partial bijections of $Y_1$. That is, the restriction of $r$  to $Y_1^2$ is a non-degenerate partial solution. This solution itself is  decomposable. Indeed, $Y_1=\{x_1\}\cup \, Y_2$, where $Y_2=\{x_2,x_3,...\}$,   and  $\{x_1\}$ and $ Y_2=\{x_2,...\}$ are non-degenerate invariant subsets of $Y_1$. This process can be iterated infinitely.
 \end{proof}
\begin{rem}
	In the proof of Lemma \ref{lem-f-decomposable}, we have  that $\{x_0\}$ and $ Y_1=\{x_1,x_2,...\}$ are non-degenerate invariant subsets of $X$. Let  $G_0=\operatorname{Gp}\langle x_0\mid\; \rangle$ and  $G_1=G(Y_1,r\mid_{Y_1^2})$,  the structure group of  $(Y_1,r\mid_{Y_1^2})$. For a classical set-theoretic  solution, if there is an action of  $G_0$ on  $G_1$ by conjugation,   we have $G(X,r)\simeq G_1 \rtimes G_0$, and iteratively  $G(X,r)\simeq (\,(\,(G_n \rtimes \langle x_{n-1}\rangle)\rtimes \langle x_{n-2}\rangle\,)\rtimes ...\rtimes \langle x_1\rangle)\rtimes\langle x_0\rangle$.
	In the case of  the partial solution $\mathcal{F}$, as  we have $x_0x_nx_0^{-1}=x_{n-1}$, only for every $n\geq 2$ (i.e no relation $x_0x_1x_0^{-1}=...$), the  action of $G_0$ on  $G_1$ by conjugation is not total. It is possible to consider the inverse monoids, $M_0=\operatorname{Inv}\langle x_0\mid -\rangle$ and  $M_1=\operatorname{IM}(Y_1,r\mid_{Y_1^2})$, and  the   action of $M_0$ on $M_1$ by partial permutations, but nevertheless it is not clear which kind of decomposition is obtained.
\end{rem} 
 We now turn to the study of the existence of a cycle set corresponding to $\mathcal{F}$. 
 \begin{thm}\label{theo-cycle-F}
Let $\mathcal{F}$ as defined above, with  $X=\{x_0,x_1,..\}$. 
Then  $\mathcal{F}$  induces  a non-degenerate and square-free cycle set $(X,\star)$.
\end{thm}

\subsection{Comparison between solutions, partial solutions  and  $\mathcal{F}$}
A question  that arises naturally is which properties do the structure groups of  set-theoretic solutions and of partial set-theoretic solutions  share in common. In particular, as set-theoretic  solutions and their structure groups have been intensively studied, it would be interesting to investigate which properties  satisfied by  the  structure group of a set-theoretic solution is also satisfied by the structure group of  a partial solution. The structure inverse monoid has not been defined for  set-theoretic solutions, so  it would be interesting to understand whether its definition would provide some interesting information on the set-theoretic  solutions, and on the comparison between properties of set-theoretic   solutions and partial solutions.  As  Thompson's group $F$ is the structure group of  a square-free partial set-theoretic  solution,  we can examine  some of its  properties in comparison with those of structure groups of set-theoretic  solutions.  It is not clear if  we can generalize these observations, since on one hand   $F$ is the structure group of  $(X,r)$, with $(X,r)$  partial and also $X$ infinite, and on the other hand  structure groups of solutions form an infinite family of groups. Nevertheless, it is interesting to see the differences and we present some points below.
 
 \begin{itemize}
 	\item  The structure monoid of a solution  $(X,r)$, with $X$ finite, is  Garside, and its group of fractions is a Garside group \cite{chou_art}.  Garside groups are torsion-free and biautomatic \cite{deh_torsion,dehornoy,garside1}.  Thompson's group $F$  is also torsion-free \cite{thompson}, but it is not known wether it is automatic.

 \item    The structure group of a solution  $(X,r)$, with $X$ finite, is solvable \cite{etingof}. Thompson's group $F$   commutator subgroup $F'$ is simple, so $F''=F'$, and $F$ is not nilpotent, nor solvable \cite{canon}, \cite{brown-geo}, \cite{belk}.
 
 \item The centre  of the structure group of an indecomposable solution $(X,r)$, with $X$ finite,  is cyclic \cite{chou_art}. The center of $F$ is trivial \cite{canon}.
 
 \item The structure group of a solution  $(X,r)$, with $X=n$,  is a Bieberbach group,  as it  acts on the Euclidean $n$-dimensional space \cite{gateva_van}. As far as we know, there is no result of  this kind  for $F$.
 
 \item The quotient group $F/F'$ is isomorphic to $\mathbb{Z}^2$, and so any proper quotient of $F$ is abelian \cite{brown-geo}, \cite{belk}. This is not necessarily the case for the structure group of a solution.  Indeed, we describe the structure group $G(X,r)$  of   a  square-free  solution  $(X,r)$, with $X$ infinite, and with a  quotient group   not abelian. Let $G=\operatorname{Gp}\langle x_0,x_1,x_2\mid
x_0x_1=x_1x_0\,;\, x_2x_1=x_0x_2\,;\,x_2x_0=x_1x_2\rangle $, the structure group of a square-free solution with set $\{x_0,x_1,x_2\}$,  and $G \simeq  \mathbb{Z}^2 \rtimes \mathbb{Z}$.
The subgroup $N=\langle x_0^2,x_1^2,x_2^2 \rangle $  is normal and free abelian of rank $3$, and the quotient group of $G$ by $N$  is a finite group of order $8$  \cite{chou_godel2}. We do not get into the details of the construction of this kind of quotient group and refer the reader to \cite{chou_godel2}, and \cite{deh_coxeterlike}  for a generalisation. Let  $G(X,r)$  and $N(X,r)$ 
be the direct product of an infinite number of copies of $G$  and $N$,  respectively. So, 
 $G(X,r)$   is the structure group of a square-free  infinite solution, and  quotienting it by $N(X,r)$,    it has an infinite quotient group $W$  which is not abelian, since $W$ isomorphic to  the direct product of an infinite number of copies of  $(\mathbb{Z}_2 \times\mathbb{Z}_2)\rtimes \mathbb{Z}_2$.
  \end{itemize}
 To conclude, there are three Thompson's groups $F,T,V$,  with  $F\subset T \subset V$. There have been also  several generalisations of these groups. It would be interesting to know if there is some kind of solution 
such  that these groups are derived from them.
\section{Appendix}
\subsection{Appendix 1: Proofs of Lemma  \ref{lem-thompson-partial-sol}	$(vi)$ and Theorem \ref{theo-cycle-F}}
\begin{proof}[Proof of Lemma  \ref{lem-thompson-partial-sol}	$(vi)$  ]
	
We show by a case by case proof that the conditions from Lemma \ref{lem-formules-invol+braided-partial}$(ii)$ are satisfied. For brevity of notation, instead of $x_m$, $x_n$, $x_k$,  we write $m,n,k$
	
	$\sigma_m\sigma_n(k)=\left\{
	\begin{array}{lll}
	
	k           && k<n<m\\
	
	k           && k<m<n, \,n \neq m+1\\
	
	k-2    &&  m<n<k, \,k\neq n+1,m+2\\
	
	k-1    &&  m<k<n, \,n,k\neq m+1\\
	k-1    &&  n<k<m, \,k\neq n+1\\

	k-2    &&  n<m<k, \,k> m+2\\
	
	\text{not defined}    &&  n<m<k, \,k= m+1\\
\end{array} 
\right.$

$\sigma_{\sigma_m(n)}\sigma_{\gamma_n(m)}(k)=\left\{
\begin{array}{lll}

\sigma_n\sigma_{m+1}(k)=k           && k<n<m\\

\sigma_{n-1}\sigma_{m}(k)=k           && k<m<n, \,n \neq m+1\\

\sigma_{n-1}\sigma_{m}(k)=k-2    &&  m<n<k, \,k\neq n+1,m+2\\

\sigma_{n-1}\sigma_{m}(k)=k-1    &&  m<k<n, \,n,k\neq m+1\\

\sigma_{n}\sigma_{m+1}(k)=k-1    &&  n<k<m, \,k\neq n+1\\

\sigma_{n}\sigma_{m+1}(k)=k-2    &&  n<m<k, \,k> m+2\\

\sigma_{m}(k)=\text{not defined}    &&  n<m<k, \,k= m+1\\
\end{array} 
\right.$

$\sigma_{\gamma_{\sigma_n(k)}(m)}(\gamma_k(n))=\left\{
\begin{array}{lll}

\sigma_{m+1}\gamma_{k}(n)=n+1           && k<n<m\\

\sigma_{m+1}\gamma_{k}(n)=n           && k<m<n, \,n \neq m+1\\

\sigma_{m}\gamma_{k}(n)=n-1   &&  m<n<k, \,k\neq n+1,m+2\\

\sigma_{m}\gamma_{k}(n)=n   &&  m<k<n, \,n,k\neq m+1\\

\sigma_{m+1}\gamma_{k}(n)=n    &&  n<k<m, \,k\neq n+1\\

\sigma_{m}\gamma_{k}(n)=n   &&  n<m<k, \,k> m+2\\

\sigma_{m}\gamma_{k}(n)=n  &&  n<m<k, \,k= m+1\\
\end{array} 
\right.$\\

$\gamma_{\sigma_{\gamma_n(m)}(k)}(\sigma_m(n))=\left\{
\begin{array}{lll}

\gamma_{k}\sigma_{m}(n)=n+1         && k<n<m\\

\gamma_{k}\sigma_{m}(n)=n       && k<m<n, \,n \neq m+1\\

\gamma_{k-1}\sigma_{m}(n)=n-1    &&  m<n<k, \,k\neq n+1,m+2\\

\gamma_{k-1}\sigma_{m}(n)=n    &&  m<k<n, \,n,k\neq m+1\\

\gamma_{k}\sigma_{m}(n)=n   &&  n<k<m, \,k\neq n+1\\

\gamma_{k-1}\sigma_{m}(n)=n  &&  n<m<k, \,k> m+2\\

\gamma_{k-1}\sigma_{m}(n)=n  &&  n<m<k, \,k= m+1\\
\end{array} 
\right.$

The third relation is proved in the same way.
\end{proof}
\begin{proof}[Proof of Theorem \ref{theo-cycle-F}]
	For every $x_n,x_k \in X$, let 
	\[x_n \star x_k=x_{\sigma_n^{-1}(k)}\] with   $\sigma^{-1}_n: X \rightarrow X \setminus\{x_{n+1}\}$ the map defined by:
	
	\begin{equation}\label{defn-fcycle-set-F}
	\sigma^{-1}_n(k) = \left\{
	\begin{array}{lll}
	
	k              &&  k \leq n\\
	k+1    &&  k\geq n+1\\
	\end{array} \right. 
	\end{equation}

	Le $ x_m, x_n,x_k\in X$.  We show $(x_n \star x_k)\star (x_n \star x_m)=(x_k\star x_n)\star (x_k\star x_m)$. Instead of $x_m$, $x_n$, $x_k$,  we write $m,n,k$. So, we need to show 
	$\sigma^{-1}_{\sigma^{-1}_n(k)}  \sigma^{-1}_n(m) = \sigma^{-1}_{\sigma^{-1}_k(n)}  \sigma^{-1}_k(m)$. \\
	
	$\sigma^{-1}_{\sigma^{-1}_n(k)}  \sigma^{-1}_n(m)=\left\{
	\begin{array}{lll}
	
	\sigma^{-1}_{k}(m)=m       && k\leq n, \,m\leq n,\,m\leq k\\
	
	\sigma^{-1}_{k}(m)=m+1       && k\leq n, \,m\leq n,\,m> k\\

	\sigma^{-1}_{k}(m+1)=m+2       && k\leq n, \,m> n\\
	
	\sigma^{-1}_{k+1}(m)=m      && k>n, \,m\leq n\\

	\sigma^{-1}_{k+1}(m+1)=m+1      &&k>n, \,m> n, \,m\leq k\\

	\sigma^{-1}_{k+1}(m+1)=m+2      && k>n, \,m> n, \,m>k\\

\end{array} 
\right.$

$\sigma^{-1}_{\sigma^{-1}_k(n)}  \sigma^{-1}_k(m)=\left\{
\begin{array}{lll}

\left\{ \begin{array}{ccccccccccc}
	\sigma^{-1}_{n}(m)=m       &&&&&& k= n, \,m\leq n,\,m\leq k\\
	\sigma^{-1}_{n+1}(m)=m  &&&&&& k<n, \,m\leq n,\,m\leq k \\
\end{array}\right.\\

\sigma^{-1}_{n}(m+1)=m+1       && k\leq n, \,m\leq n,\,m> k\\

\left\{\begin{array}{llllll}
	\sigma^{-1}_{n}(m+1)=m+2       &&&& k= n, \,m> n\\
	\sigma^{-1}_{n+1}(m+1)=m+2      &&&& k<n, \,m> n \\
\end{array}\right.\\

\sigma^{-1}_{n}(m)=m                && k>n, \,m\leq n\\

\sigma^{-1}_{n}(m)=m+1      &&k>n, \,m> n, \,m\leq k\\

\sigma^{-1}_{n}(m+1)=m+2      && k>n, \,m> n, \,m>k\\
\end{array} 
\right.$

From its definition, the map $\sigma^{-1}_n: X \rightarrow X \setminus\{x_{n+1}\}$ is bijective for every $x_n \in X$  (a partial bijection of $X$). So, $(X,\star)$ is a cycle set.
The cycle set   $(X,\star)$ is square-free, since  $x_n \star x_n=x_n$, for all $x_n \in X$, indeed $x_n \star x_n=x_{\sigma^{-1}_n(n)}=x_n$,  and it is non-degenerate since  the map $x_n \mapsto x_n \star x_n=x_n$ is bijective,  for all $x_n \in X$.
\end{proof}
\subsection{Appendix 2: Case by case proof in  Theorem \ref{thm-2-precise} (for Theorem $2$)}

Let 
 $(\mathfrak{l}=\mathfrak{r})$  in $\mathcal{R}_{\rho}$.  Let   $x^{\epsilon}_{k}\in X\cup X^*$.  We show that $(\mathfrak{l'}=\mathfrak{r'})$  belongs to  $\mathcal{R}_{\rho}$ and $x_{r'}=x_r$. 
	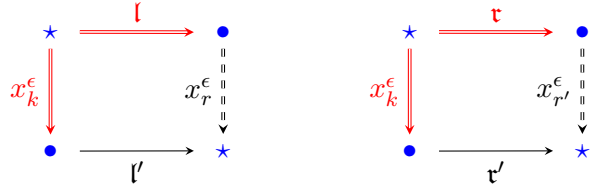
\begin{figure}[H]\label{fig-appendix}

		\begin{tikzpicture}
		\matrix (m) [matrix of math nodes,row sep=3em,column sep=4em,minimum width=2em,color=blue, ampersand replacement=\&]
		{ \star \& \bullet 	\\
		\bullet \&\star	\\};
		\path[-stealth]
		(m-1-1) edge [double,red] node [above] {$\mathfrak{l}$} (m-1-2)
		(m-1-1) edge [double,red]node [left] {$x^{\epsilon}_{k}$} (m-2-1)
			(m-1-2) edge [double, dashed] node [left] {$x^{\epsilon}_{r}$} (m-2-2)
		(m-2-1) edge   node [below] {$\mathfrak{l'}$} (m-2-2)
		;
		\end{tikzpicture} 
		\hspace{30pt}
			\begin{tikzpicture}
		\matrix (m) [matrix of math nodes,row sep=3em,column sep=4em,minimum width=2em,color=blue, ampersand replacement=\&]
		{ \star \& \bullet 	\\
			\bullet \&\star	\\};
		\path[-stealth]
		(m-1-1) edge [double,red] node [above] {$\mathfrak{r}$} (m-1-2)
		(m-1-1) edge [double,red]node [left] {$x^{\epsilon}_{k}$} (m-2-1)
		(m-1-2) edge [double, dashed] node [left] {$x^{\epsilon}_{r'}$} (m-2-2)
		(m-2-1) edge   node [below] {$\mathfrak{r'}$} (m-2-2)
		;
		\end{tikzpicture} 
		\caption{Computation of  $\mathfrak{l} \oplus x^{\epsilon}_{k}$ at left and $\mathfrak{r} \oplus x^{\epsilon}_{k}$ at right.}
	\end{figure}
There are six cases to check, from the form of the defining relations in $\mathcal{R}_{\rho}$ and $x^{\epsilon}_{k} \in \{x_{k}, x^{*}_{k}\}$. For the proof, we assume all  the diagrams can be completed, which  means that the relations needed to close squares  belong to $\mathcal{R}_\rho$ and that the  partial bijections  $\sigma_i, \gamma_i$ are   computed in their domain of definition. 

\textbf{Case 1:} \boldm{$\{\mathfrak{l}=x_ix_{\sigma_i^{-1}(j)}\;\;; \;\;\;  \mathfrak{r}=x_jx_{\sigma_j^{-1}(i)}\;\;;\;\;\; x_k\}$}:\\ 
\begin{figure}[h]\label{fig-case1}
	\scalebox{0.8}[0.9]{
	\begin{tikzpicture}
	\matrix (m) [matrix of math nodes,row sep=3em,column sep=4em,minimum width=2em,color=blue, ampersand replacement=\&]
	{ \star \& \bullet  \& \bullet	\\
		\bullet \& \bullet \& \star	\\};
	\path[-stealth]
	(m-1-1) edge [double,red] node [above] {$x_i$} (m-1-2)
		(m-1-2) edge [double,red] node [above] {$x_{\sigma_i^{-1}(j)}$} (m-1-3)
	(m-1-1) edge [double,red]node [left] {$x_{k}$} (m-2-1)
	(m-1-2) edge[dotted] node [left] {$x_{q}$} (m-2-2)
	(m-2-1) edge[] node[below]{$x_{p}$} (m-2-2)
	
		(m-1-3) edge [double, dashed]node [right] {$x_{\sigma^{-1}_{\sigma_i^{-1}(j)}\sigma_i^{-1}(k)}$} (m-2-3)
			(m-2-2) edge node [below] {$x_{s}$} (m-2-3)
	;
	\end{tikzpicture} }
	\hspace{30pt}
	\scalebox{0.8}[0.9]{	\begin{tikzpicture}
	\matrix (m) [matrix of math nodes,row sep=3em,column sep=4em,minimum width=2em,color=blue, ampersand replacement=\&]
	{ \star \& \bullet  \& \bullet	\\
		\bullet \& \bullet \& \star	\\};
	\path[-stealth]
(m-1-1) edge [double,red] node [above] {$x_j$} (m-1-2)
(m-1-2) edge [double,red] node [above] {$x_{\sigma_j^{-1}(i)}$} (m-1-3)
	(m-1-1) edge [double,red]node [left] {$x_{k}$} (m-2-1)
	(m-1-2) edge[dotted] node [left] {$x_{n}$} (m-2-2)
	(m-2-1) edge  node [below] {$x_{m}$} (m-2-2)
	
		(m-1-3) edge [double, dashed]node [right] {$x_{\sigma^{-1}_{\sigma_j^{-1}(i)}\sigma_j^{-1}(k)}$} (m-2-3)
		(m-2-2) edge node [below] {$x_{l}$} (m-2-3)
	;
	\end{tikzpicture} }
	\caption{Case 1:  $p=\sigma^{-1}_k(i)$, $q=\sigma^{-1}_i(k)$,   $m=\sigma^{-1}_k(j)$, $n=\sigma^{-1}_j(k)$,   {\boldm $l=\sigma^{-1}_{\sigma_j^{-1}(k)}\sigma_j^{-1}(i)$},  {\boldm$s=\sigma^{-1}_{\sigma_i^{-1}(k)}\sigma_i^{-1}(j)$}}
\end{figure}
From Lemma \ref{lem-defn-R}, the relations have the form $x_ix_{\sigma_i^{-1}(j)}=x_jx_{\sigma_j^{-1}(i)}$, so 
from Equation \ref{eqn-sigma-par}, $\sigma_i\sigma_{\sigma_i^{-1}(j)}(k)=\sigma_j\sigma_{\sigma_j^{-1}(i)}(k)$, for every $k$ in their domain of definition. This implies that  $r=\sigma^{-1}_{\sigma_i^{-1}(j)}\sigma_i^{-1}(k)$ and $r'=\sigma^{-1}_{\sigma_j^{-1}(i)}\sigma_j^{-1}(k)$ are equal.\\
 We show  that $\mathfrak{l'}=x_{p}\,x_{s}$ and $\mathfrak{r'}=x_{m}\,x_{l}$ are equal and belong to  $\mathcal{R}$. For that, we  show that  
 	$s=\sigma^{-1}_p(m)=\sigma^{-1}_{\sigma^{-1}_k(i)}\sigma^{-1}_k(j)$ and 	$l=\sigma^{-1}_m(p)=\sigma^{-1}_{\sigma^{-1}_k(j)}\sigma^{-1}_k(i)$.  This is true from  Equation \ref{eqn-sigma-par}, since 
$(x_ix_{\sigma^{-1}_i(k)}=x_kx_{\sigma^{-1}_k(i)})$ and 
$(x_jx_{\sigma^{-1}_j(k)}=x_kx_{\sigma^{-1}_k(j)})$  belong to $\mathcal{R}$.\\
\textbf{Case 2:} \boldm{$\{\mathfrak{l}=x_{\sigma_i^{-1}(j)}x^*_{\sigma_j^{-1}(i)}\;\;; \;\;\;  \mathfrak{r}=x_i^*x_j\;\;;\;\;\; x_k\}$}:\\ 
\begin{figure}[h]\label{fig-case2}
	\scalebox{0.8}[0.9]{
		\begin{tikzpicture}
		\matrix (m) [matrix of math nodes,row sep=3em,column sep=4em,minimum width=2em,color=blue, ampersand replacement=\&]
		{ \star \& \bullet  \& \bullet	\\
			\bullet \& \bullet \& \star	\\};
		\path[-stealth]
		(m-1-1) edge [double,red] node [above] {$x_{\sigma_i^{-1}(j)}$} (m-1-2)
		(m-1-2) edge [double,red] node [above] {$x^*_{\sigma_j^{-1}(i)}$} (m-1-3)
		(m-1-1) edge [double,red]node [left] {$x_{k}$} (m-2-1)
		(m-1-2) edge[dotted] node [left] {$x_{q}$} (m-2-2)
		(m-2-1) edge[] node[below]{$x_{p}$} (m-2-2)
		
		(m-1-3) edge [double, dashed]node [right] {$x_{\sigma_{\sigma_j^{-1}(i)}\sigma^{-1}_{\sigma_i^{-1}(j)}(k)}$} (m-2-3)
		(m-2-2) edge node [below] {$x^*_{s}$} (m-2-3)
		;
		\end{tikzpicture} }
	\hspace{30pt}
	\scalebox{0.8}[0.9]{	\begin{tikzpicture}
		\matrix (m) [matrix of math nodes,row sep=3em,column sep=4em,minimum width=2em,color=blue, ampersand replacement=\&]
		{ \star \& \bullet  \& \bullet	\\
			\bullet \& \bullet \& \star	\\};
		\path[-stealth]
		(m-1-1) edge [double,red] node [above] {$x_i^*$} (m-1-2)
		(m-1-2) edge [double,red] node [above] {$x_j$} (m-1-3)
		(m-1-1) edge [double,red]node [left] {$x_{k}$} (m-2-1)
		(m-1-2) edge[dotted] node [left] {$x_{n}$} (m-2-2)
		(m-2-1) edge  node [below] {$x^*_{m}$} (m-2-2)
		
		(m-1-3) edge [double, dashed]node [right] {$x_{\sigma^{-1}_j\sigma_i(k)}$} (m-2-3)
		(m-2-2) edge node [below] {$x_{l}$} (m-2-3)
		;
		\end{tikzpicture} }
	\caption{Case 2:  $q=\sigma^{-1}_{\sigma^{-1}_i(j)}(k)$, $n=\sigma_i(k)$, $m=\sigma^{-1}_n(i)$, 
	\boldm{$p=\sigma^{-1}_k\sigma^{-1}_i(j)$}, 	\boldm{$s=\sigma^{-1}_r\sigma^{-1}_j(i)$}
	}
\end{figure}
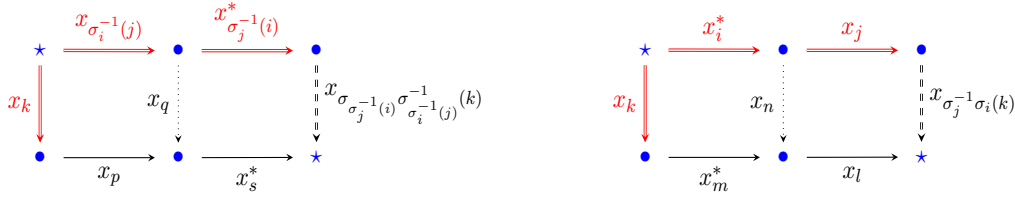
From Equation \ref{eqn-sigma-par},  $r=\sigma_{\sigma_j^{-1}(i)}\sigma^{-1}_{\sigma_i^{-1}(j)}(k)$ and $r'=\sigma^{-1}_j\sigma_i(k)$ are equal. We show that 
$\mathfrak{l'}\,=\,x_{p}x_s^*$ and $\mathfrak{r'}\,=\,x^*_{m}x_{l}$ are equal and belong to  $\mathcal{R}_{\rho}$, which is equivalent to 
$x_{m}\,x_{p}\,=\, x_{l}x_s$ belongs to $\mathcal{R}$. For that, we show  $p=\sigma^{-1}_{m}(l)=\sigma^{-1}_m\sigma^{-1}_n(j)$ and $s=\sigma^{-1}_{l}(m)=\sigma^{-1}_l\sigma^{-1}_n(i)$.  This is true, from Equation \ref{eqn-sigma-par}, since $(x_ix_k=x_nx_m)$ and $(x_jx_r=x_nx_l)$ in
$\mathcal{R}$.\\

\textbf{Case 3:} \boldm{$\{\mathfrak{l}=x^*_{\sigma_i^{-1}(j)}x_i^*\;\;; \;\;\;  \mathfrak{r}=x^*_{\sigma_j^{-1}(i)}x_j^*\;\;;\;\;\; x_k\}$}:\\ 
\begin{figure}[h]\label{fig-case3}
	\scalebox{0.8}[0.9]{
		\begin{tikzpicture}
		\matrix (m) [matrix of math nodes,row sep=3em,column sep=4em,minimum width=2em,color=blue, ampersand replacement=\&]
		{ \star \& \bullet  \& \bullet	\\
			\bullet \& \bullet \& \star	\\};
		\path[-stealth]
		(m-1-1) edge [double,red] node [above] {$x^*_{\sigma_i^{-1}(j)}$} (m-1-2)
		(m-1-2) edge [double,red] node [above] {$x_i^*$} (m-1-3)
		(m-1-1) edge [double,red]node [left] {$x_{k}$} (m-2-1)
		(m-1-2) edge[dotted] node [left] {$x_q$} (m-2-2)
		(m-2-1) edge[] node[below]{$x^*_{p}$} (m-2-2)
		
		(m-1-3) edge [double, dashed]node [right] {$x_{\sigma_i\sigma_{\sigma_i^{-1}(j)}(k)}$} (m-2-3)
		(m-2-2) edge node [below] {$x^*_s=x^*_{\sigma^{-1}_r(i)}$} (m-2-3)
		;
		\end{tikzpicture} }
	\hspace{30pt}
	\scalebox{0.8}[0.9]{	\begin{tikzpicture}
		\matrix (m) [matrix of math nodes,row sep=3em,column sep=4em,minimum width=2em,color=blue, ampersand replacement=\&]
		{ \star \& \bullet  \& \bullet	\\
			\bullet \& \bullet \& \star	\\};
		\path[-stealth]
		(m-1-1) edge [double,red] node [above] {$x^*_{\sigma_j^{-1}(i)}$} (m-1-2)
		(m-1-2) edge [double,red] node [above] {$x^*_j$} (m-1-3)
		(m-1-1) edge [double,red]node [left] {$x_{k}$} (m-2-1)
		(m-1-2) edge[dotted] node [left] {$x_n$} (m-2-2)
		(m-2-1) edge  node [below] {$x^*_{m}$} (m-2-2)
		
		(m-1-3) edge [double, dashed]node [right] {$x_{\sigma_j\sigma_{\sigma^{-1}_j(i)}(k)}$} (m-2-3)
		(m-2-2) edge node [below] {$x^*_l=x^*_{\sigma^{-1}_{r'}(j)}$} (m-2-3)
		;
		\end{tikzpicture} }
	\caption{Case 3:  $s=\sigma^{-1}_r(i)$, $l=\sigma^{-1}_r(j)$,
	\boldm{$p=\sigma^{-1}_q\sigma^{-1}_i(j)$}, 
		\boldm{$m=\sigma^{-1}_n\sigma^{-1}_j(i)$} 
	}
\end{figure}
From 
Equation \ref{eqn-sigma-par},  $r=\sigma_i\sigma_{\sigma_i^{-1}(j)}(k)$ and $r'=\sigma_j\sigma_{\sigma_j^{-1}(i)}(k)$ are equal.
We show that 
$\mathfrak{l'}\,=\,x_p^*x^*_{s}$ and $\mathfrak{r'}\,=\,x^*_m x^*_{l}\,=\,x^*_m x^*_{\sigma^{-1}_{r}(j)}$ are equal and belong to  $\mathcal{R}_{\rho}$, which is equivalent to 
$(x_{s}x_p\,=\, x_{l}x_m)$ in  $\mathcal{R}$.
We show that  $p=\sigma^{-1}_s\sigma^{-1}_r(j)$ and 
$m=\sigma^{-1}_{l}\sigma^{-1}_{r}(i)$. Since $(x_ix_q=x_rx_s)$ and $(x_jx_n=x_rx_l)$ in  $\mathcal{R}$, $p=\sigma^{-1}_q\sigma^{-1}_i(j)=\sigma^{-1}_s\sigma^{-1}_r(j)$ and $m=\sigma^{-1}_n\sigma^{-1}_j(i)=\sigma^{-1}_{l}\sigma^{-1}_{r}(i)$, from Equation \ref{eqn-sigma-par}.\\

\textbf{Case 4:} \boldm{$\{\mathfrak{l}=x_ix_{\sigma_i^{-1}(j)}\;\;; \;\;\;  \mathfrak{r}=x_jx_{\sigma_j^{-1}(i)}\;\;;\;\;\; x_k^*\}$}:\\ 
\begin{figure}[h]\label{fig-case4}
	\scalebox{0.8}[0.9]{
		\begin{tikzpicture}
		\matrix (m) [matrix of math nodes,row sep=3em,column sep=4em,minimum width=2em,color=blue, ampersand replacement=\&]
		{ \star \& \bullet  \& \bullet	\\
			\bullet \& \bullet \& \star	\\};
		\path[-stealth]
		(m-1-1) edge [double,red] node [above] {$x_i$} (m-1-2)
		(m-1-2) edge [double,red] node [above] {$x_{\sigma_i^{-1}(j)}$} (m-1-3)
		(m-1-1) edge [double,red]node [left] {$x^*_{k}$} (m-2-1)
		(m-1-2) edge[dotted] node [left] {$x^*_{q}$} (m-2-2)
		(m-2-1) edge[] node[below]{$x_{p}$} (m-2-2)
		
		(m-1-3) edge [double, dashed]node [right] {$x^*_{r}$} (m-2-3)
		(m-2-2) edge node [below] {$x_{s}$} (m-2-3)
		;
		\end{tikzpicture} }
	\hspace{30pt}
	\scalebox{0.8}[0.9]{	\begin{tikzpicture}
		\matrix (m) [matrix of math nodes,row sep=3em,column sep=4em,minimum width=2em,color=blue, ampersand replacement=\&]
		{ \star \& \bullet  \& \bullet	\\
			\bullet \& \bullet \& \star	\\};
		\path[-stealth]
		(m-1-1) edge [double,red] node [above] {$x_j$} (m-1-2)
		(m-1-2) edge [double,red] node [above] {$x_{\sigma_j^{-1}(i)}$} (m-1-3)
		(m-1-1) edge [double,red]node [left] {$x^*_{k}$} (m-2-1)
		(m-1-2) edge[dotted] node [left] {$x^*_{n}$} (m-2-2)
		(m-2-1) edge  node [below] {$x_{m}$} (m-2-2)
		
		(m-1-3) edge [double, dashed]node [right] {$x^*_{r'}$} (m-2-3)
		(m-2-2) edge node [below] {$x_{l}$} (m-2-3)
		;
		\end{tikzpicture} }
	\caption{Case 4:  $p=\sigma_k(i)$, $q=\sigma^{-1}_p(k)$,   $n=\sigma^{-1}_m(k)$, \boldm{$l=\sigma_n\sigma^{-1}_j(i)$}, \boldm{$s=\sigma_q\sigma^{-1}_i(j)$},  \boldm{$r=\sigma^{-1}_s\sigma^{-1}_p(k)\;,\; r'=\sigma^{-1}_l\sigma^{-1}_m(k)$ }}
\end{figure}
We show  that $\mathfrak{l'}=x_{p}\,x_{s}$ and $\mathfrak{r'}=x_{m}\,x_{l}$ are equal and belong to  $\mathcal{R}$. For that, we  show that  
$s=\sigma^{-1}_p(m)=\sigma^{-1}_{p}\sigma_k(j)$ and 	$l=\sigma^{-1}_m(p)=\sigma^{-1}_{m}\sigma_k(i)$.  This is true from  Equation \ref{eqn-sigma-par}, since 
$(x_kx_i=x_px_q)$ and $(x_kx_j=x_mx_n)$ in  $\mathcal{R}$. So, $x_{p}\,x_{s}=x_{m}\,x_{l}$ and this implies 
$r=\sigma^{-1}_s\sigma^{-1}_p(k)$, $r'=\sigma^{-1}_l\sigma^{-1}_m(k)$  are equal.\\

\textbf{Case 5:} \boldm{$\{\mathfrak{l}=x_{i}x^*_j\;\;; \;\;\;  \mathfrak{r}=x^*_{\gamma^{-1}_i(j)}x_{\gamma^{-1}_j(i)}\;\;;\;\;\; x_k^*\}$}:\\ 
\begin{figure}[h]\label{fig-case5}
	\scalebox{0.8}[0.9]{
		\begin{tikzpicture}
		\matrix (m) [matrix of math nodes,row sep=3em,column sep=4em,minimum width=2em,color=blue, ampersand replacement=\&]
		{ \star \& \bullet  \& \bullet	\\
			\bullet \& \bullet \& \star	\\};
		\path[-stealth]
		(m-1-1) edge [double,red] node [above] {$x_{i}$} (m-1-2)
		(m-1-2) edge [double,red] node [above] {$x^*_{j}$} (m-1-3)
		(m-1-1) edge [double,red]node [left] {$x^*_{k}$} (m-2-1)
		(m-1-2) edge[dotted] node [left] {$x^*_{q}$} (m-2-2)
		(m-2-1) edge[] node[below]{$x_{p}$} (m-2-2)
		
		(m-1-3) edge [double, dashed]node [right] {$x_r^*$} (m-2-3)
		(m-2-2) edge node [below] {$x^*_{s}$} (m-2-3)
		;
		\end{tikzpicture} }
	\hspace{30pt}
	\scalebox{0.8}[0.9]{	\begin{tikzpicture}
		\matrix (m) [matrix of math nodes,row sep=3em,column sep=4em,minimum width=2em,color=blue, ampersand replacement=\&]
		{ \star \& \bullet  \& \bullet	\\
			\bullet \& \bullet \& \star	\\};
		\path[-stealth]
		(m-1-1) edge [double,red] node [above] {$x^*_{\gamma^{-1}_i(j)}$} (m-1-2)
		(m-1-2) edge [double,red] node [above] {$x_{\gamma^{-1}_j(i)}$} (m-1-3)
		(m-1-1) edge [double,red]node [left] {$x^*_{k}$} (m-2-1)
		(m-1-2) edge[dotted] node [left] {$x^*_{n}$} (m-2-2)
		(m-2-1) edge  node [below] {$x^*_{m}$} (m-2-2)
		
		(m-1-3) edge [double, dashed]node [right] {$x_{r'}^*$} (m-2-3)
		(m-2-2) edge node [below] {$x_{l}$} (m-2-3)
		;
		\end{tikzpicture} }
	\caption{Case 5: 
		$q=\gamma_i(k)$, $p=. \gamma^{-1}_q(i)$, 
		 $n=\gamma^{-1}_{\gamma^{-1}_i(j)}(k)$,
		 $s=\gamma_q^{-1}(j)$,
			\boldm{ $l=\gamma^{-1}_{r'}\gamma^{-1}_j(i)$},
			\boldm{$m=\gamma^{-1}_k\gamma^{-1}_i(j)$},
			\boldm{$r=\gamma^{-1}_j(q) \;,\;
	r'=\gamma_{\gamma^{-1}_j(i)}(n)$}}
\end{figure}
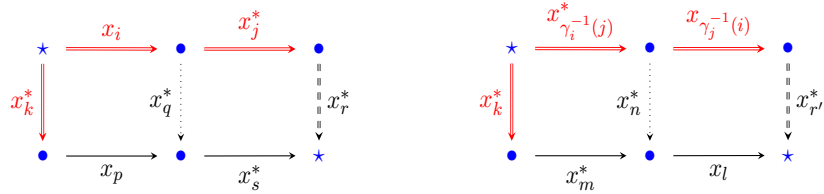
Since  $r=\gamma^{-1}_j(q)=\gamma^{-1}_j\gamma_i(k)$ and 
$r'=\gamma_{\gamma^{-1}_j(i)}(n)=\gamma_{\gamma^{-1}_j(i)}\gamma^{-1}_{\gamma^{-1}_i(j)}(k)$, we have $r=r'$ from Equation \ref{eqn-gamma-par}. We show that 
$\mathfrak{l'}\,=\,x_{p}x_s^*$ and $\mathfrak{r'}\,=\,x^*_{m}x_{l}$ are equal and belong to  $\mathcal{R}_{\rho}$, which is equivalent to 
$(x_{m}\,x_{p}\,=\, x_{l}x_s)$ in  $\mathcal{R}$.
For that, we  show that $l=\gamma^{-1}_{s}(p)=\gamma^{-1}_{s}\gamma^{-1}_{q}(i)$ and $m=\gamma^{-1}_{p}(s)=\gamma^{-1}_{p}\gamma^{-1}_{q}(j)$.  This is true, from Equation \ref{eqn-sigma-par}, since $(x_kx_i=x_px_q)$ and $(x_{r}x_j=x_sx_q)$ in 
$\mathcal{R}$.\\

\textbf{Case 6:} \boldm{$\{\mathfrak{l}=x^*_ix^*_{\gamma_i^{-1}(j)}\;\;; \;\;\;  \mathfrak{r}=x_j^*x^*_{\gamma_j^{-1}(i)}\;\;;\;\;\; x_k^*\}$}:\\ 
\begin{figure}[h]\label{fig-case3}
	\scalebox{0.8}[0.9]{
		\begin{tikzpicture}
		\matrix (m) [matrix of math nodes,row sep=3em,column sep=4em,minimum width=2em,color=blue, ampersand replacement=\&]
		{ \star \& \bullet  \& \bullet	\\
			\bullet \& \bullet \& \star	\\};
		\path[-stealth]
		(m-1-1) edge [double,red] node [above] {$x^*_i$} (m-1-2)
		(m-1-2) edge [double,red] node [above] {$x^*_{\gamma_i^{-1}(j)}$} (m-1-3)
		(m-1-1) edge [double,red]node [left] {$x^*_{k}$} (m-2-1)
		(m-1-2) edge[dotted] node [left] {$x_q^*$} (m-2-2)
		(m-2-1) edge[] node[below]{$x^*_{p}$} (m-2-2)
		
		(m-1-3) edge [double, dashed]node [right] {$x^*_{\gamma^{-1}_{\gamma_i^{-1}(j)}\gamma^{-1}_{i}(k)}$} (m-2-3)
		(m-2-2) edge node [below] {$x^*_s$} (m-2-3)
		;
		\end{tikzpicture} }
	\hspace{30pt}
	\scalebox{0.8}[0.9]{	\begin{tikzpicture}
		\matrix (m) [matrix of math nodes,row sep=3em,column sep=4em,minimum width=2em,color=blue, ampersand replacement=\&]
		{ \star \& \bullet  \& \bullet	\\
			\bullet \& \bullet \& \star	\\};
		\path[-stealth]
		(m-1-1) edge [double,red] node [above] {$x_j^*$} (m-1-2)
		(m-1-2) edge [double,red] node [above] {$x^*_{\gamma_j^{-1}(i)}$} (m-1-3)
		(m-1-1) edge [double,red]node [left] {$x^*_{k}$} (m-2-1)
		(m-1-2) edge[dotted] node [left] {$x^*_n$} (m-2-2)
		(m-2-1) edge  node [below] {$x^*_{m}$} (m-2-2)
		
		(m-1-3) edge [double, dashed]node [right] {$x^*_{\gamma^{-1}_{\gamma_j^{-1}(i)}\gamma_j^{-1}(k)}$} (m-2-3)
		(m-2-2) edge node [below] {$x^*_{l}$} (m-2-3)
		;
		\end{tikzpicture} }
	\caption{Case 6:  $p=\gamma^{-1}_{k}(i)$, $q=\gamma^{-1}_{i}(k)$,  $m=\gamma_k^{-1}(j)$, 
	$n=\gamma_j^{-1}(k)$, \boldm{$s=\gamma^{-1}_q\gamma_i^{-1}(j)$},
\boldm{$l=\gamma^{-1}_{n}\gamma_j^{-1}(i)$}}
\end{figure}
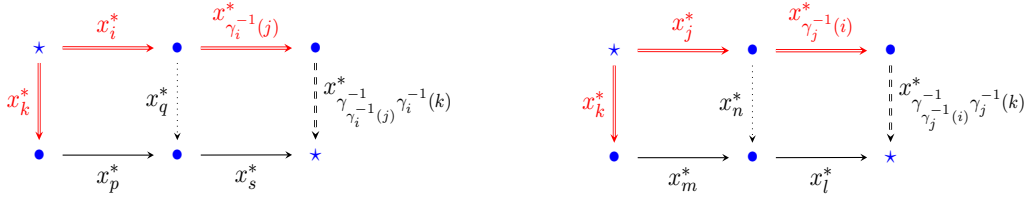
From Lemma \ref{lem-defn-R}, the relations have the form $x_{\gamma_i^{-1}(j)}x_i=x_{\gamma_j^{-1}(i)}x_j$, so 
from Equation \ref{eqn-gamma-par},
$\gamma_i\gamma_{\gamma_i^{-1}(j)}(k)=\gamma_j\gamma_{\gamma_j^{-1}(i)}(k)$, for every $k$ in their domain of definition. This implies $r=\gamma^{-1}_{\gamma_i^{-1}(j)}\gamma^{-1}_{i}(k)$ and  $r'=\gamma^{-1}_{\gamma_j^{-1}(i)}\gamma_j^{-1}(k)$ are equal. We show  $\mathfrak{l'}=x^*_p \,x^*_s$
and $\mathfrak{r'}=x^*_mx^*_l$  are equal  and belong to $\mathcal{R}_\rho$, that is  $(x_sx_p=x_lx_m)$ in $\mathcal{R}$. For that, we show $s=\gamma^{-1}_{p}(m)=\gamma^{-1}_p\gamma^{-1}_{k}(j)$ and 
$l=\gamma^{-1}_m(p)=\gamma^{-1}_m\gamma^{-1}_k(i)$. This  is true from Equation \ref{eqn-gamma-par}, since 
$(x_qx_i=x_px_k)$ and $(x_nx_j=x_mx_k)$ in  $\mathcal{R}$.


\bigskip\bigskip\noindent
{ Fabienne Chouraqui}

\smallskip\noindent
University of Haifa at Oranim, Israel.

\smallskip\noindent
E-mail: {\tt fabienne.chouraqui@gmail.com} \\

                {\tt fchoura@sci.haifa.ac.il}

\begin{thebibliography}{100}
\bibitem{bachiller0}D. Bachiller, \emph{Study of the algebraic structure of left braces and the Yang-Baxter equation}, Ph.D. thesis 2016, Universitat Autonoma de Barcelona.
\bibitem{bachiller1} D. Bachiller, \emph{Counterexample to a conjecture about braces}, J. Algebra {\bf 453} (2016), 160-176.
\bibitem{bachiller2} D. Bachiller, \emph{Extensions, matched products and simple braces}, J. Pure Appl. Algebra {\bf 222} (2018) 1670-1691.

\bibitem{belk}J. Belk, \emph{Thompson's group $F$}, Ph.D. Thesis (Cornell University), 
ArXiv 0708.3609.
\bibitem{partial-s}V. Bergelson, N. Hindman, A. Blass, \emph{Partition Theorems for Spaces of Variable Words}, Proc. London Math. Soc. {\bf 68} (1994), n.3, 449-476.
\bibitem{brin}M.G. Brin, C.C.  Squier, \emph{Groups of piecewise linear homeomorphisms of the real line}, Invent. Math. {\bf 79}(1985), n. 3, 485-498.
\bibitem{brown-geo}K.S. Brown, R. Geoghegan, \emph{An infinite-dimesional torsion-free $FP_\infty$  group}, Invent. Math. {\bf 77} (1984), 367-381.

\bibitem{canon} J.W. Canon, W.J. Floyd, W.R. Parry,  \emph{Introductory notes to Richard Thompson's groups}, L'enseignement Mathematique {\bf 42}(1996), 215-256.
\bibitem{catino-cycle2} M. Castelli, F. Catino, G. Pinto, \emph{A new family of set-theoretic solutions of the
Yang-Baxter equation}, Comm. Algebra {\bf 46} (2018), 1622-1629.
\bibitem{catino1}F. Catino, I. Colazzo, P. Stefanelli, \emph{Skew left braces with non-trivial annihilator}, to appear J.Algebra Appl. 
\bibitem{catino3} F. Catino, I. Colazzo, P. Stefanelli, \emph{Semi-braces and the Yang-Baxter equation}, J.
Algebra {\bf 483} (2017), 163-187.
\bibitem{catino4} F. Catino, M. Mazzotta, P. Stefanelli, \emph{Inverse Semi-braces and the Yang-Baxter equation}, ArXiv 2007.05730.
\bibitem{catino-cycle} F. Catino, M.M. Miccoli, \emph{Construction of quasi-linear left cycle sets}, J. Algebra Appl. {\bf 14}, no.1, (2015).

\bibitem{cedo} F. Cedo, E. Jespers,  A. del Rio, \emph{Involutive Yang-Baxter Groups}, Trans.
Amer.  Math.  Soc.  {\bf 362} (2010), 2541-2558.
\bibitem{jespers-adv} F. Cedo, E. Jespers, J. Okninski, \emph{Retractability of set theoretic solutions
	of the Yang-Baxter equation},  Advances in Mathematics  {\bf 224} (2010), 2472-2484.
\bibitem{brace} F. Cedo, E. Jespers, J. Okninski, \emph{Braces and the Yang-Baxter equation}, Comm. Math. Phys. {\bf 327} (2014), 101-116.
\bibitem{brace-ag}F. Cedo, T. Gateva-Ivanova, A. Smoktunowicz, \emph{On the Yang-Baxter equation and left nilpotent left braces}, J. Pure App. Algebra {\bf 221}(2017), n.4, 751-756.
\bibitem{chou_art} F.~Chouraqui, \emph{Garside groups and the Yang-Baxter equation}, Comm. in Algebra {\bf 38} (2010) 4441-4460.
\bibitem{chou_godel1} F.~Chouraqui and E.~Godelle, \emph{Folding of set theoretical solutions of the Yang-Baxter Equation}, Algebra and Representation Theory {\bf 15 } (2012) 1277-1290.
\bibitem{chou_godel2} F.~Chouraqui and E.~Godelle, \emph{Finite quotients of $I$-type groups}, Adv. Math. { \bf 258} (2014), 46-68.
\bibitem{chou_aut} F. Chouraqui, \emph{An algorithmic construction  of group automorphisms and the quantum Yang-Baxter equation}, Comm. in Algebra {\bf 46} (2018), n.11,  4710-4723.
\bibitem{chou_left_garside} F. Chouraqui, \emph{Left orders in Garside groups},  Int. J. of    
                  Alg.  and Comp. (2016), vol.26, n.7,  1349-1359.
                  
                  
                  \bibitem{clifford} A.H. Clifford, G.B. Preston, \emph{The Algebraic Theory of Semigroups},  Math. Surveys of the American Math. Soc. {\bf 7}, 1967.
                  
 \bibitem{deh_torsion}P. Dehornoy, \emph{Gaussian groups are torsion free}, J. of Algebra {\bf 210} (1998) 291-297.           
 \bibitem{DePa}  P.~Dehornoy and L.~Paris, \emph{Gaussian groups and Garside groups, two generalisations of Artin groups}, Proc. London Math Soc.  {\bf 79} (1999) 569-604.      
\bibitem{dehornoy}P. Dehornoy, \emph{Groupes de Garside}, Ann. Scient. Ec.. Norm. Sup. {\bf 35} (2002), 
267-306.
\bibitem{dehornoy2}P. Dehornoy, \emph{The subword reversing method}, Int. J. of Alg. and Comp.{\bf 21 } (2011),  n.1,    71-118.
\bibitem{deh-digne}P. Dehornoy, F. Digne and J. Michel, \emph{Garside families and Garside germs}, J. Algebra {\bf 380} (2013), 109-145. 
\bibitem{deh_coxeterlike} P.~Dehornoy, \emph{Coxeter-like groups for groups of set-theoretic solutions of the Yang-Baxter equation}, Comptes Rendus Mathematiques {\bf 351} (2013) 419-424.

\bibitem{garside1} P. Dehornoy, F. Digne, E. Godelle, D. Krammer, J. Michel, \emph{Foundations of Garside theory}, EMS Tracts in Mathematics (2015), volume 24.









\bibitem{drinf} V.G. Drinfeld, \emph{On some unsolved problems in quantum group theory}, Lec. Notes Math. {\bf 1510} (1992) 1-8.


\bibitem{etingof} P.~Etingof, T.~Schedler, A.~Soloviev, \emph{Set-theoretical solutions to the Quantum Yang-Baxter equation}, Duke Math. J. {\bf 100} (1999) 169-209.
\bibitem{etingof2} P. Etingof, R. Guralnick, A. Soloviev, \emph{Indecomposable set-theoretical solutions to the quantum Yang-Baxter equation on a set with a prime number of elements}, J. Algebra {\bf 249} (2001), 709-719.

\bibitem{gateva_conj}   T. Gateva-Ivanova, \emph{Regularity of skew-polynomials rings with binomial relations}, Talk at the International Algebra Conference, Miskolc, Hungary, 1996.
\bibitem{gateva_van} T.~Gateva-Ivanova and M.~Van~den~Bergh, \emph{Semigroups of $I$-type}, J.~Algebra { \bf 206 } (1998) 97-112.
\bibitem{gat_je-o}T. Gateva-Ivanova,  E. Jespers, J. Okninski,  \emph{Quadratic  algebras  of  skew
	type  and  the  underlying  semigroups},  J.  Algebra  {\bf 270}  (2003),  no. 2,  635-659.
\bibitem{gateva-phys1}  T. Gateva-Ivanova, \emph{A combinatorial approach to the set-theoretic solutions of the Yang–Baxter equation}, 
J. Math. Phys. {\bf 45} (2004),  3828-3858.

\bibitem{gateva}   T. Gateva-Ivanova, \emph{Garside Structures on Monoids with Quadratic Square-Free Relations}, Algebra and Representation Theory { \bf 14 } (2011) 779-802.

\bibitem{gateva_new}   T. Gateva-Ivanova. \emph{Set-theoretic solutions of the Yang–Baxter equation, braces and symmetric groups},  Adv. Math. {\bf 388 } (2018), n.7, 649-701.
\bibitem{gateva-phys2} T. Gateva-Ivanova and  P.J. Cameron, \emph{Multipermutation solutions of the Yang–Baxter equation},  Comm. in Math. Phys. {\bf 309}(3) (2012) 583-621. 


\bibitem{g-vendra1} L. Guarnieri, L. Vendramin, \emph{Skew braces and the Yang-Baxter equation}, Math. Comp. {\bf 86} (2017), 2519-2534.


\bibitem{jespers_i-type} E.~Jespers, J.~Okninski, \emph{Monoids and groups of I-type}, Algebra Rep. Theory {\bf 8}(2005), 709-729.

\bibitem{jespers_book} E.~Jespers, J.~Okninski, \emph{Noetherian Semigroup Algebras}, Algebra and applications, vol.7 (2007).
\bibitem{jespers-seul}E. Jespers, \emph{Groups and set-theoretic solutions of the Yang Baxter equation}, Note Mat. {\bf 30} (2010)  n. 1, 9-20.


 \bibitem{lawson} M.V. Lawson, \emph{Inverse semigroups:  The Theory of Partial Symmetries}, World Scientific, 1998.
 
 \bibitem{lebed2} V. Lebed, L. Vendramin, \emph{Homology of left non-degenerate set-theoretic solutions to the Yang-Baxter equation}, Advances Math.  {\bf 304} (2017), 1219-1261.
 
 \bibitem{lebed1} V. Lebed, L. Vendramin, \emph{ On Structure Groups of Set-Theoretic Solutions to the Yang-Baxter Equation}, Proc. Edinb. Math. Soc. {\bf 62 }(2019), n. 3, 683-717.
 
 
 
 \bibitem{mckenzie}R. MacKenzie, R.J. Thompson, \emph{An elementary construction of unsolvable word problems in group theory}, Word Problems:  (W.W. Boone, F.B. Canonito,  and R.C. Lyndon eds.), Studies in Logic and the foundations of Mathematics, vol.71, North-Holland, Amsterdam, 1973, p. 457-478.
 
 \bibitem{moore} G. Moore, N. Seiberg, \emph{Polynomial equations for rational conformal field theories}, Physics Letters {\bf 212 } (1988), n.4, 451-460.
\bibitem{preston}G.B. Preston, \emph{Inverse semi-groups}, J. of the London Math. Soc. {\bf 29}  (1954), n. 4,  396-403. 
 

\bibitem{rump}W. Rump, \emph{A decomposition theorem for square-free unitary solutions of the quantum Yang-Baxter equation}, Adv. Math. {\bf 193}(2005), 40-55.
\bibitem{rump-braces3} W. Rump, \emph{Modules over braces}, Algebra Dicrete Math. (2006), 127-137.
\bibitem{rump_braces} W. Rump,  \emph{Braces, radical rings, and the quantum Yang-Baxter equation}, J. Algebra {\bf 307}(2007), 153-170.
\bibitem{rump_braces2} W. Rump,  \emph{Classification of cyclic braces}, J. Pure Appl.Algebra {\bf 209}(2007), 671-685.



\bibitem{smot}A. Smoktunowicz, \emph{On Engel groups, nilpotent groups, rings, braces and the Yang-Baxter equation}, Trans. Amer. Math. Soc. {\bf 370} (2018), 6535-6564. 
\bibitem{s-vendra2}A. Smoktunowicz, L. Vendramin, \emph{On skew-braces},  J. Comb. Algebra {\bf 2 }(2018), no. 1, 47-86

\bibitem{stein} M. Stein, \emph{Groups of piecewise linear homeomorphisms}, Trans. Amer. Math. Soc {\bf 332} (1992), 477-514.
\bibitem{thompson}R.J. Thompson, \emph{Embeddings into finitely-generated simple groups which preserve the word problem}, Word Problems  II: The Oxford book (S.I. Adian, W.W. Boone and G. Higman eds.), Studies in Logic and the foundations of Mathematics, vol.95, North-Holland, Amsterdam, 1980, p. 401-441.
\bibitem{vendra-cycle} L. Vendramin, \emph{Extensions of set-theoretic solutions of the Yang-Baxter equation and a conjecture of Gateva-Ivanova},  
J. of Pure and Applied Alg. {\bf 220} (2016), n.5, 2064-2076.

\bibitem{wagner}V.V. Wagner, \emph{Generalised groups},  Proc. of the USSR Academy of Sciences (in Russian) {\bf 84}  (1952), 1119–1122 (there exists an English translation).


 \end{thebibliography}
\end{document}